\documentclass[leqno,12pt]{article}
\usepackage{graphicx}
\usepackage{fancyhdr}

\usepackage[T1]{fontenc}
\usepackage{mathtools}  
\usepackage{amsmath}
\usepackage{amsthm}
\usepackage{amsfonts}
\usepackage{amssymb} 
\usepackage[useregional]{datetime2}
\newtheorem{theorem}{Theorem}[section]
\newtheorem{lemma}[theorem]{Lemma}
\newtheorem{corollary}[theorem]{Corollary}
\theoremstyle{definition}
\newtheorem{definition}[theorem]{Definition}
\newtheorem{remark}[theorem]{Remark}

\def\address#1{{\center{#1}}}
\date{}

\newcommand{\eqal}[1]{\begin{aligned}#1\end{aligned}}
\newcommand{\R}{\mathbb{R}}

\newcommand{\N}{\mathbb{N}}

\DeclareMathOperator{\divv}{div}
\DeclareMathOperator{\rot}{rot}
\DeclareMathOperator{\supp}{supp}

\numberwithin{equation}{section}

\title{On regularity estimates for the axially symmetric Navier–Stokes Cauchy problem and the critical-wedge occurrence problem}
\author{Wies\l aw J. Grygierzec$^{(1)}$, Wojciech M. Zaj\c{a}czkowski$^{(2,3)}$}

\begin{document}
\maketitle
%%%%%%%%%%%%%%%%%%%%%%%%%%%%%%%%%%%%%%%%%%%%%%%%%%%%%%%%%%%%%%%%%%%%%%%%%
\thispagestyle{fancy}
%%%%%%%%%%%%%%%%%%%%%%%%%%%%%%%%%%%%%%%%%%%%%%%%%%%%%%%%%%%%%%%%%%%%%%%%%

\address{$^1$Department of Statistics and Social Policy,\\
  University of Agriculture in Krak\'ow, Al. Mickiewicza 21,\\
  31-120 Krak\'ow, Poland.\\
  e-mail: wieslaw.grygierzec@urk.edu.pl\\
  $^2$Institute of Mathematics, Polish Academy of Sciences,\\
  \'Sniadeckich 8, 00-656 Warsaw, Poland\\
  e-mail: wz@impan.pl\\
  $^3$Institute of Mathematics and Cryptology, %\\
  Cybernetics Faculty, \\
  Military University of Technology,\\% Warsaw\\
  S. Kaliskiego 2, 00-908 Warsaw, Poland\\}

\begin{abstract}
  We consider the Cauchy problem for the axially symmetric Navier--Stokes equations in $\R^3$.
  Our aim is to derive estimates for the scaled vorticity components \(\Phi = \omega_r/r\) and \(\Gamma = \omega_\varphi/r\), measured in the energy norm:
  $$
    X(t) := \|\Phi\|_{V(\bar\R_t^3)} + \|\Gamma\|_{V(\bar\R_t^3)}.
  $$
  The original closure mechanism depends on the relation between the $L^s$ norm and the $L^\infty$ norm of the angular velocity component $v_\varphi$. We identify a critical wedge in the corresponding phase geometry, namely the localized regime
  $$
    W_{A,c_0}^{\rm loc} := \left\{
    t\in(0,T):
    L_s(t)>A,\quad
    \frac{L_s(t)}{M(t)}<c_0
    \right\},
  $$
  where $L_s(t)=\|\widetilde v_\varphi(t)\|_{L^s(\R^3)}$, $M(t)=\|\widetilde v_\varphi(t)\|_{L^\infty(\R^3)}$ for a smoothly localized velocity profile $\widetilde v_\varphi = \zeta_1 v_\varphi$ near the axis of symmetry. We use the convention that the ratio is $+\infty$ when $M(t)=0$.

  The main result is a conditional a priori estimate in which the possible loss of control is measured by the part of the nonlinear interaction accumulated during the times belonging to $W_{A,c_0}^{\rm loc}$. More precisely, denoting this nonnegative wedge contribution by $E_{W,s}$, we obtain
  $$
    X(t)
    \leq
    \Psi_{s,A,c_0}\left(
    \mathrm{data},
    \int_0^t E_{W,s}(\tau)\,d\tau
    \right),
    \qquad 0<t<T,
  $$
  for an increasing positive function $\Psi_{s,A,c_0}$. In particular, if the critical-wedge contribution vanishes, for instance when the trajectory does not enter $W_{A,c_0}^{\rm loc}$, the residual-free a priori estimate depending only on the data is recovered. Under additional regularity assumptions on the force and the initial velocity, we derive the corresponding higher-order Sobolev estimate on every finite time interval on which the wedge residual remains controlled.
  The result does not provide an unconditional global regularity theorem; rather, it isolates the only concentration regime not controlled by the two original closure mechanisms.
\end{abstract}



\section{Introduction}\label{s1}

We study a priori estimates for regular axially symmetric solutions of the Navier--Stokes Cauchy problem
\begin{equation}\eqal{
    &v_{,t}+v\cdot\nabla v-\nu\Delta v+\nabla p=f\quad &{\rm in}\ \ \R^3\times(0,T)\equiv\R_T^3,\cr
    &\divv v=0\quad &{\rm in}\ \ \R^3\times(0,T)\equiv\R_T^3,\cr
    &v|_{t=0}=v(0)\quad &{\rm in}\ \ \R^3,\cr}
  \label{1.1}
\end{equation}
where $v=(v_1(x,t),v_2(x,t),v_3(x,t))\in\R^3$ is the velocity of the fluid, $p=p(x,t)\in\R$ is the pressure, $f=(f_1(x,t),f_2(x,t),f_3(x,t))\in\R^3$ is the external force field, $\nu>0$ is the constant viscosity coefficient.

The main issue is the closure of the energy estimate for the scaled vorticity components $\Phi$ and $\Gamma$. The previously used two-regime argument does not cover a localized concentration regime characterized by large $L^s$-amplitude and a small $L^s/L^\infty$ ratio. We isolate this regime as a localized critical wedge and formulate the resulting estimate conditionally through a wedge residual.

More specifically, for fixed $s>6$, $A>0$, and $c_0>0$, we divide the time interval into the three disjoint localized regimes
\[
  \mathcal R_1
  =
  \bigl\{t\in(0,T): L_s(t)\le A\bigr\},
\]
\[
  \mathcal R_2
  =
  \left\{t\in(0,T):
  L_s(t)>A,\quad \frac{L_s(t)}{M(t)}\ge c_0
  \right\},
\]
and
\[
  \mathcal R_3
  =
  \left\{t\in(0,T):
  L_s(t)>A,\quad \frac{L_s(t)}{M(t)}<c_0
  \right\}
  =
  W_{A,c_0}^{\rm loc}.
\]
The first two regimes correspond to the two available closure mechanisms. The third regime is the localized critical wedge and gives rise to the residual term appearing in the main estimate.

The analysis relies on $H^2$--$H^3$ elliptic estimates for the modified stream function, the maximum principle and energy estimates for the swirl, a local reduction estimate near the axis, and complementary estimates away from the axis. These ingredients yield a conditional estimate for the scaled vorticity variables and, under additional assumptions, a conditional higher-order Sobolev estimate.

We emphasize that the Liu--Wang expansions \cite{LW} of the radial and angular velocity components and of the stream function near the axis of symmetry are essential in the present argument. The small-data continuation result is treated separately in Section~8 and is not used to claim an unconditional global regularity theorem for arbitrary data.

By $x=(x_1,x_2,x_3)$ we denote Cartesian coordinates in $\R^3$. To describe axially-symmetric solutions to (\ref{1.1}) we introduce cylindrical coordinates $r$, $\varphi$, $z$ by the relations
\begin{equation}
  x_1=r\cos\varphi,\quad x_2=r\sin\varphi,\quad x_3=z.
  \label{1.2}
\end{equation}
The three unit vectors along the radial, the angular and the axial directions have the form
\begin{equation}
  \bar e_r=(\cos\varphi,\sin\varphi,0),\quad \bar e_\varphi=(-\sin\varphi,\cos\varphi,0),\quad \bar e_z=(0,0,1).
  \label{1.3}
\end{equation}
Any vector $u$ for the axially symmetric motions can be decomposed as follows
\begin{equation}
  u=u_r(r,z,t)\bar e_r+u_\varphi(r,z,t)\bar e_\varphi+u_z(r,z,t)\bar e_z,
  \label{1.4}
\end{equation}
where $u_r$, $u_\varphi$, $u_z$ are cylindrical coordinates of $u$.

Therefore, velocity $v$ and vorticity $\omega=\rot v$ can be decomposed in the form
\begin{equation}
  v=v_r(r,z,t)\bar e_r+v_\varphi(r,z,t)\bar e_\varphi+v_z(r,z,t)\bar e_z
  \label{1.5}
\end{equation}
and
\begin{equation}
  \omega=\omega_r(r,z,t)\bar e_r+\omega_\varphi(r,z,t)\bar e_\varphi+\omega_z(r,z,t)\bar e_z.
  \label{1.6}
\end{equation}

Problem (\ref{1.1}) in cylindrical coordinates becomes
\begin{equation}\eqal{
  &v_{r,t}+v\cdot\nabla v_r-{v_\varphi^2\over r}-\nu\Delta v_r+\nu{v_r\over r^2}=-p_{,r}+f_r,\cr
  &v_{\varphi,t}+v\cdot\nabla v_\varphi+{v_r\over r}v_\varphi-\nu\Delta v_\varphi+\nu{v_\varphi\over r^2}=f_\varphi,\cr
  &v_{z,t}+v\cdot\nabla v_z-\nu\Delta v_z=-p_{,z}+f_z,\cr
  &(rv_r)_{,r}+(rv_z)_{,z}=0,\cr
  &v_r|_{t=0}=v_r(0),\quad v_\varphi|_{t=0}=v_\varphi(0),\quad v_z|_{t=0}=v_z(0),\cr}
  \label{1.7}
\end{equation}
where
\begin{equation}\eqal{
  &v\cdot\nabla=(v_r\bar e_r+v_z\bar e_z)\cdot\nabla=v_r\partial_r+v_z\partial_z,\cr
  &\Delta u={1\over r}(ru_{,r})_{,r}+u_{,zz}.\cr}
  \label{1.8}
\end{equation}
On the other hand, the vorticity formulation becomes
\begin{equation}\eqal{
  &\omega_{r,t}+v\cdot\nabla\omega_r-\nu\Delta\omega_r+\nu{\omega_r\over r^2}=\omega_rv_{r,r}+\omega_zv_{r,z}+F_r,\cr
  &\omega_{\varphi,t}+v\cdot\nabla\omega_\varphi-{v_r\over r}\omega_\varphi-\nu\Delta\omega_\varphi+\nu{\omega_\varphi\over r^2}={2\over r}v_\varphi v_{\varphi,z}+F_\varphi,\cr
  &\omega_{z,t}+v\cdot\nabla\omega_z-\nu\Delta\omega_z=\omega_rv_{z,r}+\omega_zv_{z,z}+F_z,\cr
  &\omega_r|_{t=0}=\omega_r(0),\quad \omega_\varphi|_{t=0}=\omega_\varphi(0),\quad \omega_z|_{t=0}=\omega_z(0),\cr}
  \label{1.9}
\end{equation}
where $F=\rot f$ and
\begin{equation}
  F=F_r(r,z,t)\bar e_r+F_\varphi(r,z,t)\bar e_\varphi+F_z(r,z,t)\bar e_z.
  \label{1.10}
\end{equation}
The swirl
\begin{equation}
  u=rv_\varphi
  \label{1.11}
\end{equation}
is a solution to the problem
\begin{equation}\eqal{
  &u_{,t}+v\cdot\nabla u-\nu\Delta u+{2\nu\over r}u_{,r}=rf_\varphi\equiv f_0,\cr
  &u|_{t=0}=u(0).\cr}
  \label{1.12}
\end{equation}
Note that
\begin{equation}\eqal{
  &\omega_r=-v_{\varphi,z}=-{1\over r}u_{,z},\cr
  &\omega_\varphi=v_{r,z}-v_{z,r},\cr
  &\omega_z={1\over r}(rv_\varphi)_{,r}=v_{\varphi,r}+{v_\varphi\over r}={1\over r}u_{,r}.\cr}
  \label{1.13}
\end{equation}
Equation $(\ref{1.7})_4$ implies the existence of the stream function $\psi$ which solves the problem
\begin{equation}\eqal{
    &-\Delta\psi+{\psi\over r^2}=\omega_\varphi,\cr
    &\psi\to 0\quad {\rm as}\ \ r+|z|\to\infty.\cr}
  \label{1.14}
\end{equation}
The cylindrical coordinates of velocity can be expressed in terms of the stream function,
\begin{equation}\eqal{
  &v_r=-\psi_{,z},\quad &v_z={1\over r}(r\psi)_{,r}=\psi_{,r}+{\psi\over r},\cr
  &v_{r,r}=-\psi_{,zr},\quad &v_{r,z}=-\psi_{,zz},\cr
  &v_{z,z}=\psi_{,rz}+{\psi_{,z}\over r},\quad &v_{z,r}=\psi_{,rr}+{1\over r}\psi_{,r}-{\psi\over r^2}.\cr}
  \label{1.15}
\end{equation}
Introduce the modified stream function
\begin{equation}
  \psi_1=\psi/r.
  \label{1.16}
\end{equation}
Expressing the cylindrical coordinates of velocity in terms of $\psi_1$ yields
\begin{equation}\eqal{
    &v_r=-r\psi_{1,z},\quad &v_z=(r\psi_1)_{,r}+\psi_1=r\psi_{1,r}+2\psi_1,\cr
    &v_{r,r}=-\psi_{1,z}-r\psi_{1,rz},\quad &v_{r,z}=-r\psi_{1,zz},\cr
    &v_{z,z}=r\psi_{1,rz}+2\psi_{1,z},\quad &v_{z,r}=3\psi_{1,r}+r\psi_{1,rr}.\cr}
  \label{1.17}
\end{equation}
Introduce the notation
\begin{equation}
  (\Phi,\Gamma)=(\omega_r/r,\omega_\varphi/r).
  \label{1.18}
\end{equation}
From \cite[(1.6)]{CFZ} we have
\begin{equation}
  \Phi_{,t}+v\cdot\nabla\Phi-\nu\bigg(\Delta+{2\over r}\partial_r\bigg)\Phi-(\omega_r\partial_r+\omega_z\partial_z){v_r\over r}=F_r/r=F'_r,
  \label{1.19}
\end{equation}
\begin{equation}
  \Gamma_{,t}+v\cdot\nabla\Gamma-\nu\bigg(\Delta+{2\over r}\partial_r\bigg)\Gamma+2{v_\varphi\over r}\Phi=F_\varphi/r\equiv F'_\varphi,
  \label{1.20}
\end{equation}
\begin{equation}
  \Gamma|_{t=0}=\Gamma(0),\quad \Phi|_{t=0}=\Phi(0),
  \label{1.21}
\end{equation}
\begin{equation}
  \Phi\to 0,\ \ \Gamma\to 0\quad {\rm for}\ \ r+|z|\to\infty
  \label{1.22}
\end{equation}
Using (\ref{1.16}) in (\ref{1.14}) yields
\begin{equation}\eqal{
  &-\Delta\psi_1-{2\over r}\psi_{1,r}=\Gamma\cr
  &\psi_1\to 0\ \ {\rm as}\ \ r+|z|\to\infty.\cr}
  \label{1.23}
\end{equation}
Problem (\ref{1.1}) in a cylindrical domain was considered by Zaj\c{a}czkowski \cite{Z1, Z2}, O\.za\'nski and Zaj\c{a}czkowski \cite{OZ}, and Grygierzec and Zaj\c{a}czkowski \cite{GZ1, GZ2, GZ3}.

In these cases the following boundary conditions are assumed: the normal component of velocity and angular component of vorticity vanish on the boundary, and the angular component of velocity on the lateral part and its normal derivative on the top and bottom must vanish also.

In \cite{Z1, Z2} the global a priori estimate for regular solutions is proved under the following Serrin type conditions:
$$
  \psi_1|_{r=0}=0
  \leqno(*)
$$
and there exists a positive constant $c_*$ such that
$$
  {|v_\varphi|_{s,\infty,\Omega^t}\over|v_\varphi|_{\infty,\Omega^t}}\ge c_*\quad {\rm for\ any}\ s.
  \leqno(**)
$$
In \cite{OZ} condition $(*)$ is dropped but the authors still needed $(**)$. Condition $(**)$ is very restrictive. To address the complementary regime in bounded domains, \cite{GZ2} considered the case with the opposite condition. However, these global alternatives do not by themselves describe the pointwise, localized concentration geometry relevant in $\R^3$. Rather than imposing either alternative as a global hypothesis, we introduce a partition of unity $\zeta(r)$, $\vartheta(r)$ near and away from the axis of symmetry and perform a localized analysis. This isolates the unresolved concentration regime as the critical wedge $W_{A,c_0}^{\rm loc}$ and leads to a conditional a priori estimate depending on the residual interaction accumulated inside the wedge.

It was demonstrated by \cite{CFZ} that the solution $v$ is controlled by the energy norms of $\Phi$, $\Gamma$,
$$
  X(t)=\|\Phi\|_{V(\bar\R_t^3)}+\|\Gamma\|_{V(\bar\R_t^3)},
$$
where
$$
  \|\omega\|_{V(\bar\R_t^3)}=\|\omega\|_{L_\infty(0,t;L_2(\bar\R^3))}+\|\nabla\omega\|_{L_2(\bar\R_t^3)}
$$
and
$$\eqal{
    &\bar\R_t^3=\bar\R^3\times(0,t),\quad \bar\R^3=\{x\in\R^3\colon r\in(0,\infty),\ z\in\R^1\},\cr
    &\|u\|_{L_p(\bar\R^3)}=\bigg(\intop_{\bar\R^3}|u(r,z)|^prdrdz\bigg)^{1/p}.\cr}
$$
To derive an estimate for $X$ we have to find independently estimates for
$$
  \tilde X(t)=\|\tilde\Phi\|_{V(\bar\R_t^3)}+\|\tilde\Gamma\|_{V(\bar\R_t^3)}
$$
and also for
$$
  \hat X(t)=\|\hat\Phi\|_{V(\bar\R_t^3)}+\|\hat\Gamma\|_{V(\bar\R_t^3)},
$$
where $\tilde u=u\zeta$, $\hat u=u\vartheta$
$$\eqal{
    &{\rm and}\ \ \zeta(r)=1\ \ &{\rm for}\ \ r\le r_0,\cr
    &{\rm and}\ \ \zeta(r)=0\ \ &{\rm for}\ \ r\ge 2r_0.\cr
    &{\rm Then}\ \ \vartheta(r)=0\ \ &{\rm for}\ \ r\le r_0,\cr
    &{\rm and}\ \ \vartheta(r)=1\ \ &{\rm for}\ \ r\ge 2r_0.\cr}
$$
By data we mean all quantities appeared in Notation \ref{s2.2}. By $\phi$, $\phi_i$, $i=1,\dots$, we mean an increasing positive function.

\textit{Structure of the paper.}

The paper is organized as follows. Section 2 introduces notation and auxiliary results.
Section 3 establishes $H^2$ and $H^3$ estimates for the modified stream function $\psi_1$.
Section 4 contains the core local analysis near the axis with the splitting technique.
Section 5 provides complementary estimates for $\omega_r$ and $\omega_z$.
Section 6 derives bounds for the swirl component $v_\varphi$.
Section 7 shows how to increase regularity to obtain the higher-order Sobolev bounds.
Section 8 discusses the continuation consequences of the conditional estimates and treats separately the small-data regime.

Throughout the paper, the Liu--Wang expansions \cite{LW} for regular solutions near the
axis of symmetry play an essential role.

\vspace{6pt}
\textit{Statement of main results.}

To state the main result, let $\zeta_1=\zeta_1(r)$ be the cutoff function
introduced in Section~\ref{s6}, equal to $1$ on a neighborhood containing
$\operatorname{supp}\zeta$ and vanishing for sufficiently large $r$. We set
\begin{equation}
  v_{\varphi}^{\rm loc}(x,t):=\zeta_1(r)v_\varphi(x,t)
  \label{def:vphi_loc}
\end{equation}
and, for fixed $s>6$, define the instantaneous spatial norms
\begin{equation}
  L_s(t):=\|v_{\varphi}^{\rm loc}(t)\|_{L^s(\bar{\mathbb R}^3)},
  \qquad
  M(t):=\|v_{\varphi}^{\rm loc}(t)\|_{L^\infty(\bar{\mathbb R}^3)}.
  \label{def:Ls_M}
\end{equation}
Here and below the reduced axisymmetric measure is $dx=r\,dr\,dz$.
Equivalently, one may use the corresponding norms on $\mathbb R^3$,
up to the fixed factor $2\pi$. We adopt the convention
\[
  \frac{L_s(t)}{M(t)}=+\infty
  \qquad\text{when }M(t)=0.
\]

Let $A>0$ and $c_0>0$. The interval $(0,T)$ is decomposed into the
three disjoint sets
\begin{equation}
  \mathcal R_1
  :=
  \{t\in(0,T):L_s(t)\leq A\},
  \label{def:R1}
\end{equation}
\begin{equation}
  \mathcal R_2
  :=
  \left\{
  t\in(0,T):
  L_s(t)>A,\quad
  \frac{L_s(t)}{M(t)}\geq c_0
  \right\},
  \label{def:R2}
\end{equation}
and
\begin{equation}
  \mathcal R_3
  :=
  W_{A,c_0}^{\rm loc}
  :=
  \left\{
  t\in(0,T):
  L_s(t)>A,\quad
  \frac{L_s(t)}{M(t)}<c_0
  \right\}.
  \label{def:wedge}
\end{equation}
Thus
\[
  (0,T)=\mathcal R_1\mathbin{\dot\cup}\mathcal R_2
  \mathbin{\dot\cup}W_{A,c_0}^{\rm loc}.
\]

Recall that
\[
  \widetilde\Phi=\Phi\zeta,
  \qquad
  \widetilde\Gamma=\Gamma\zeta,
\]
where $\zeta$ is the cutoff used in the localized vorticity equations.
For almost every $t\in(0,T)$ define the signed instantaneous interaction
\begin{equation}
  i(t)
  :=
  \int_{\bar{\mathbb R}^3}
  \frac{v_\varphi(x,t)}{r}\,
  \widetilde\Phi(x,t)\widetilde\Gamma(x,t)\,dx .
  \label{def:interaction_i}
\end{equation}
The wedge residual is defined by the full nonnegative contribution of this
interaction on the localized critical wedge:
\begin{equation}
  E_{W,s}(t)
  :=
  \mathbf 1_{W_{A,c_0}^{\rm loc}}(t)\,|i(t)|.
  \label{def:residual_E}
\end{equation}
This definition is deliberately conservative. A sharper excess residual
may be introduced after an explicit closable pointwise majorant has been
proved.

\begin{theorem}\label{t1.1}
  Assume that the quantities from Notation~\ref{s2.2} are finite on $(0,T)$.
  Let $v$ be a regular solution to problem~\eqref{1.1} for which
  expansions~\eqref{1.63}--\eqref{1.64} are valid. Fix $s>6$, $A>0$, and
  $c_0>0$. Then there exists an increasing positive function
  $\Psi_{s,A,c_0}$ such that
  \begin{equation}
    X(t)
    \leq
    \Psi_{s,A,c_0}
    \left(
    {\rm data},
    \int_0^t E_{W,s}(\tau)\,d\tau
    \right),
    \qquad 0<t<T,
    \label{1.24}
  \end{equation}
  where data denote the quantities listed in Notation~\ref{s2.2}.
\end{theorem}

\begin{proof}[Proof of Theorem~1.1]

  First we obtain an energy-type estimate for $\widetilde{\Phi}$ and
  $\widetilde{\Gamma}$.
  Localizing~(1.19) and~(1.20) to a neighborhood of the support of $\zeta $,
  multiplying the resulting equations by $\widetilde{\Phi}$ and
  $\widetilde{\Gamma}$, respectively, and integrating over $\bar\R_t^3$,
  we obtain the inequality (see Lemma~4.1)
  \begin{equation}\label{1.25}
    \begin{aligned}
      \widetilde X(t)^2
      \leq {} &
      C_0(\mu,\mathrm{data})H_\mu \cdot \\
              & \cdot \left[
        |I|+\delta_1X(t)^2+
        C_1(\mu,\mathrm{data})(1+\delta_1^{-1})
        \right],
    \end{aligned}
  \end{equation}
  where $\mu>0$ is small, $H_\mu(t) = 1+\|v_\varphi\|_{L^\infty(\bar\R_t^3)}^{2\mu}$, and it is assumed that the integral of the interaction term
  \begin{equation}\label{1.26}
    I = \int_0^t i(\tau)\,d\tau = \int_{\mathbb{R}_{t}^{3}} \frac{v_{\varphi}}{r}
    \,\widetilde{\Phi}\,\widetilde{\Gamma}\,dx\,dt'
  \end{equation}
  is finite.
  Moreover, Lemma~6.1 yields the estimate
  \begin{equation}\label{1.27}
    |v_{\varphi}|_{\infty,\bar{\mathbb{R}}_{t}^{3}}^2
    \le \phi(\text{data})\bigl(X^{3/2}+1\bigr).
  \end{equation}

  Recall that on the left-hand side of~\eqref{1.25} we have $\widetilde{X}^{2}$
  and that $X^2 \le C(\widetilde{X}^2+\widehat{X}^2)$.
  Lemmas~\ref{l4.5} and~\ref{l4.6} imply
  \begin{equation}\label{1.28}
    \widehat{X}^2 \le c\delta_2 D_1^2 X^2 + D_{15}^2.
  \end{equation}
  Choosing $\delta_2$ sufficiently small such that $c\delta_2 D_1^2 \le 1/2$, we can absorb the $\widehat X^2$ component to deduce $X^2 \le C(\widetilde X^2 + D_{15}^2)$.
  Using this bound in~\eqref{1.27} yields
  \begin{equation}\label{1.29}
    \|v_{\varphi}\|_{L_{\infty}(\bar{\mathbb{R}}_{t}^{3})}^2
    \le \phi(\text{data})\bigl(\widetilde{X}^{3/2}+1\bigr).
  \end{equation}

  Now we estimate $I$.
  We abandon the global dichotomy in time and use the pointwise partition of the evolution $(0,t)$. Consequently, the interaction integral $I$ splits into two components:
  \begin{equation}\label{1.30}
    |I| \le \int_{(0,t)\setminus W_{A,c_0}^{\rm loc}} |i(\tau)|\,d\tau + \int_{W_{A,c_0}^{\rm loc}\cap(0,t)} |i(\tau)|\,d\tau.
  \end{equation}

  First, we consider the contribution from the good times $\tau \in (0,t)\setminus W_{A,c_0}^{\rm loc}$. By Corollary~\ref{c4.good.times}, for almost every $\tau$ in this set we have the pointwise bound
  \[
    \|v_\varphi^{\rm loc}(\tau)\|_{L_s(\bar\R^3)} \le A_\sharp := \max\left\{A,\, c_0^{-(s-2)/2}D_1\right\}.
  \]
  This means we can apply Lemma~\ref{l4.4} directly with $E = (0,t)\setminus W_{A,c_0}^{\rm loc}$ and the upper bound $A_\sharp$. This yields
  \begin{equation}\label{1.31}
    \begin{aligned}
      \int_{(0,t)\setminus W_{A,c_0}^{\rm loc}} |i(\tau)|\,d\tau
      \leq {} &
      \phi(\mathrm{data},A_\sharp)
      \left(
      1+|v_\varphi|_{L^\infty(\bar{\mathbb R}^3_t)}^{2\delta}
      \right)^{\alpha_0/2}
      \cdot           \\
              & \cdot
      \left(
      \widetilde X(t)^{2-\alpha_0/2}
      +
      \widetilde X(t)^{2-\alpha_0}
      +1
      \right),
    \end{aligned}
  \end{equation}
  where $\alpha_{0}=\frac{(s-3)(1-b)}{2s} > 0$. This gives a closable bound for the interaction outside the critical wedge.

  Next, for the remaining times $\tau \in W_{A,c_0}^{\rm loc}$, the pointwise interaction may exceed this closable bound. The entire interaction contribution accumulated during the wedge times is recorded in $E_{W,s}(\tau)$. By definition~\eqref{def:residual_E}, we have
  \begin{equation}\label{1.32}
    \int_{W_{A,c_0}^{\rm loc}\cap(0,t)} |i(\tau)|\,d\tau = \int_0^t E_{W,s}(\tau)\,d\tau.
  \end{equation}

  On the other hand, the estimates away from the axis imply
  \[
    X(t)^2\leq C_2\left(\widetilde X(t)^2+D_{15}^2\right).
  \]
  In view of~\eqref{1.25}, we choose
  \[
    \delta_1(t)=\min\left\{\delta_*,\,\frac{1}{4C_0(\mu,\mathrm{data})C_2H_\mu(t)}\right\},
  \]
  where $\delta_*>0$ is the admissible smallness threshold in Lemma~\ref{l4.1}. Consequently,
  \[
    C_0(\mu,\mathrm{data})H_\mu(t)\delta_1(t)X(t)^2
    \leq
    \frac14\widetilde X(t)^2+C(\mathrm{data}),
  \]
  and the first term on the right-hand side can be absorbed.
  Moreover,
  \[
    H_\mu(t)(1+\delta_1(t)^{-1})
    \leq C(\mathrm{data})(1+H_\mu^2)
    \leq C(\mathrm{data})
    \left(1+\widetilde X(t)^{3\mu}\right).
  \]
  Choosing $\mu>0$ so small that $3\mu<2$, this term is
  subquadratic and can be absorbed by Young's inequality.

  Combining the bounds~\eqref{1.31} and~\eqref{1.32}, substituting them into~\eqref{1.25}, and using~\eqref{1.29}, we find that the leading exponent of $\widetilde X$ is
  \[
    \beta = 2 - \frac{\alpha_0}{2} + \frac{3\mu}{2} + \frac{3\delta\alpha_0}{4}.
  \]
  We choose $\mu, \delta > 0$ sufficiently small so that $\frac{3\mu}{2} + \frac{3\delta\alpha_0}{4} < \frac{\alpha_0}{2}$, which ensures $\beta < 2$.
  Let $R_W(t) = \int_0^t E_{W,s}(\tau)\,d\tau$. The wedge residual contribution is bounded by $C(1+\widetilde X^{3\mu/2})R_W(t)$. Applying Young's inequality, we have
  \[
    \widetilde X^{3\mu/2}R_W(t) \le \varepsilon\widetilde X^2 + C_\varepsilon R_W(t)^{2/(2-3\mu/2)}.
  \]
  After absorbing the $\widetilde X^2$ terms into the left-hand side, we deduce
  \[
    \widetilde X(t) \le \Psi_{s,A,c_0}(\mathrm{data}, R_W(t)).
  \]
  The absorbed estimate obtained from~\eqref{1.28} gives
  \[
    X(t)^2
    \leq
    C\left(\widetilde X(t)^2+D_{15}^2\right).
  \]
  Therefore, the conditional bound for $\widetilde X(t)$ implies
  \begin{equation}\label{1.34}
    X(t) \le \Psi_{s,A,c_0}\left(\mathrm{data}, \int_0^t E_{W,s}(\tau)\,d\tau\right).
  \end{equation}
  The assumption $s>6$ in particular guarantees $\alpha_0 > 0$.
  This concludes the proof.
\end{proof}

\begin{theorem}\label{t1.2}
  Let the assumptions of Theorem~\ref{t1.1} hold.
  Assume that
  \[
    f\in W^{1,1/2}_{3}(\mathbb{R}_{t}^{3}),
    \qquad
    v(0)\in B^{7/3}_{3,3}(\mathbb{R}^{3}),
    \qquad
    \operatorname{div}v(0)=0.
  \]
  Then
  \[
    v\in W^{3,3/2}_{3}(\mathbb{R}_{t}^{3}),
    \qquad
    \nabla p\in W^{1,1/2}_{3}(\mathbb{R}_{t}^{3}),
  \]
  and
  \begin{equation}\label{1.53}
    \begin{aligned}
       & \|v\|_{W^{3,3/2}_{3}(\mathbb{R}_{t}^{3})}
      +\|\nabla p\|_{W^{1,1/2}_{3}(\mathbb{R}_{t}^{3})} \\
       &
      \le
      \Psi_{s,A,c_0}\negthinspace\left(
      \text{data},
      \|f\|_{W^{1,1/2}_{3}(\mathbb{R}_{t}^{3})},
      \|v(0)\|_{B^{7/3}_{3,3}(\mathbb{R}^{3})},
      \int_0^t E_{W,s}(\tau)\,d\tau
      \right).
    \end{aligned}
  \end{equation}
\end{theorem}

\begin{proof}[Proof of Theorem~1.2]
  From Theorem~\ref{t1.1} we have the conditional estimate for $X(t)$. We define
  \begin{equation}\label{1.54}
    \mathcal{K}(t) := \Psi_{s,A,c_0}\left(\mathrm{data}, \int_0^t E_{W,s}(\tau)\,d\tau\right).
  \end{equation}
  Since $\|\Phi\|_{V(\bar{\mathbb{R}}_{t}^{3})} + \|\Gamma\|_{V(\bar{\mathbb{R}}_{t}^{3})} \le C X(t)$, we obtain
  \begin{equation}\label{1.45}
    \|\Phi\|_{V(\bar{\mathbb{R}}_{t}^{3})} + \|\Gamma\|_{V(\bar{\mathbb{R}}_{t}^{3})} \le C\mathcal{K}(t).
  \end{equation}
  Then~\eqref{7.20} yields
  \begin{equation}\label{1.55}
    \|v'\|_{L_{\infty}(0,t;L_{6}(\bar{\mathbb{R}}^{3}))}
    \le c\mathcal K(t),
  \end{equation}
  where $v'=(v_r,v_z)$. For the swirl component, interpolation with the energy norm gives
  \[
    \|v_\varphi(t)\|_{L_6(\bar\R^3)}
    \leq
    \|v_\varphi(t)\|_{L^\infty(\bar\R^3)}^{2/3}
    \|v_\varphi(t)\|_{L_2(\bar\R^3)}^{1/3}.
  \]
  Hence
  \[
    \|v_\varphi\|_{L^\infty(0,t;L_6(\bar\R^3))}
    \leq
    D_1^{1/3}\|v_\varphi\|_{L^\infty(\bar\R_t^3)}^{2/3}
    \leq
    \phi(\mathcal K(t),\mathrm{data}).
  \]
  Defining $K_1(t):=\phi(\mathcal K(t),\mathrm{data})$, we conclude that
  \begin{equation}\label{1.55b}
    \|v\|_{L_{\infty}(0,t;L_{6}(\bar{\mathbb{R}}^{3}))}
    \le K_1(t).
  \end{equation}
  Moreover, Lemma~\ref{l2.1} yields
  \begin{equation}\label{1.56}
    \|\nabla v\|_{L_{2}(\bar{\mathbb{R}}_{t}^{3})}\le D_{1}.
  \end{equation}
  The structural decomposition $v=w+z$ and the complete exponent-loss bootstrap are carried out in Lemma~\ref{l7.2}. Applying that lemma with the bounds~\eqref{1.45}, \eqref{1.55b}, and~\eqref{1.56}, the estimate~\eqref{1.53} follows.
\end{proof}
\begin{remark}\label{r1.3}
  Theorem~\ref{t1.1} is conditional. It separates the times belonging to
  $W_{A,c_0}^{\rm loc}$ and records the corresponding interaction through
  $E_{W,s}$. If the trajectory avoids the localized critical wedge, then
  $E_{W,s}=0$ almost everywhere and the residual-free estimate depending
  only on the data is recovered. This does not constitute an unconditional
  global regularity theorem, because avoidance of the wedge is itself an
  additional property of the trajectory.
\end{remark}

\begin{remark}\label{r1.2}
  We emphasize that estimates (\ref{1.25}), (\ref{1.26}), (\ref{1.29}) and (\ref{1.30}) are valid for regular solutions.

  In particular, the estimates need the fact that a smooth $v$ satisfies the following expansions near the axis
  \begin{equation}\eqal{
      &v_r(r,z,t)=a_1(z,t)r+a_2(z,t)r^3+\cdots,\cr
      &v_\varphi(r,z,t)=b_1(z,t)r+b_2(z,t)r^3+\cdots,\cr
      &\psi(r,z,t)=d_1(z,t)r+d_2(z,t)r^3+\cdots,\cr}
    \label{1.63}
  \end{equation}
  where $a_i(z,t)=-d_{i,z}(z,t)$, $i=1,2$, and
  \begin{equation}
    \psi_1(r,z,t)=d_1(z,t)+d_2(z,t)r^2+\cdots
    \label{1.64}
  \end{equation}
  which were shown by Liu and Wang \cite{LW}. The expansions are crucial for the proofs in \cite{Z1, Z2, OZ, GZ1, GZ2,GZ3}.

  To satisfy (\ref{1.63})--(\ref{1.64}) we need that $v\in W_3^{3,3/2}(\mathbb R^3_t)$, $\nabla p\in W_3^{1,1/2}(\mathbb R^3_t)$. We show in Section \ref{s7} that this is true as long as $X(t)$ remains bounded.

  Moreover, (\ref{1.63}) and (\ref{1.64}) are necessary in the proofs of Lemmas \ref{l3.1}--\ref{l3.5}, where $H^2$ and $H^3$ elliptic estimates of $\psi_1$, $\tilde\psi_1$ are derived in terms of~$X$.

  Results of Lemmas \ref{l3.1}--\ref{l3.5} are used in proofs of (\ref{1.25}), (\ref{1.26}), (\ref{1.29}).
\end{remark}

\begin{remark}\label{r1.5}
  However, to prove the main result of this paper, the estimate (\ref{1.24}), we need to use the expansions (\ref{1.63}), (\ref{1.64}), due to Liu and Wang \cite{LW}, so a sufficient regularity of solutions to (\ref{1.1}) near the axis of symmetry must be assumed.

  According to the partial regularity theory of Caffarelli, Kohn and Nirenberg \cite{CKN} any singularity of axisymmetric solutions to (\ref{1.1}) must occur on the axis of symmetry. The methods presented in this paper make use of regular solutions, for which there are no singularities at the axis. Whether it is possible to control $X(t)$ without exploiting these expansions, remains an interesting open problem.
\end{remark}

\section{Notation and auxiliary results}\label{s2}

First we introduce some notations
\subsection{Notation}\label{s2.1}

We use the following notation for Lebesgue and Sobolev spaces
$$
  \|u\|_{L_p(\R^3)}=:|u|_{p,\R^3},\quad \|u\|_{L_p(\R_t^3)}=:|u|_{p,\R_t^3},
$$
$$\eqal{
    \R_t^3&=\R^3\times(0,t),\quad \|u\|_{L_{p,q}(\R_t^3)}=\|u\|_{L_q(0,t;L_p(\R^3))}=:|u|_{p,q,\R_t^3},\cr
    &{\rm where}\quad p,q\in[1,\infty].\cr}
$$
Let $W_p^s(\R^3)$, $s\in\N$ be the Sobolev space with the finite norm
$$
  \|u\|_{W_p^s(\R^3)}=\bigg(\sum_{|\alpha|\le s}\intop_{\R^3}|D_x^\alpha u(x)|^pdx\bigg)^{1/p},
$$
where $D_x^\alpha=\partial_{x_1}^{\alpha_1}\partial_{x_2}^{\alpha_2}\partial_{x_3}^{\alpha_3}$, $|\alpha|=\alpha_1+\alpha_2+\alpha_3$, $\alpha_i\in\N_0=\N\cup\{0\}$, $i=1,2,3$ and $p\in[1,\infty]$.

Let $H^s(\R^3)=W_2^s(\R^3)$. Then we denote
$$\eqal{
    &\|u\|_{H^s(\R^3)}=:\|u\|_{s,\R^3},\quad &\|u\|_{W_p^s(\R^3)}=:\|u\|_{s,p,\R^3},\cr
    &\|u\|_{L_q(0,t;W_p^k(\R^3))}=:\|u\|_{k,p,q,\R_t^3},\quad &\|u\|_{k,p,p,\R_t^3}=:\|u\|_{k,p,\R_t^3},\cr}
$$
where $s,k\in\N_0$.

We need the energy type space $V(\R_t^3)$ appropriate for a description of weak solutions to the Navier-Stokes equations
$$
  \|u\|_{V(\R_t^3)}=|u|_{2,\infty,\R_t^3}+|\nabla u|_{L_2(\R_t^3)}.
$$
We use the notation
$$\eqal{
    \R^3&=\{x=(x_1,x_2,x_3),x_i\in\R^1,i=1,2,3\}\cr
    &=\{(r,\varphi,z)\colon r\in\R_+,\varphi\in[0,2\pi],z\in\R^1\}\cr
    \bar\R^3&=\{(r,z)\colon r\in\R_+,z\in\R^1\}\cr}
$$
and
$$\eqal{
    &\intop_{\R^3}fd\bar x=\intop_{\R^3}fdx_1dx_2dx_3=\intop_{\R^3}frdrdzd\varphi,\cr
    &\intop_{\bar\R^3}fdx=\intop_{\bar\R^3}frdrdz.\cr}
$$
For axially symmetric $f$,
$$
  \intop_{\R^3}fd\bar x=2\pi\intop_{\bar\R^3}fdx.
$$

\subsection{Notation of constants}\label{s2.2}

$$\eqal{
    &D_1^2=3\|f\|_{L_{2,1}(\bar\R_t^3)}^2+2\|v(0)\|_{L_2(\bar\R^3)}\quad ({\rm see}\ (\ref{2.1}),\cr
    &D_2=|f_0|_{\infty,1,\bar\R_t^3}+|u(0)|_{\infty,\bar\R^3}\quad &({\rm see}\ (\ref{2.9})\cr}
$$
$$
  f_0=rf_\varphi,\quad u=rv_\varphi,
$$
\begin{equation*}
  \begin{aligned}
     & D_4 = D_1 D_2 + |u_{,z}(0)|_{2,\bar{\mathbb{R}}^3} + |f_0|_{2,\bar{\mathbb{R}}_t^3}                                            &  & (\text{see } (\ref{2.22})), \\
     & D_5 = D_1 + D_1 D_2 + |u_{,r}(0)|_{2,\bar{\mathbb{R}}^3} + |f_0|_{2,\bar{\mathbb{R}}_t^3}                                      &  & (\text{see } (\ref{2.23})), \\
     & D_6^2 = D_1^2 (D_2 + 1) + D_1 D_2 D_4 + |\tilde{F}_r|_{6/5,2,\bar{\mathbb{R}}_t^3} + |\tilde{\Phi}(0)|_{2,\bar{\mathbb{R}}^3},                                  \\
     & F'_r = F_r / r,\quad \tilde{F}'_r = F'_r \zeta,\quad \tilde{\Phi} = \Phi \zeta,                                                                                 \\
     & F'_\varphi = F_\varphi / r,\quad \tilde{F}'_\varphi = F'_\varphi \zeta,\quad \tilde{\Gamma} = \Gamma \zeta,                                                     \\
     & D_{10}(s) = |f_0|_{\frac{5s}{2s+3},\bar{\mathbb{R}}_t^3} + |u(0)|_{s,\bar{\mathbb{R}}^3}                                       &  & (\text{see } (\ref{2.28})).
  \end{aligned}
\end{equation*}
$$\eqal{
  &\phi_1(s)\ \ {\rm is\ defined\ below}\ (\ref{6.8}),\cr
  &\phi_2(s)\ \ {\rm is\ defined\ below}\ (\ref{6.9}),\cr}
$$
where $s\le{2\over{3\over 2}\beta-1}$, $\beta\in\big({2\over 3},1\big)$,

$$\left\{\eqal{\left.\eqal{
  &D_{11}=(D_1+D_4+D_5)D_2,\quad D_{12}=(D_1+D_4)D_1^{1-\varepsilon_0}{r_0^{\varepsilon_0}\over\varepsilon_0},\cr
  &D_8=D_1D_{11}+D_1(D_1+D_4+D_5)(D_2+1)+|\tilde F_r|_{6/5,2,\bar\R_t^3}^2\cr
  &+|\tilde F_z|_{6/5,2,\bar\R_t^3}^2+|\tilde\omega_r(0)|_{2,\bar\R^3}^2+|\tilde\omega_z(0)|_{2,\bar\R^3}^2\quad {\rm see\ the\ }\cr
  &\hskip5.0cm{\rm assumptions\ of\ Lemma\ \ref{l5.1}}\cr}\right\}\ &{\rm see}\ (\ref{5.37})\cr
  &\hskip-10,8cm D_9=D_1D_2^{2(1-\varepsilon_0)}{r_0^{2\varepsilon_0}\over\varepsilon_0^2}\ &\hskip-2.0cm{\rm see}\ (\ref{5.40})\cr}\right.
$$
$$\eqal{
    &D_{13}=D_{11}+D_{12}^2+(1+D_9)\quad &({\rm see}\ (\ref{5.41})),\cr
    &D_{14}=D_1^2D_{12}^2+(1+D_1D_9)\quad &({\rm see}\ (\ref{5.41})),\cr
    &D_{15}^2=D_1^2D_2^2+D_1^2D_{10}^2(3)+D_1D_4+c(1/\delta_2)D_1^2+D_{10}^2(6/5)\cr
    &+D_2D_1^2+|\hat F'_r|_{6/5,2,\bar\R_t^3}^2+|\hat F_\varphi|_{6/5,2,\bar\R_t^3}^2+|\hat\Phi(0)|_{2,\bar\R^3}^2+|\hat\Gamma(0)|_{2,\bar\R^3}^2,\cr
    &D_{16}^2=2D_2|f_\varphi/r|_{\infty,1,\bar\R_t^3}+|v_\varphi(0)|_{\infty,\bar\R^3}^2,\cr}
$$

$$\eqal{
  &D_{17}={r_0^{2\delta}\over\delta^2}D_2^{4-2\delta},\quad D_{18}=(4-2\delta)[|\tilde f'_\varphi|_{10/(1+6\delta),\bar\R_t^3}D_1^{3-2\delta}+D_1^2],\cr
  &D_{19}=|\tilde v'_\varphi(0)|_{s,\bar\R^3}^{4-2\delta},\quad \tilde f'_\varphi=f_\varphi\zeta_1,\quad \tilde v'_\varphi=v_\varphi\zeta_1,\quad \zeta_1=\begin{cases}1&r\le 2r_0\cr 0&r\ge 4r_0\cr\end{cases}\cr}
$$

\subsection{Basic estimates}\label{s2.3}

\begin{lemma}\label{l2.1}
  Let $f\in L_{2,1}(\bar\R^3\times(0,t))$, $v(0)\in L_2(\bar\R^3)$, $\bar\R^3=\{(r,z)\colon r\in\R_+,z\in\R\}$. Let solutions to (\ref{1.7}) be sufficiently regular.\\
  Then solutions to (\ref{1.7}) vanishing sufficiently fast as $r+|z|$ tends to infinity satisfy the estimate
  \begin{equation}\eqal{
      &\|v(t)\|_{L_2(\bar\R^3)}^2+\nu\intop_0^tdt'\intop_{\bar\R^3}(|\nabla v_r|^2+|\nabla v_\varphi|^2+|\nabla v_z|^2)dx\cr
      &\quad+\nu\intop_0^tdt'\intop_{\bar\R^3}\bigg({v_r^2\over r^2}+{v_\varphi^2\over r^2}\bigg)dx\le 3\|f\|_{L_{2,1}(\bar\R_t^3)}^2+2\|v(0)\|_{L_2(\bar\R^3)}^2\equiv D_1^2,\cr}
    \label{2.1}
  \end{equation}
  where $\bar\R_t^3=\bar\R^3\times(0,t)$, $dx=rdrdz$.
\end{lemma}

\begin{proof}
  Introduce the notation
  $$
    r_\varepsilon=\left\{\eqal{
      &r\ \ &r\ge\varepsilon\cr
      &\varepsilon\ \ &r<\varepsilon\cr}\right.
  $$
  Then (\ref{1.7}) for $v_r^\varepsilon$, $v_\varphi^\varepsilon$, $v_z^\varepsilon$, $p^\varepsilon$ has the form

  \begin{equation}
    \begin{aligned}
       & v_{r,t}^\varepsilon + v^\varepsilon \cdot \nabla v_r^\varepsilon - \frac{(v_\varphi^\varepsilon)^2}{r_\varepsilon} - \nu\Delta v_r^\varepsilon + \nu\frac{v_r^\varepsilon}{r_\varepsilon^2} = -p_{,r}^\varepsilon + f_r
       & \text{in } \mathbb{R}^3 \times (0,t),                                                                                                                                                                                                      \\
       & v_{\varphi,t}^\varepsilon + v^\varepsilon \cdot \nabla v_\varphi^\varepsilon + \frac{v_r^\varepsilon}{r_\varepsilon}v_\varphi^\varepsilon - \nu\Delta v_\varphi^\varepsilon + \nu\frac{v_\varphi^\varepsilon}{r_\varepsilon^2} = f_\varphi
       & \qquad \text{in } \mathbb{R}^3 \times (0,t),                                                                                                                                                                                               \\
       & v_{z,t}^\varepsilon + v^\varepsilon \cdot \nabla v_z^\varepsilon - \nu\Delta v_z^\varepsilon = -p_{,z}^\varepsilon + f_z
       & \qquad \text{in } \mathbb{R}^3 \times (0,t),                                                                                                                                                                                               \\
       & (rv_r^\varepsilon)_{,r} + (rv_z^\varepsilon)_{,z} = 0
       & \qquad \text{in } \mathbb{R}^3 \times (0,t),                                                                                                                                                                                               \\
       & v_r^\varepsilon|_{t=0} = v_r^\varepsilon(0),\ \ v_\varphi^\varepsilon|_{t=0} = v_\varphi^\varepsilon(0),\ \ v_z^\varepsilon|_{t=0} = v_z^\varepsilon(0)
       & \qquad \text{in } \mathbb{R}^3.
    \end{aligned}
    \label{2.2}
  \end{equation}

  Multiply $(\ref{2.2})_1$ by $v_r^\varepsilon$, $(\ref{2.2})_2$ by $v_\varphi^\varepsilon$, $(\ref{2.2})_3$ by $v_z^\varepsilon$, add the results and integrate over $\R^3$. Then we get
  \begin{equation}\eqal{
      &{1\over 2}{d\over dt}\intop_{\R^3}(|v_r^\varepsilon|^2+|v_\varphi^\varepsilon|^2+|v_z^\varepsilon|^2)d\bar x+\nu\intop_{\R^3}(|\nabla v_r^\varepsilon|^2+|\nabla v_\varphi^\varepsilon|^2+|\nabla v_z^\varepsilon|^2)d\bar x\cr
      &\quad+\nu\intop_{\R^3}\bigg({|v_r^\varepsilon|^2\over r_\varepsilon^2}+{|v_\varphi^\varepsilon|^2\over r_\varepsilon^2}\bigg)d\bar x+\intop_{\R^3}(p_{,r}^\varepsilon v_r^\varepsilon+p_{,z}^\varepsilon v_z^\varepsilon)d\bar x\cr
      &=\intop_{\R^3}(f_rv_r^\varepsilon+f_\varphi v_\varphi^\varepsilon+f_zv_z^\varepsilon)d\bar x,\cr}
    \label{2.3}
  \end{equation}
  where $d\bar x=rdrd\varphi dz=dxd\varphi$.

  Integrating (\ref{2.3}) with respect to time and passing with $\varepsilon\to 0$, we obtain
  \begin{equation}\eqal{
    &{1\over 2}|v(t)|_{2,\R^3}^2+\nu\intop_{\R_t^3}(|\nabla v_r|^2+|\nabla v_\varphi|^2+|\nabla v_z|^2)d\bar xdt'\cr
    &\quad+\nu\intop_{\R_t^3}\bigg({v_r^2\over r^2}+{v_\varphi^2\over r^2}\bigg)d\bar xdt'+\intop_{\R_t^3}(p_{,r}v_r+p_{,z}v_z)d\bar xdt'\cr
    &=\intop_{\R_t^3}(f_rv_r+f_\varphi v_\varphi+f_zv_z)d\bar xdt'+{1\over 2}|v(0)|_{2,\R^3}^2.\cr}
    \label{2.4}
  \end{equation}
  The last term on the l.h.s. denoted by $I$, equals
  $$\eqal{
    I&=\intop_{\R_t^3}(pv_rr)_{,r}drdzd\varphi dt'+\intop_{\R_t^3}(pv_zr)_{,z}drdzd\varphi dt'\cr
    &\quad-\intop_{\R_t^3}p[(v_rr)_{,r}+(v_zr)_{,z}]drdzd\varphi dt'.\cr}
  $$
  Since $\lim_{r+|z|\to\infty}v_r=0$, $\lim_{r+|z|\to\infty}v_z=0$, the second term in $I$ vanishes and the first equals
  $$
    \lim_{r\to\infty}|p(r)|_{2,\R_t^2}|v_r(r)|_{2,\R_t^2}r=0,
  $$
  where
  $$
    |u(r)|_{2,\R^2}=\bigg(\intop_{\R^2}|u(r)|^2dzd\varphi\bigg)^{1/2}
  $$
  and we used that $p$ and $v$ are sufficiently regular.

  The last term in $I$ vanishes in view of $(\ref{1.7})_4$.

  Hence (\ref{2.4}) implies
  \begin{equation}
    {1\over 2}\sup_t|v(t)|_{2,\R^3}^2\le|f|_{2,1,\R_t^3}\sup_t|v(t)|_{2,\R^3}+{1\over 2}|v(0)|_{2,\R^3}^2.
    \label{2.5}
  \end{equation}
  The above inequality implies
  \begin{equation}
    \sup_t|v(t)|_{2,\R^3}^2\le 4|f|_{2,1,\R_t^3}^2+2|v(0)|_{2,\R^3}^2.
    \label{2.6}
  \end{equation}
  Using the estimate in (\ref{2.4}), we obtain
  \begin{equation}\eqal{
      &{1\over 2}\sup_t|v(t)|_{2,\R^3}^2+2\nu\intop_{\R_t^3}(|\nabla v_r|^2+|\nabla v_\varphi|^2+|\nabla v_z|^2)d\bar xdt'\cr
      &\quad+2\nu\intop_{\R_t^3}\bigg({v_r^2\over r^2}+{v_\varphi^2\over r^2}\bigg)d\bar xdt'\le|f|_{2,1,\R_t^3}(2|f|_{2,1,\R_t^3}+2|v(0)|_{2,\R^3}^2)\cr
      &\quad+{1\over 2}|v(0)|_{2,\R^3}^2.\cr}
    \label{2.7}
  \end{equation}
  Dropping the integration with respect to $\varphi$ yields (\ref{2.1}). This ends the proof.
\end{proof}

\begin{remark}\label{r2.2}
  The last term on the l.h.s. of (\ref{2.1}) implies
  \begin{equation}
    v_r|_{r=0}=v_\varphi|_{r=0}=0.
    \label{2.8}
  \end{equation}
\end{remark}

\begin{lemma}\label{l2.3}
  Consider problem (\ref{1.12}). Assume that $f_0\in L_{\infty,1}(\bar\R_t^3)$ and $u(0)\in L_\infty(\bar\R^3)$. Using that $v_\varphi|_{r=0}=0$ and that $u$ vanishes sufficiently fast as $r+|z|\to\infty$ we obtain
  \begin{equation}
    |u(t)|_{\infty,\bar\R^3}\le|f_0|_{\infty,1,\bar\R_t^3}+|u(0)|_{\infty,\bar\R^3}\equiv D_2.
    \label{2.9}
  \end{equation}
\end{lemma}

\begin{proof}
  Multiplying $(\ref{1.12})_1$ by $u|u|^{s-2}$, $s>2$, integrating over $\bar\R^3$ and by parts, we obtain
  \begin{equation}\eqal{
    &{1\over s}{d\over dt}|u|_{s,\bar\R^3}^s+{4\nu(s-1)\over s^2}|\nabla|u|^{s/2}|_{2,\bar\R^3}^2+{\nu\over s}\intop_{\bar\R^3}(|u|^s)_{,r}drdz\cr
    &=\intop_{\bar\R^3}f_0u|u|^{s-2}dx.\cr}
    \label{2.10}
  \end{equation}
  In view of Remark \ref{r2.2} and that $\lim_{r\to\infty}u=0$ it follows that the last term on the l.h.s. of (\ref{2.10}) vanishes. Simplifying (\ref{2.10}) yields
  \begin{equation}
    {d\over dt}|u|_{s,\bar\R^3}\le|f_0|_{s,\bar\R^3}.
    \label{2.11}
  \end{equation}
  Integrating with respect to time and passing with $s\to\infty$ we derive (\ref{2.9}). This ends the proof.
\end{proof}

\begin{lemma}\label{l2.4}
  Let estimates (\ref{2.1}) and (\ref{2.9}) hold. Then
  \begin{equation}
    |v_\varphi|_{4,\bar\R_t^3}\le D_1^{1/2}D_2^{1/2}.
    \label{2.12}
  \end{equation}
\end{lemma}

\begin{proof}
  See the proof of Lemma \ref{l2.4} in \cite{Z1}.
\end{proof}

\begin{lemma}[see Lemma 2.16 in \cite{BIN}]\label{l2.5}
  Let $1\le p\le\infty$, $\beta\not=1/p$.\\
  Let $F(x)=\intop_0^xf(y)dy$ for $\beta>{1\over p}$, $F(x)=\intop_x^\infty f(y)dy$ for $\beta<{1\over p}$.\\
  Then
  \begin{equation}
    |x^{-\beta}F(x)|_{p,\R_+}\le{1\over|\beta-{1\over p}|}|x^{-\beta+1}f(x)|_{p,\R_+}.
    \label{2.13}
  \end{equation}
\end{lemma}

\begin{lemma}\label{l2.6}
  Let $\psi$ be a solution to (\ref{1.14}) and $\psi_1$ be the modified stream function (see (\ref{1.16})). Let $\psi|_{r=0}=0$. Then the following estimates hold
  \begin{equation}
    \sup_t(\|\psi\|_{1,\bar\R^3}^2+|\psi_1|_{2,\bar\R^3}^2)\le cD_1^2
    \label{2.14}
  \end{equation}
  and
  \begin{equation}
    \intop_0^t(\|\psi_{,z}\|_{1,\bar\R^3}^2+|\psi_{1,z}|_{2,\bar\R^3}^2)dt'\le cD_1^2.
    \label{2.15}
  \end{equation}
\end{lemma}

\begin{proof}
  Multiply (\ref{1.14}) by $\psi$ and integrate over $\R^3$. Integrating by parts yields
  $$\eqal{
    &|\nabla\psi|_{2,\R^3}^2+|\psi_1|_{2,\R^3}^2=\intop_{\R^3}\omega_\varphi\psi d\bar x=\intop_{\R^3}(v_{r,z}-v_{z,r})\psi d\bar x\cr
    &=\intop_\Omega(v_z\psi_{,r}-v_r\psi_{,z})d\bar x+\intop_\Omega v_z{\psi\over r}d\bar x\cr
    &\le{1\over 2}\bigg(|\psi_{,r}|_{2,\R^3}^2+|\psi_{,z}|_{2,\R^3}^2+\bigg|{\psi\over r}\bigg|_{2,\R^3}^2\bigg)+{1\over 2}(|v_r|_{2,\R^3}^2+2|v_z|_{2,\R^3}^2),\cr}
  $$
  where we used that $\psi|_{r=0}=0$.

  Dropping the integration with respect to $\varphi$ yields (\ref{2.14}).

  Differentiate (\ref{1.14}) with respect to $z$, multiply by $\psi_{,z}$ and integrate over $\R_t^3$. Integrating by parts implies
  $$\eqal{
    &\intop_{\R_t^3}|\nabla\psi_{,z}|^2d\bar xdt'+\intop_{\R_t^3}|\psi_{1,z}|^2d\bar xdt'=\intop_{\R_t^3}\omega_{\varphi,z}\psi_{,z}d\bar xdt'\cr
    &=-\intop_{\R_t^3}\omega_\varphi\psi_{,zz}d\bar xdt'\le{1\over 2}|\psi_{,zz}|_{2,\R_t^3}^2+{1\over 2}|\omega_\varphi|_{2,\R_t^3}^2.\cr}
  $$
  Dropping the integration with respect to $\varphi$ gives (\ref{2.15}). This ends the proof.
\end{proof}

Lemma 2.4 from \cite{CFZ} yields

\begin{lemma}\label{l2.7}
  Let $f\in C_0^\infty(\bar\R^3)$. Let $1<p<3$, $0\le s\le p$, $s\le 2$, $q\in\big[p,{p(3-s)\over 3-p}\big]$. Then there exists a positive constant $c=c(p,s)$ such that
  \begin{equation}
    \bigg(\intop_{\bar\R^3}{|f|^q\over r^s}dx\bigg)^{1/q}\le c|f|_{p,\bar\R^3}^{{3-s\over q}-{3\over p}+1}|\nabla f|_{p,\bar\R^3}^{{3\over p}-{3-s\over q}},
    \label{2.16}
  \end{equation}
  where $f$ does not depend on $\varphi$.
\end{lemma}
\goodbreak

\begin{remark}\label{r2.8}
  Let $(v,p)$ be a solution to (\ref{1.1}) such that $v\in W_2^{3,3/2}(\R_t^3)$, $\nabla p\in W_2^{1,1/2}(\R_t^3)$.

  Then \cite{LW} shows the following expansions near the axis of symmetry
  \begin{equation}
    v_r(r,z,t)=a_1(z,t)r+a_2(z,t)r^3+\cdots,
    \label{2.17}
  \end{equation}
  \begin{equation}
    v_\varphi(r,z,t)=b_1(z,t)r+b_2(z,t)r^3+\cdots,
    \label{2.18}
  \end{equation}
  \begin{equation}
    \psi(r,z,t)=d_1(z,t)r+d_2(z,t)r^3+\cdots,.
    \label{2.19}
  \end{equation}
  Then
  \begin{equation}
    \psi_1(r,z,t)=d_1(z,t)+d_2(z,t)r^2+\cdots,
    \label{2.20}
  \end{equation}
  \begin{equation}
    \psi_{1,r}(r,z,t)=2d_2(z,t)r+\cdots,
    \label{2.21}
  \end{equation}
  where $a_1=-d_{1,z}$, $a_2=-d_{2,z}$.
\end{remark}

\begin{lemma}\label{l2.9}
  For solutions to problem (\ref{1.12}) the following estimates hold
  \begin{equation}
    \|u_{,z}\|_{V(\bar\R_t^3)}\le c(D_1D_2+|u_{,z}(0)|_{2,\bar\R^3}+|f_0|_{2,\bar\R_t^3})\equiv cD_4
    \label{2.22}
  \end{equation}
  and
  \begin{equation}
    \|u_{,r}\|_{V(\bar\R_t^3)}\le c(D_1+D_1D_2+|u_{,r}(0)|_{2,\bar\R^3}+|f_0|_{2,\bar\R_t^3})\equiv cD_5.
    \label{2.23}
  \end{equation}
\end{lemma}

\begin{proof}
  Differentiate (\ref{1.12}) with respect to $z$, multiply by $u_{,z}$ and integrate over $\bar\R^3$. Then we get
  \begin{equation}\eqal{
      &{1\over 2}{d\over dt}|u_{,z}|_{2,\bar\R^3}^2+\intop_{\bar\R^3}v\cdot\nabla u_{,z}u_{,z}dx+\intop_{\bar\R^3}v_{,z}\cdot\nabla uu_{,z}dx\cr
      &\quad+\nu|\nabla u_{,z}|_{2,\bar\R^3}^2+2\nu\intop_{\bar\R^3}{1\over r}u_{,zr}u_{,z}dx=\intop_{\bar\R^3}f_{0,z}u_{,z}dx.\cr}
    \label{2.24}
  \end{equation}
  The second integral vanishes after integration by parts. The third term, after integration by parts, equals
  $$
    \intop_{\bar\R^3}v_{,z}\cdot\nabla u_{,z}udx.
  $$
  Using the H\"older and Young inequalities the integral is bounded by
  $$
    {\nu\over 4}|\nabla u_{,z}|_{2,\bar\R^3}^2+c|v_{,z}|_{2,\bar\R^3}^2D_2^2.
  $$
  The last integral on the l.h.s. of (\ref{2.24}) equals
  $$
    -\nu\intop_{\R^1}u_{,z}^2|_{r=0}dz.
  $$
  It vanishes in view of the expansion (\ref{2.18}).

  Integrating by parts in the term on the r.h.s. of (\ref{2.24}) and applying the H\"older and Young inequalities to obtain the bound
  $$
    {\nu\over 4}|u_{,zz}|_{2,\bar\R^3}^2+c|f_0|_{2,\bar\R^3}^2.
  $$
  Using the results in (\ref{2.24}) yields
  \begin{equation}
    {d\over dt}|u_{,z}|_{2,\bar\R^3}^2+\nu|\nabla u_{,z}|_{2,\bar\R^3}^2\le c|v_{,z}|_{2,\bar\R^3}^2D_2^2+c|f_0|_{2,\bar\R^3}^2.
    \label{2.25}
  \end{equation}
  Integrating (\ref{2.25}) with respect to time yields (\ref{2.22}).

  Differentiating (\ref{1.12}) with respect to $r$, multiplying the result by $u_{,r}$ and integrating over $\bar\R^3$, we obtain
  \begin{equation}\eqal{
      &{1\over 2}{d\over dt}|u_{,r}|_{2,\bar\R^3}^2+\intop_{\bar\R^3}v\cdot\nabla u_{,r}u_{,r}dx+\intop_{\bar\R^3}v_{,r}\cdot\nabla uu_{,r}dx\cr
      &\quad-\nu\intop_{\bar\R^3}(\Delta u)_{,r}u_{,r}dx+2\nu\intop_{\bar\R^3}{1\over r}u_{,rr}u_{,r}dx-2\nu\intop_{\bar\R^3}{u_{,r}^2\over r^2}dx\cr
      &=\intop_{\bar\R^3}f_{0,r}u_{,r}dx.\cr}
    \label{2.26}
  \end{equation}
  After integration by parts the second term in (\ref{2.26}) vanishes. The third term equals
  $$
    \intop_{\bar\R^3}v_{,r}\cdot\nabla(uu_{,r})dx-\intop_{\bar\R^3}v_{,r}\cdot\nabla u_{,r}udx\equiv I_1+I_2,
  $$
  where
  $$
    I_1=\intop_{\bar\R^3}\divv[v_{,r}u\cdot u_{,r}]dx-\intop_{\bar\R^3}\divv v_{,r}uu_{,r}dx,
  $$
  where the first term vanishes and the second equals
  $$
    \intop_{\bar\R^3}{v_r\over r^2}u_{,r}udx
  $$
  because $\divv v_{,r}=(\divv v)_{,r}+{v_r\over r^2}={v_r\over r^2}$.

  Next,
  $$
    |I_2|\le{\nu\over 4}|\nabla u_{,r}|_{2,\bar\R^3}^2+cD_2^2|v_{,r}|_{2,\bar\R^3}^2.
  $$
  Since, $(\Delta u)_{,r}=\Delta u_{,r}-{1\over r^2}u_{,r}$, we have
  $$
    -\intop_{\bar\R^3}(\Delta u)_{,r}u_{,r}=\intop_{\bar\R^3}|\nabla u_{,r}|^2+\intop_{\bar\R^3}{1\over r^2}u_{,r}^2dx.
  $$
  In virtue of the above results, (\ref{2.26}) takes the form
  \begin{equation}\eqal{
    &{1\over 2}{d\over dt}|u_{,r}|_{2,\bar\R^3}^2+{3\nu\over 4}|\nabla u_{,r}|_{2,\bar\R^3}^2+\nu\intop_{\bar\R^3}(u_{,r}^2)_{,r}drdz\cr
    &\le 5\nu\intop_{\bar\R^3}{u_{,r}^2\over r^2}dx+{\nu\over 4}\intop_{\bar\R^3}|u_{,rr}|^2dx+c|f_0|_{2,\bar\R^3}^2\cr
    &\quad+cD_2^2|v_{,r}|_{2,\bar\R^3}^2+\bigg|\intop_{\bar\R^3}{v_r\over r^2}u_{,r}udx\bigg|,\cr}
    \label{2.27}
  \end{equation}
  where we applied the H\"older and Young inequalities to the term on the r.h.s. of (\ref{2.26}).

  The last term on the l.h.s. of (\ref{2.27}) equals
  $$
    -
    \intop_{\R^1}u_{,r}^2|_{r=0}dz
  $$
  which vanishes in view of (\ref{2.18}).

  Integrating (\ref{2.27}) with respect to time and using (\ref{2.1}) and (\ref{2.9}), we obtain
  $$
    \|u_{,r}\|_{V(\bar\R_t^3)}^2\le c(D_1^2+D_2^2D_1^2+D_2D_1^2+|f_0|_{2,\bar\R_t^3}^2+|u_{,r}(0)|_{2,\bar\R^3}^2).
  $$
  The above inequality implies (\ref{2.23}) and concludes the proof.
\end{proof}

\begin{lemma}\label{l2.11}
  Let $u$ be a solution to problem (\ref{1.12}). Assume $s>1$, $f_0\in L_{5s\over 2s+3}(\bar\R_t^3)$, $u(0)\in L_s(\bar\R^3)$, $\nu_0=\min\big\{1,{4\nu(s-1)\over s}\big\}>0$.\\
  Then for $s$ finite,
  \begin{equation}
    |u|_{{5\over 3}s,\bar\R_t^3}\le cs|f_0|_{{5s\over 2s+3},\bar\R_t^3}+c|u(0)|_{s,\bar\R^3}\equiv cD_{10}(s).
    \label{2.28}
  \end{equation}
\end{lemma}

\begin{proof}
  Integrating (\ref{2.10}) with respect to time yields
  \begin{equation}
    |u(t)|_{s,\bar\R^3}^s+{4\nu(s-1)\over s}|\nabla|u|^{s/2}|_{2,\bar\R_t^3}^2\le cs\intop_{\bar\R_t^3}|f_0|\,|u|^{s-1}dxdt+c|u(0)|_{s,\bar\R^3}^s.
    \label{2.29}
  \end{equation}
  Let $\nu_0=\min\big\{1,{4\nu(s-1)\over s}\big\}$, $s>1$.

  Then (\ref{2.29}) implies
  $$
    \nu_0|u|_{{5\over 3}s,\bar\R_t^3}^s\le sc|f_0|_{{5s\over 2s+3},\bar\R_t^3}|u|_{{5\over 3}s,\bar\R_t^3}^{s-1}+c|u(0)|_{s,\bar\R^3}^s.
  $$
  Applying the Young inequality, we get
  \begin{equation}
    |u|_{{5\over 3}s,\bar\R_t^3}^s\le cs^s|f_0|_{{5s\over 2s+3},\bar\R_t^3}^s+c|u(0)|_{s,\bar\R^3}^s.
    \label{2.30}
  \end{equation}
  The above inequality implies (\ref{2.29}). This concludes the proof.
\end{proof}

For the reader convenience we recall the interpolation (see \cite[Sect. 15]{BIN})

\begin{lemma}\label{l2.12}
  Let $\theta$ satisfy the equality
  \begin{equation}
    {n\over p}-r=(1-\theta){n\over p_1}+\theta\bigg({n\over p_2}-l\bigg),\quad {r\over l}\le\theta\le 1,
    \label{2.31}
  \end{equation}
  where $1\le p_i\le\infty$, $i=1,2$, $0\le r<l$.

  Then the interpolation holds
  \begin{equation}
    \sum_{|\alpha|=r}|D_x^\alpha f|_{p,\Omega}\le c|f|_{p_1,\Omega}^{1-\theta}\|f\|_{W_{p_2}^l(\Omega)}^\theta,
    \label{2.32}
  \end{equation}
  where $\Omega\subset\R^n$ and $D_x^\alpha f=\partial_{x_1}^{\alpha_1}\dots\partial_{x_n}^{\alpha_n}f$, $|\alpha|=\alpha_1+\alpha_2+\dots+\alpha_n$.
\end{lemma}

\begin{definition}\label{d2.12}
  (Anisotropic Sobolev and Sobolev-Slobodetskii spaces). We denote by

  \begin{itemize}
    \item[1.] $W_{p,p_0}^{k,k/2}(\Omega^T)$, $k,k/2\in\N\cup\{0\}$, $p,p_0\in[1,\infty]$ -- the anisotropic Sobolev space with a mixed norm, which is a completion of $C^\infty(\Omega^T)$-functions under the norm
          $$
            \|u\|_{W_{p,p_0}^{k,k/2}(\Omega^T)}=\bigg(\intop_0^T\bigg(\sum_{|\alpha|+2a\le k}\intop_\Omega|D_x^\alpha\partial_t^au|^p\bigg)^{p_0/p}dt\bigg)^{1/p_0}.
          $$
    \item[2.] $W_{p,p_0}^{s,s/2}(\Omega^T)$, $s\in\R_+$, $p,p_0\in[1,\infty]$ -- the Sobolev-Slobodetskii space with the finite norm
          $$\eqal{
            &\|u\|_{W_{p,p_0}^{s,s/2}(\Omega^T)}=\sum_{|\alpha|+2a\le [s]}\|D_x^\alpha\partial_t^au\|_{L_{p,p_0}(\Omega^T)}\cr
            &+\bigg[\intop_0^T\bigg(\intop_\Omega\intop_\Omega\sum_{|\alpha|+2a=[s]}\!\!\! {|D_x^\alpha\partial_t^au(x,t)-D_{x'}^\alpha\partial_t^au(x',t)|^p\over|x-x'|^{n+p(s-[s])}}dxdx'\bigg)^{p_0/p}\!\! dt\bigg]^{1/p_0}\cr
            &+\bigg[\intop_\Omega\bigg(\intop_0^T\intop_0^T\sum_{|\alpha|+2a=2[s/2]}\!\!\! {|D_x^\alpha\partial_t^au(x,t)-D_x^\alpha\partial_{t'}^au(x,t')|^{p_0}\over|t-t'|^{1+p_0({s\over 2}-[{s\over 2}])}}dtdt'\bigg)^{p/p_0}\!\!dx\bigg]^{1/p},\cr}
          $$
          where $a\in\N\cup\{0\}$, $[s]$ is the integer part of $s$ and $D_x^\alpha$ denotes the partial derivative in the spatial variable $x$ corresponding to multiindex $\alpha$. The second term in the above norm vanishes if $s\in\mathbb N$, and the third term vanishes if $s/2\in\mathbb N$. We also use notation $L_p(\Omega^T)=L_{p,p}(\Omega^T)$, $W_p^{s,s/2}(\Omega^T)=W_{p,p}^{s,s/2}(\Omega^T)$.

    \item[3.] $B_{p,p_0}^l(\Omega)$, $l\in\R_+$, $p,p_0\in[1,\infty)$ -- the Besov space with the finite norm
          $$
            \|u\|_{B_{p,p_0}^l(\Omega)}=\|u\|_{L_p(\Omega)}+\bigg(\sum_{i=1}^n\intop_0^\infty {\|\Delta_i^m(h,\Omega)\partial_{x_i}^ku\|_{L_p(\Omega)}^{p_0}\over h^{1+(l-k)p_0}}dh\bigg)^{1/p_0},
          $$
          where $k\in\N\cup\{0\}$, $m\in\N$, $m>l-k>0$, $\Delta_i^j(h,\Omega)u$, $j\in\N$, $h\in\R_+$ is the finite difference of the order $j$ of the function $u(x)$ with respect to $x_i$ with
          $$\eqal{
              &\Delta_i^1(h,\Omega)u=\Delta_i(h,\Omega)\cr
              &\quad=u(x_1,\dots,x_{i-1},x_i+h,x_{i+1},\dots,x_n)-u(x_1,\dots,x_n),\cr
              &\Delta_i^j(h,\Omega)=\Delta_i(h,\Omega)\Delta_i^{j-1}(h,\Omega)u\quad {\rm and}\ \ \Delta_i^j(h,\Omega)u=0\cr
              &{\rm for}\ \ x+jh\not\in\Omega.\cr}
          $$
          In has been proved in \cite{G} that the norms of the Besov space $B_{p,p_0}^l(\Omega)$ are equivalent for different $m$ and $k$ satisfying the condition\break $m>l-k>0$.
  \end{itemize}
\end{definition}

We need the following interpolation lemma.

\begin{lemma}[Anisotropic interpolation, see {\cite[Ch. 4, Sect. 18]{BIN}} ]\label{l2.13}
  Let $u\in W_{p,p_0}^{s,s/2}(\Omega^T)$, $s\in\R_+$, $p,p_0\in[1,\infty]$, $\Omega\subset\R^3$. Let $\sigma\in\R_+\cup\{0\}$, and
  $$
    \varkappa={3\over p}+{2\over p_0}-{3\over q}-{2\over q_0}+|\alpha|+2a+\sigma<s.
  $$
  Then $D_x^\alpha\partial_t^au\in W_{q,q_0}^{\sigma,\sigma/2}(\Omega^T)$, $q\ge p$, $q_0\ge p_0$ and there exists $\varepsilon\in(0,1)$ such that
  $$
    \|D_x^\alpha\partial_t^au\|_{W_{q,q_0}^{\sigma,\sigma/2}(\Omega^T)}\le\varepsilon^{s-\varkappa} \|u\|_{W_{p,p_0}^{s,s/2}(\Omega^t)}+c\varepsilon^{-\varkappa}\|u\|_{L_{p,p_0}(\Omega^t)}.
  $$
  We recall from \cite{B} the trace and the inverse trace theorems for Sobolev spaces with a mixed norm.
\end{lemma}
\goodbreak

\begin{lemma}[traces in $W_{p,p_0}^{s,s/2}(\Omega^T)$, see\cite{B}]\label{l2.14}\ \\
  \begin{itemize}
    \item[(i)] Let $u\in W_{p,p_0}^{s,s/2}(\Omega^t)$, $s\in\R_+$, $p,p_0\in(1,\infty)$. Then $u(x,t_0)=u(x,t)|_{t=t_0}$ for $t_0\in[0,T]$ belongs to $B_{p,p_0}^{s-2/p_0}(\Omega)$, and
          $$
            \|u(\cdot,t_0)\|_{B_{p,p_0}^{s-2/p_0}(\Omega)}\le c\|u\|_{W_{p,p_0}^{s,s/2}(\Omega^T)},
          $$
          where $c$ does not depend on $u$.
    \item[(ii)] For given $\bar u\in B_{p,p_0}^{s-2/p_0}(\Omega)$, $s\in\R_+$, $s>2/p_0$, $p,p_0\in(1,\infty)$, there exists a function $u\in W_{p,p_0}^{s,s/2}(\Omega^T)$ such that $u|_{t=0}=\bar u$ and
          $$
            \|u\|_{W_{p,p_0}^{s,s/2}(\Omega^T)}\le c\|\bar u\|_{B_{p,p_0}^{s-2/p_0}(\Omega)},
          $$
          where constant $c$ does not depend on $\bar u$.
  \end{itemize}
\end{lemma}

We need the following imbeddings between Besov spaces

\begin{lemma}[{see \cite[Th. 4.6.1]{T}}]\label{l2.15}
  Let $\Omega\subset\R^n$ be an arbitrary domain.
  \begin{itemize}
    \item[(a)] Let $s\in\R_+$, $\varepsilon>0$, $p\in(1,\infty)$, and $1\le q_1\le q_2\le\infty$. Then
          $$
            B_{p,\infty}^{s+\varepsilon}(\Omega)\subset B_{p,1}^{s+\varepsilon}(\Omega)\subset B_{p,q_2}^s(\Omega)\subset B_{p,q_1}^s(\Omega)\subset B_{p,\infty}^{s-\varepsilon}(\Omega)\subset B_{p,1}^{s-\varepsilon}(\Omega).
          $$
    \item[(b)] Let $\infty>q\ge p>1$, $1\le r\le\infty$, $0\le t\le s<\infty$ and
          $$
            t+{n\over p}-{n\over q}\le s.
          $$
          Then $B_{p,r}^s(\Omega)\subset B_{q,r}^t(\Omega)$.
  \end{itemize}
\end{lemma}

\begin{lemma}[{see \cite[Ch. 4. Th. 18.8]{BIN}}]\label{l2.16}
  Let $1\le\theta_1<\theta_2\le\infty$. Then
  $$
    \|u\|_{B_{p,\theta_2}^l(\Omega)}\le c\|u\|_{B_{p,\theta_1}^l(\Omega)},
  $$
  where $c$ does not depend on $u$.
\end{lemma}

\begin{lemma}[see {\cite[Ch. 4, Th. 18.9]{BIN}}]\label{l2.17}
  Let $l\in\N$ and $\Omega$ satisfy the $l$-horn condition. Then the following imbeddings hold
  $$\eqal{
      &\|u\|_{B_{p,2}^l(\Omega)}\le c\|u\|_{W_p^l(\Omega)}\le c\|u\|_{B_{p,p}^l(\Omega)},\quad &1\le p\le 2,\cr
      &\|u\|_{B_{p,p}^l(\Omega)}\le c\|u\|_{W_p^l(\Omega)}\le c\|u\|_{B_{p,2}^l(\Omega)},\quad &2\le p<\infty,\cr
      &\|u\|_{B_{p,\infty}^l(\Omega)}\le c\|u\|_{W_p^l(\Omega)}\le c\|u\|_{B_{p,1}^l(\Omega)},\quad &1\le p\le\infty.\cr}
  $$
\end{lemma}

Consider the nonstationary Stokes system in $\Omega\subset\R^3$:
$$\eqal{
    &v_t-\nu\Delta v+\nabla p=f,\cr
    &\divv v=0\cr}
$$
with the initial condition $v(0)$.

\begin{lemma}[see \cite{MS}]\label{l2.18}
  Assume that $f\in L_{q,r}(\Omega^T)$, $v(0)\in B_{q,r}^{2-2/r}(\Omega)$, $r,q\in(1,\infty)$. Then there exists a unique solution to the above system such that $v\in W_{q,r}^{2,1}(\Omega^T)$, $\nabla p\in L_{q,r}(\Omega^T)$ with the following estimate
  \begin{equation}
    \|v\|_{W_{q,r}^{2,1}(\Omega^t)}+\|\nabla p\|_{L_{q,r}(\Omega^t)}\le c(\|f\|_{L_{q,r}(\Omega^t)}+\|v(0)\|_{B_{q,r}^{2-2/r}(\Omega)}),
    \label{2.33}
  \end{equation}
  where $t\le T$.
\end{lemma}

\section{Estimates for the modified stream\\ function $\psi_1$}\label{s3}

Recall that $\psi_1$ is a solution to the problem
\begin{equation}\eqal{
  &-\psi_{1,rr}-\psi_{1,zz}-{3\over r}\psi_{1,r}=\Gamma\quad &{\rm in}\ \ \bar\R^3,\cr
  &\psi_1\to 0\quad &{\rm as}\ \ r+|z|\to\infty.\cr}
  \label{3.1}
\end{equation}

\begin{lemma}\label{l3.1}
  Assume that Remark \ref{r2.8} holds. Assume that $\Gamma\in L_2(\bar\R^3)$. Then for sufficiently regular solutions to (\ref{3.1}) the following estimate holds
  \begin{equation}\eqal{
      &\intop_{\bar\R^3}(\psi_{1,rr}^2+\psi_{1,rz}^2+\psi_{1,zz}^2)dx+\intop_{\bar\R^3}{1\over r^2}\psi_{1,r}^2dx\cr
      &\quad+\intop_{\R^1}\psi_{1,z}^2\bigg|_{r=0}dz\le 27|\Gamma|_{2,\bar\R^3}^2.\cr}
    \label{3.2}
  \end{equation}

\end{lemma}

\begin{proof}
  Multiplying (\ref{3.1}) by $\psi_{1,zz}$ and integrating over $\bar\R^3$ yields
  \begin{equation}\eqal{
      &-\intop_{\bar\R^3}\psi_{1,rr}\psi_{1,zz}dx-\intop_{\bar\R^3}\psi_{1,zz}^2dx-3\intop_{\bar\R^3}{1\over r}\psi_{1,r}\psi_{1,zz}dx\cr
      &=\intop_{\bar\R^3}\Gamma\psi_{1,zz}dx.\cr}
    \label{3.3}
  \end{equation}

  Recall that $d x=r d r d z $  so  $\int_{\overline{\mathbb{R}}^3} \frac{1}{r}(\cdot) d x=\int_{\mathbb{R}_{+} \times \mathbb{R}}(\cdot) d r d z $.
  Integrating by parts with respect to $r$ in the first term implies
  $$\eqal{
      &-\intop_{\bar\R^3}(\psi_{1,r}\psi_{1,zz}r)_{,r}drdz+\intop_{\bar\R^3}\psi_{1,r}\psi_{1,zzr}dx+ \intop_{\bar\R^3}\psi_{1,r}\psi_{1,zz}drdz\cr
      &\quad-\intop_{\bar\R^3}\psi_{1,zz}^2dx-3\intop_{\bar\R^3}\psi_{1,r}\psi_{1,zz}drdz= \intop_{\bar\R^3}\Gamma\psi_{1,zz}dx.\cr}
  $$
  Recall that  $d x=r d r d z$. Hence $\int_{\overline{\mathbb{R}}^3} \frac{1}{r} f d x=\int_{\mathbb{R}_{+} \times \mathbb{R}} f d r d z $.
  Using $(\ref{3.1})_2$ yields
  \begin{equation}\eqal{
      &\intop_{\R^1}\psi_{1,r}\psi_{1,zz}r\bigg|_{r=0}dz-\intop_{\bar\R^3}\psi_{1,zr}^2dx- \intop_{\bar\R^3}\psi_{1,zz}^2dx\cr
      &\quad-2\intop_{\bar\R^3}\psi_{1,r}\psi_{1,zz}drdz=\intop_{\bar\R^3}\Gamma\psi_{1,zz}dx.\cr}
    \label{3.4}
  \end{equation}

  Remark \ref{r2.8} implies that $\psi_{1,r}|_{r=0}=0$ and using that solutions to (\ref{3.1}) are so regular that $|\psi_{1,zz}|_{2,\bar\R^3}<\infty$ we obtain that the first term in (\ref{3.4}) vanishes. Multiplying \eqref{3.4} by$-1 $ and integrating by parts with respect to $z$ in the last term $\int \psi_{1, r} \psi_{1, z z} d r d z$  (boundary terms vanish by decay as $|z|\to \infty$), we obtain
  \begin{equation}
    \intop_{\bar\R^3}(\psi_{1,rz}^2+\psi_{1,zz}^2)dx-\intop_{\bar\R^3}\partial_r\psi_{1,z}^2drdz= -\intop_{\bar\R^3}\Gamma\psi_{1,zz}dx.
    \label{3.5}
  \end{equation}
  Integrating by parts  the last term on the l.h.s. of (\ref{3.5}) and using $(\ref{3.1})_2$ we have
  $$
    -\int_{\bar\R^3}\partial_r(\psi_{1,z}^2)\,drdz=\intop_{\R^1}\psi_{1,z}^2\bigg|_{r=0}dz,
  $$
  where the boundary term at $r=\infty$ vanishes by decay. Moreover, by Young's inequality,
  \[
    -\int_{\bar\R^3}\Gamma\psi_{1,zz}\,dx
    \le \frac12\int_{\bar\R^3}\Gamma^2\,dx+\frac12\int_{\bar\R^3}\psi_{1,zz}^2\,dx.
  \]
  Hence  from \eqref{3.5} follows
  \begin{equation}
    {1\over 2}\intop_{\bar\R^3}(\psi_{1,rz}^2+\psi_{1,zz}^2)dx+\intop_{\R^1}\psi_{1,z}^2\bigg|_{r=0}dz\le{1\over 2}\intop_{\bar\R^3}\Gamma^2dx.
    \label{3.6}
  \end{equation}
  We multiply $(\ref{3.1})_1$ by ${1\over r}\psi_{1,r}$ and integrate over $\bar\R^3$ to obtain
  \begin{equation}\eqal{
    &3\intop_{\bar\R^3}\bigg|{1\over r}\psi_{1,r}\bigg|^2dx=-\intop_{\bar\R^3}\psi_{1,rr}{1\over r}\psi_{1,r}dx-\intop_{\bar\R^3}\psi_{1,zz}{1\over r}\psi_{1,r}dx\cr
    &\quad-\intop_{\bar\R^3}\Gamma{1\over r}\psi_{1,r}dx.\cr}
    \label{3.7}
  \end{equation}
  Integrating by parts in $r$ and using Remark~\ref{r2.8} we obtain
  \[
    -\int_{\bar\R^3}\psi_{1,rr}\frac1r\psi_{1,r}\,dx
    = -\frac12\int_{\R}\Big[\psi_{1,r}^2\Big]_{r=0}^{\infty}dz=0.
  \]
  Hence, applying Young's inequality to the remaining terms in \eqref{3.7} yields \eqref{3.8}.

  The first term on the r.h.s. of (\ref{3.7}) equals
  $$%\eqal{&
    -{1\over 2}\intop_{\bar\R^3}\partial_r\psi_{1,r}^2drdz=-{1\over 2}\lim_{R\to\infty}\intop_{\R^1}\psi_{1,r}^2\bigg|_{r=R}dz%\cr&\quad
    +{1\over 2}\intop_{\R^1}\psi_{1,r}^2\bigg|_{r=0}dz,%\cr}
  $$
  where the first term vanishes in view of $(\ref{3.1})_2$ and Remark \ref{r2.8} implies that the second term also vanishes.

  Applying the H\"older and Young inequalities to the last two terms on the r.h.s. of (\ref{3.7}) yields
  \begin{equation}
    2\intop_{\bar\R^3}\bigg|{1\over r}\psi_{1,r}\bigg|^2dx\le{1\over 2}\intop_{\bar\R^3}(\psi_{1,zz}^2+\Gamma^2)dx.
    \label{3.8}
  \end{equation}
  From (\ref{3.8}) we have
  \begin{equation}
    \intop_{\bar\R^3}\bigg|{1\over r}\psi_{1,r}\bigg|^2dx\le{1\over 4}\intop_{\bar\R^3}(\psi_{1,zz}^2+\Gamma^2)dx.
    \label{3.9}
  \end{equation}
  Adding (\ref{3.6}) and (\ref{3.9}) yields
  \begin{equation}
    {1\over 4}\intop_{\bar\R^3}(\psi_{1,rz}^2+\psi_{1,zz}^2)dx+\intop_{\bar\R^3}\bigg|{1\over r}\psi_{1,r}\bigg|^2dx+\intop_{\R^1}\psi_{1,z}^2\bigg|_{r=0}dz\le{3\over 4}\intop_{\bar\R^3}\Gamma^2dx.
    \label{3.10}
  \end{equation}
  Simplifying, we get
  \begin{equation}
    \intop_{\bar\R^3}(\psi_{1,rz}^2+\psi_{1,zz}^2)dx+4\intop_{\bar\R^3}\bigg|{1\over r}\psi_{1,r}\bigg|^2dx+4\intop_{\R^1}\psi_{1,z}^2\bigg|_{r=0}dz\le 3\intop_{\bar\R^3}\Gamma^2dx.
    \label{3.11}
  \end{equation}
  From $(\ref{3.1})_1$ we have
  \begin{equation}\eqal{
    &\intop_{\bar\R^3}\psi_{1,rr}^2dx=\intop_{\bar\R^3}\bigg(\psi_{1,zz}+{3\over r}\psi_{1,r}+\Gamma\bigg)^2dx\cr
    &\le 3\intop_{\bar\R^3}\psi_{1,zz}^2dx+27\intop_{\bar\R^3}\bigg|{1\over r}\psi_{1,r}\bigg|^2dx+3\intop_{\bar\R^3}|\Gamma|^2dx.\cr}
    \label{3.12}
  \end{equation}
  Multiplying \eqref{3.11} by $8$ and adding \eqref{3.12} we obtain
  \[
    8\int \psi_{1,rz}^2 + 11\int \psi_{1,zz}^2 + \int \psi_{1,rr}^2
    +59\int \Big|\frac1r\psi_{1,r}\Big|^2 + 32\int_{\R}\psi_{1,z}^2|_{r=0}
    \le 27\int \Gamma^2 .
  \]
  Dropping some positive terms  yields \eqref{3.13}.

  \begin{equation}\eqal{
    &5\intop_{\bar\R^3}(\psi_{1,rz}^2+\psi_{1,zz}^2)dx+\intop_{\bar\R^3}\psi_{1,rr}^2dx\cr
    &\quad+5\intop_{\bar\R^3}\bigg|{1\over r}\psi_{1,r}\bigg|^2dx+32\intop_{\R^1}\psi_{1,z}^2\bigg|_{r=0}dz\le 27\intop_{\bar\R^3}\Gamma^2dx.\cr}
    \label{3.13}
  \end{equation}
  From (\ref{3.13}) we derive (\ref{3.2}). This concludes the proof.
\end{proof}

\begin{lemma}\label{l3.2}
  Let Remark \ref{r2.8} hold. Let $\Gamma$, $\Gamma_{,z}\in L_2(\bar\R^3)$. Then for sufficiently regular solutions to (\ref{3.1}) the following estimates hold
  \begin{equation}
    \intop_{\bar\R^3}(\psi_{1,zzr}^2+\psi_{1,zzz}^2)dx+ 2\intop_{\R^1}\psi_{1,zz}^2\bigg|_{r=0}dz\le\intop_{\bar\R^3}\Gamma_{,z}^2dx
    \label{3.14}
  \end{equation}
  and
  \begin{equation}
    \intop_{\bar\R^3}(\psi_{1,rrz}^2+\psi_{1,rzz}^2+\psi_{1,zzz}^2)dx+\intop_{\R^1}\psi_{1,zz}^2\bigg|_{r=0}dz\le 2\intop_{\bar\R^3}|\Gamma_{,z}|^2dx.
    \label{3.15}
  \end{equation}
\end{lemma}

\begin{proof}
  First we show (\ref{3.14}). Differentiate $(\ref{3.1})_1$ with respect to  $z$, multiply by $-\psi_{1,zzz}$ and integrate over $\bar\R^3$. Then we obtain
  \begin{equation}\eqal{
      &\intop_{\bar\R^3}\psi_{1,rrz}\psi_{1,zzz}dx+\intop_{\bar\R^3}\psi_{1,zzz}^2dx+3\intop_{\bar\R^3}{1\over r}\psi_{1,rz}\psi_{1,zzz}dx\cr
      &=-\intop_{\bar\R^3}\Gamma_{,z}\psi_{1,zzz}dx.\cr}
    \label{3.16}
  \end{equation}
  Integrating by parts with respect to $z$ in the first term yields
  \begin{equation}\eqal{
      &\intop_{\bar\R^3}\psi_{1,rrz}\psi_{1,zzz}dx=\intop_{\bar\R^3}(\psi_{1,rrz}\psi_{1,zz})_{,z}dx- \intop_{\bar\R^3}\psi_{1,rrzz}\psi_{1,zz}dx,\cr}
    \label{3.17}
  \end{equation}
  where the first term vanishes in view of $(\ref{3.1})_2$. Integrating by parts in the second term gives
  $$\eqal{
      &-\intop_{\bar\R^3}\psi_{1,rrzz}\psi_{1,zz}dx=-\intop_{\bar\R^3}(\psi_{1,rzz}\psi_{,zz}r)_{,r}drdz\cr
      &\quad+\intop_{\bar\R^3}\psi_{1,rzz}^2dx+\intop_{\bar\R^3}\psi_{1,rzz}\psi_{1,zz}drdz,\cr}
  $$
  where the first term on the r.h.s. vanishes because of $(\ref{3.1})_2$ and Remark \ref{r2.8} which implies that $\psi_{1,rzz}|_{r=0}=0$.

  In view of the above considerations, (\ref{3.16}) takes the form
  \begin{equation}\eqal{
      &\intop_{\bar\R^3}(\psi_{1,rzz}^2+\psi_{1,zzz}^2)dx+\intop_{\bar\R^3}\psi_{1,rzz}\psi_{1,zz}drdz+3\intop_{\bar\R^3}\psi_{1,rz}\psi_{1,zzz}drdz\cr
      &=-\intop_{\bar\R^3}\Gamma_{,z}\psi_{1,zzz}dx.\cr}
    \label{3.18}
  \end{equation}
  Integrating by parts with respect to $z$ in the last term on the l.h.s. and using $(\ref{3.1})_2$, we get
  \begin{equation}\eqal{
      &\intop_{\bar\R^3}(\psi_{1,rzz}^2+\psi_{1,zzz}^2)dx-\intop_{\bar\R^3}\partial_r\psi_{1,zz}^2drdz=-\intop_{\bar\R^3}\Gamma_{,z}\psi_{1,zzz}dx.\cr}
    \label{3.19}
  \end{equation}
  In view of $(\ref{3.1})_2$ the last term on the l.h.s. of (\ref{3.19}) equals $\intop_{\R^1}\psi_{1,zz}^2|_{r=0}dz$.

  Next, applying the H\"older and Young inequalities to the r.h.s. of (\ref{3.19}), we derive
  \begin{equation}
    {1\over 2}\intop_{\bar\R^3}(\psi_{1,rzz}^2+\psi_{1,zzz}^2)dx+\intop_{\R^1}\psi_{1,zz}^2\bigg|_{r=0}dz\le{1\over 2}\intop_{\bar\R^3}\Gamma_{,z}^2dx.
    \label{3.20}
  \end{equation}
  The above inequality implies (\ref{3.14}).

  Finally, we show (\ref{3.15}). Differentiate $(\ref{3.1})_1$ with respect to $z$, multiply by $\psi_{1,rrz}$ and integrate over $\bar\R^3$. Then we have
  \begin{equation}\eqal{
      &-\intop_{\bar\R^3}\psi_{1,rrz}^2dx-\intop_{\bar\R^3}\psi_{1,zzz}\psi_{1,rrz}dx-3\intop_{\bar\R^3}{1\over r}\psi_{1,rz}\psi_{1,rrz}dx\cr
      &=\intop_{\bar\R^3}\Gamma_{,z}\psi_{1,rrz}dx.\cr}
    \label{3.21}
  \end{equation}
  Integrating by parts with respect to $z$ in the second term and using $(\ref{3.1})_2$ we obtain that it equals to
  $$
    \intop_{\bar\R^3}\psi_{1,zz}\psi_{1,rrzz}dx.
  $$
  Next integrating by parts with respect to $r$ implies
  \begin{equation}\eqal{
      &\intop_{\bar\R^3}\psi_{1,zz}\psi_{1,rrzz}dx=\intop_{\bar\R^3}(\psi_{1,zz}\psi_{1,rzz}r)_{,r}drdz\cr
      &\quad-\intop_{\bar\R^3}\psi_{1,rzz}^2dx-\intop_{\bar\R^3}\psi_{1,zz}\psi_{1,rzz}drdz.\cr}
    \label{3.22}
  \end{equation}
  Using $(\ref{3.1})_2$ and Remark \ref{r1.2} implying that $\psi_{1,rzz}|_{r=0}=0$, the first term on the r.h.s. of (\ref{3.22}) vanishes.

  Using (\ref{3.22}) in (\ref{3.21}) gives
  \begin{equation}\eqal{
    &\intop_{\bar\R^3}(\psi_{1,rrz}^2+\psi_{1,rzz}^2)dx-{1\over 2}\intop_{\R^1}\psi_{1,zz}^2\bigg|_{r=0}dz=-\intop_{\bar\R^3}\Gamma_{,z}\psi_{1,rrz}dx,\cr}
    \label{3.23}
  \end{equation}
  where the last term on the l.h.s. of (\ref{3.21}) vanishes because $(\ref{3.1})_2$ is used and also Remark \ref{r2.8} implying that $\psi_{1,rz}|_{r=0}=0$.

  Applying the H\"older and Young inequalities to the r.h.s. of (\ref{3.23}) yields
  \begin{equation}
    {1\over 2}\intop_{\bar\R^3}(\psi_{1,rrz}^2+\psi_{1,rzz}^2)dx-{1\over 2}\intop_{\R^1}\psi_{1,zz}^2\bigg|_{r=0}dz\le{1\over 2}|\Gamma_{,z}|_{2,\bar\R^3}^2.
    \label{3.24}
  \end{equation}
  Inequalities (\ref{3.20}) and (\ref{3.24}) imply
  \begin{equation}
    {1\over 2}\intop_{\bar\R^3}(\psi_{1,rrz}^2+\psi_{1,rzz}^2+\psi_{1,zzz}^2)dx+{1\over 2}\intop_{\R^1}\psi_{1,zz}^2\bigg|_{r=0}dz\le|\Gamma_{,z}|_{2,\bar\R^3}^2.
    \label{3.25}
  \end{equation}
  Inequality (\ref{3.25}) implies (\ref{3.15}). This ends the proof.
\end{proof}

\begin{lemma}\label{l3.3}
  Let Remark \ref{r2.8} hold. For sufficiently regular solutions to (\ref{3.1}) the following inequality holds
  \begin{equation}
    \bigg|\frac{1}{r}\psi_{1,rz}\bigg|_{2,\bar\R^3}
    \le |\Gamma_{,z}|_{2,\bar\R^3}.
    \label{3.26}
  \end{equation}
\end{lemma}

\begin{proof}
  Differentiate $(\ref{3.1})_1$ with respect to $z$:
  \begin{equation}
    -\psi_{1,rrz}-\psi_{1,zzz}-{3\over r}\psi_{1,rz}=\Gamma_{,z}.
    \label{3.27}
  \end{equation}
  Taking the $L^2(\bar\R^3)$ norm and using the triangle inequality gives
  \begin{equation}
    3\bigg|{1\over r}\psi_{1,rz}\bigg|_{2,\bar\R^3}\le|\psi_{1,rrz}|_{2,\bar\R^3}+|\psi_{1,zzz}|_{2,\bar\R^3}+|\Gamma_{,z}|_{2,\bar\R^3}.
    \label{3.28}
  \end{equation}
  From inequality (\ref{3.25}) we have
  \[
    |\psi_{1,rrz}|_{2,\bar\R^3}^2+|\psi_{1,zzz}|_{2,\bar\R^3}^2
    \le 2|\Gamma_{,z}|_{2,\bar\R^3}^2.
  \]
  Applying the Cauchy–Schwarz inequality to the sum of the norms yields
  \[
    |\psi_{1,rrz}|_{2,\bar\R^3}+|\psi_{1,zzz}|_{2,\bar\R^3}
    \le\sqrt{2}\,\bigl(|\psi_{1,rrz}|_{2,\bar\R^3}^2
    +|\psi_{1,zzz}|_{2,\bar\R^3}^2\bigr)^{1/2}
    \le 2|\Gamma_{,z}|_{2,\bar\R^3}.
  \]
  Substituting this estimate into the previous inequality we obtain
  \[
    3\bigg|\frac{1}{r}\psi_{1,rz}\bigg|_{2,\bar\R^3}
    \le 2|\Gamma_{,z}|_{2,\bar\R^3}+|\Gamma_{,z}|_{2,\bar\R^3}
    =3|\Gamma_{,z}|_{2,\bar\R^3},
  \]
  which directly implies (\ref{3.26}).
\end{proof}
\subsection{Local estimates for the modified stream function}\label{s3.1}

Let the number $r_0>0$ be given. Let $\zeta=\zeta(r)$ be a smooth cut-off function such that
$$
  \zeta(r)=\left\{\eqal{&1\ &r\le r_0,\cr &0\ &r\ge 2r_0.\cr}
  \right.
$$
Let
\begin{equation}
  \tilde\psi_1=\psi_1\zeta,\quad \tilde\Gamma=\Gamma\zeta.
  \label{3.29}
\end{equation}
Then (\ref{3.1}) takes the form
\begin{equation}\eqal{
  &-\tilde\psi_{1,rr}-\tilde\psi_{1,zz}-{3\over r}\tilde\psi_{1,r}=\tilde\Gamma-2\psi_{1,r}\dot\zeta-\psi_1\ddot\zeta-{3\over r}\psi_1\dot\zeta\cr
  &\tilde\psi_1=0\ \ {\rm for}\ \ r\ge 2r_0\ \ {\rm and}\ \ \tilde\psi_1\to 0\ \ {\rm as}\ \ |z|\to\infty.\cr}
  \label{3.30}
\end{equation}

\begin{lemma}\label{l3.4}
  Assume that $\tilde\Gamma\in L_2(\bar\R^3)$. Let the assumptions of Lemma \ref{l2.1} hold.\\
  Then
  \begin{equation}\eqal{
    &\intop_{\bar\R^3}(\tilde\psi_{1,rr}^2+\tilde\psi_{1,rz}^2+\tilde\psi_{1,zz}^2)dx+\bigg|{1\over r}\tilde\psi_{1,r}\bigg|_{2,\bar\R^3}^2+\intop_\R\tilde\psi_{1,z}^2\bigg|_{r=0}dz\cr
    &\le c(|\tilde\Gamma|_{2,\bar\R^3}^2+D_1^2).\cr}
    \label{3.31}
  \end{equation}
\end{lemma}

\begin{proof}
  Multiplying (\ref{3.30}) by $\tilde\psi_{1,zz}$ and integrating over $\bar\R^3$ implies
  \begin{equation}\eqal{
      &-\intop_{\bar\R^3}\tilde\psi_{1,rr}\tilde\psi_{1,zz}dx-\intop_{\bar\R^3}\tilde\psi_{1,zz}^2dx-3\intop_{\bar\R^3}{1\over r}\tilde\psi_{1,r}\tilde\psi_{1,zz}dx\cr
      &=\intop_{\bar\R^3}\tilde\Gamma\tilde\psi_{1,zz}dx-\intop_{\bar\R^3}(2\psi_{1,r}\dot\zeta+\psi_1\ddot\zeta+{3\over r}\psi_1\dot\zeta)\tilde\psi_{1,zz}dx.\cr}
    \label{3.32}
  \end{equation}
  Integrating by parts in the first term yields
  \begin{equation}\eqal{
      &-\intop_{\bar\R^3}(\tilde\psi_{1,r}\tilde\psi_{1,zz}r)_{,r}drdz+\intop_{\bar\R^3}\tilde\psi_{1,r}\tilde\psi_{1,zzr}dx+ \intop_{\bar\R^3}\tilde\psi_{1,r}\tilde\psi_{1,zz}drdz\cr
      &\quad-\intop_{\bar\R^3}\tilde\psi_{1,zz}^2dx-3\intop_{\bar\R^3}\tilde\psi_{1,r}\tilde\psi_{1,zz}drdz= \intop_{\bar\R^3}\tilde\Gamma\tilde\psi_{1,zz}dx\cr
      &\quad-\intop_{\bar\R^3}\bigg(2\psi_{1,r}\dot\zeta+\psi_1\ddot\zeta+{3\over r}\psi_1\dot\zeta\bigg)\tilde\psi_{1,zz}dx,\cr}
    \label{3.33}
  \end{equation}
  where the boundary term vanishes because $\tilde\psi_{1,r}|_{r=0}=0$ and $\tilde\psi_1\to 0$ as $r\to\infty$.

  Integrating by parts with respect to $z$ implies
  \begin{equation}\eqal{
      &\intop_{\bar\R^3}(\tilde\psi_{1,rz}^2+\tilde\psi_{1,zz}^2)dx-2\intop_{\bar\R^3}\tilde\psi_{1,rz}\tilde\psi_{1,z}drdz\cr
      &=-\intop_{\bar\R^3}\tilde\Gamma\tilde\psi_{1,zz}dx+\intop_{\bar\R^3}\bigg(2\psi_{1,r}\dot\zeta+\psi_1\ddot\zeta+{3\over r}\psi_1\dot\zeta\bigg)\tilde\psi_{1,zz}dx.\cr}
    \label{3.34}
  \end{equation}
  Since $-2\tilde\psi_{1,rz}\tilde\psi_{1,z}=-\partial_r(\tilde\psi_{1,z}^2)$, we have
  \[
    -2\int_{\bar\R^3}\tilde\psi_{1,rz}\tilde\psi_{1,z}\,drdz
    =-\int_{\R}\int_0^\infty \partial_r(\tilde\psi_{1,z}^2)\,dr\,dz
    =\int_{\R}\tilde\psi_{1,z}^2(0,z)\,dz,
  \]
  where the boundary term at $r=\infty$ vanishes by the compact support of $\zeta$.

  Continuing, we have
  \begin{equation}\eqal{
      &\intop_{\bar\R^3}(\tilde\psi_{1,rz}^2+\tilde\psi_{1,zz}^2)dx+\intop_\R\tilde\psi_{1,z}^2\bigg|_{r=0}dz\cr
      &\le c|\tilde\Gamma|_{2,\bar\R^3}^2+c(|\psi_{1,r}|_{2,\bar\R^3\cap\supp\dot\zeta}^2+ |\psi_1|_{2,\bar\R^3\cap\supp\dot\zeta}^2)\cr
      &\le c|\tilde\Gamma|_{2,\bar\R^3}^2+cD_1^2,\cr}
    \label{3.35}
  \end{equation}
  where (\ref{2.14}) is used.

  Multiply (\ref{3.30}) by ${1\over r}\tilde\psi_{1,r}$ and integrate over $\bar\R^3$. Then we obtain
  \begin{equation}\eqal{
    &3\intop_{\bar\R^3}\bigg|{1\over r}\tilde\psi_{1,r}\bigg|^2dx=-\intop_{\bar\R^3}\tilde\psi_{1,rr}{1\over r}\tilde\psi_{1,r}dx-\intop_{\bar\R^3}\tilde\psi_{1,zz}{1\over r}\tilde\psi_{1,r}dx\cr
    &\quad-\intop_{\bar\R^3}\tilde\Gamma{1\over r}\tilde\psi_{1,r}dx+\intop_{\bar\R^3}\bigg(2\psi_{1,r}\dot\zeta+\psi_1\ddot\zeta+{3\over r}\psi_1\dot\zeta\bigg){1\over r}\tilde\psi_{1,r}dx.\cr}
    \label{3.36}
  \end{equation}
  The first term on the r.h.s. of (\ref{3.36}) equals
  $$\eqal{
    &-{1\over 2}\intop_{\bar\R^3}\partial_r\tilde\psi_{1,r}^2drdz=-{1\over 2}\lim_{R\to\infty}\intop_{\R^1}\tilde\psi_{1,r}^2\bigg|_{r=R}dz+{1\over 2}\intop_{\R^1}\tilde\psi_{1,r}^2\bigg|_{r=0}dz=0,\cr}
  $$
  because $\tilde\psi_{1,r}|_{r=R}=0$ for $R$ sufficiently large and $\tilde\psi_{1,r}|_{r=0}=0$ in view of Remark \ref{r2.8}.

  Applying the H\"older and Young inequalities to the last three terms on the r.h.s. of (\ref{3.36}) yields
  \begin{equation}
    \intop_{\bar\R^3}\bigg|{1\over r}\tilde\psi_{1,r}\bigg|^2dx\le c\intop_{\bar\R^3}(\tilde\psi_{1,zz}^2+\tilde\Gamma^2)dx+cD_1^2,
    \label{3.37}
  \end{equation}
  where (\ref{2.14}) is used.

  Adding (\ref{3.35}) and (\ref{3.37}) yields
  \begin{equation}\eqal{
    &\intop_{\bar\R^3}(\tilde\psi_{1,rz}^2+\tilde\psi_{1,zz}^2)dx+\intop_{\bar\R^3}\bigg|{1\over r}\tilde\psi_{1,r}\bigg|^2dx+\intop_\R\tilde\psi_{1,z}^2\bigg|_{r=0}dz\cr
    &\le c(|\tilde\Gamma|_{2,\bar\R^3}^2+D_1^2).\cr}
    \label{3.38}
  \end{equation}
  From (\ref{3.30}) we have
  \begin{equation}\eqal{
      &\intop_{\bar\R^3}\tilde\psi_{1,rr}^2dx\le c\intop_{\bar\R^3}\bigg(\tilde\psi_{1,zz}^2+\bigg|{1\over r}\tilde\psi_{1,r}\bigg|^2+|\tilde\Gamma|^2+\bigg|2\psi_{1,r}\dot\zeta+\psi_1\ddot\zeta+{3\over r}\tilde\psi_1\dot\zeta\bigg|^2\bigg)dx\cr
      &\le c\bigg(|\tilde\psi_{1,zz}|_{2,\bar\R^3}^2+\bigg|{1\over r}\tilde\psi_{1,r}\bigg|_{2,\bar\R^3}^2+|\tilde\Gamma|_{2,\bar\R^3}^2+D_1^2\bigg).\cr}
    \label{3.39}
  \end{equation}
  Inequalities (\ref{3.38}) and (\ref{3.39}) imply (\ref{3.31}). This ends the proof.
\end{proof}

\begin{lemma}\label{l3.5}
  Let Remark \ref{r2.8} and Lemma \ref{l2.6} hold. Let $\tilde\Gamma,\tilde\Gamma_{,z}\in L_2(\bar\R^3)$.\\
  Then for sufficiently regular solutions to (\ref{3.30}) the following estimates hold
  \begin{equation}\eqal{
      &|\tilde\psi_{1,rzz}|_{2,\bar\R_t^3}^2+|\tilde\psi_{1,zzz}|_{2,\bar\R_t^3}^2+\intop_0^t\intop_{\R^1} \tilde\psi_{1,zz}^2\bigg|_{r=0}dzdt'\cr
      &\le c|\tilde\Gamma_{,z}|_{2,\bar\R_t^3}^2+cD_1^2\cr}
    \label{3.40}
  \end{equation}
  and
  \begin{equation}\eqal{
    &|\tilde\psi_{1,rrz}|_{2,\bar\R_t^3}^2+|\tilde\psi_{1,rzz}|_{2,\bar\R_t^3}^2+ |\tilde\psi_{1,zzz}|_{2,\bar\R_t^3}^2+{1\over 2}\intop_0^t\intop_{\R^1}\tilde\psi_{1,zz}^2\bigg|_{r=0}dzdt'\cr
    &\le c|\tilde\Gamma_{,z}|_{2,\bar\R_t^3}^2+cD_1^2.\cr}
    \label{3.41}
  \end{equation}
\end{lemma}

\begin{proof}
  First we show (\ref{3.40}). Differentiate (\ref{3.30}) with respect to $z$, multiply by $-\tilde\psi_{1,zzz}$ and integrate over $\bar\R^3$. Then we obtain
  \begin{equation}\eqal{
      &\intop_{\bar\R^3}\tilde\psi_{1,rrz}\tilde\psi_{1,zzz}dx+\intop_{\bar\R^3}\tilde\psi_{1,zzz}^2dx+ 3\intop_{\bar\R^3}{1\over r}\tilde\psi_{1,rz}\tilde\psi_{1,zzz}dx\cr
      &=-\intop_{\bar\R^3}\tilde\Gamma_{,z}\tilde\psi_{1,zzz}dx+\intop_{\bar\R^3} \bigg(2\psi_{1,rz}\dot\zeta+\psi_{1,z}\ddot\zeta+{3\over r}\psi_{1,z}\dot\zeta\bigg)\tilde\psi_{1,zzz}dx.\cr}
    \label{3.42}
  \end{equation}
  Integrating by parts with respect to $z$ in the first term yields
  \begin{equation}\eqal{
      &\intop_{\bar\R^3}\tilde\psi_{1,rrz}\tilde\psi_{1,zzz}dx=\intop_{\bar\R^3}(\tilde\psi_{1,rrz}\tilde\psi_{1,zz})_{,z}dx-\intop_{\bar\R^3}\tilde\psi_{1,rrzz}\tilde\psi_{1,zz}dx,\cr}
    \label{3.43}
  \end{equation}
  where the first integral vanishes because
  $$
    \tilde\psi_1\to 0\quad {\rm as}\quad |z|\to\infty.
  $$
  Integrating by parts with respect to $r$ in the second term in (\ref{3.43}) gives
  $$\eqal{
      &-\intop_{\bar\R^3}\tilde\psi_{1,rrzz}\tilde\psi_{1,zz}dx=-\intop_{\bar\R^3}(\tilde\psi_{1,rzz}\tilde\psi_{1,zz}r)_{,r}drdz\cr
      &\quad+\intop_{\bar\R^3}\tilde\psi_{1,rzz}^2dx+\intop_{\bar\R^3}\tilde\psi_{1,rzz}\tilde\psi_{1,zz}drdz,\cr}
  $$
  where the first integral vanishes because $\tilde\psi_1$ has a compact support with respect to $r$ and Remark \ref{r2.8} implies that $\tilde\psi_{1,rzz}|_{r=0}=0$.

  In view of the above considerations, (\ref{3.42}) takes the form
  \begin{equation}\eqal{
      &\intop_{\bar\R^3}(\tilde\psi_{1,rzz}^2+\tilde\psi_{1,zzz}^2)dx+\intop_{\bar\R^3}\tilde\psi_{1,rzz}\tilde\psi_{1,zz}drdz\cr
      &\quad+3\intop_{\bar\R^3}\tilde\psi_{1,rz}\tilde\psi_{1,zzz}drdz=-\intop_{\bar\R^3}\tilde\Gamma_{,z}\tilde\psi_{1,zzz}dx\cr
      &\quad+\intop_{\bar\R^3}\bigg(2\psi_{1,rz}\dot\zeta+\psi_{1,z}\ddot\zeta+{3\over r}\psi_{1,z}\dot\zeta\bigg)\tilde\psi_{1,zzz}dx.\cr}
    \label{3.44}
  \end{equation}
  Integrating by parts with respect to $z$ in the last term on the l.h.s. of (\ref{3.44}) and using $(\ref{3.30})_2$ we get
  \begin{equation}\eqal{
      &\intop_{\bar\R^3}(\tilde\psi_{1,rzz}^2+\tilde\psi_{1,zzz}^2)dx-\intop_{\bar\R^3}\partial_r\tilde\psi_{1,zz}^2drdz\cr
      &=-\intop_{\bar\R^3}\tilde\Gamma_{,z}\tilde\psi_{1,zzz}dx+\intop_{\bar\R^3}\bigg(2\psi_{1,rz}\dot\zeta+\psi_{1,z}\ddot\zeta+{3\over r}\psi_{1,z}\dot\zeta\bigg)\tilde\psi_{1,zzz}dx.\cr}
    \label{3.45}
  \end{equation}
  In view of $(\ref{3.30})_2$ the last term on the l.h.s. of (\ref{3.45}) equals $\intop_{\R^1}\tilde\psi_{1,zz}^2|_{r=0}dz$.

  Applying the H\"older and Young inequalities to the r.h.s. terms of (\ref{3.45}) yields
  \begin{equation}\eqal{
      &\intop_{\bar\R^3}(\tilde\psi_{1,rzz}^2+\tilde\psi_{1,zzz}^2)dx+\intop_{\R^1}\tilde\psi_{1,zz}^2\bigg|_{r=0}dz\cr
      &\le c|\tilde\Gamma_{,z}|_{2,\bar\R^3}^2+c(|\psi_{1,rz}|_{2,\bar\R^3}^2+|\psi_{1,z}|_{2,\bar\R^3}^2).\cr}
    \label{3.46}
  \end{equation}
  Integrating the result with respect to time and using (\ref{2.15}) implies (\ref{3.40}).

  Finally, we show (\ref{3.41}). Differentiate $(\ref{3.30})_1$ with respect to $z$, multiply by $\tilde\psi_{1,rrz}$ and integrate over $\bar\R^3$. Then we obtain
  \begin{equation}\eqal{
      &-\intop_{\bar\R^3}\tilde\psi_{1,rrz}^2dx-\intop_{\bar\R^3}\tilde\psi_{1,zzz}\tilde\psi_{1,rrz}dx- 3\intop_{\bar\R^3}{1\over r}\tilde\psi_{1,rz}\tilde\psi_{1,rrz}dx\cr
      &=\intop_{\bar\R^3}\tilde\Gamma_{,z}\tilde\psi_{1,rrz}-\intop_{\bar\R^3}\bigg(2\psi_{1,rz}\dot\zeta+ \psi_{1,z}\ddot\zeta+{3\over r}\psi_{1,z}\dot\zeta\bigg)\tilde\psi_{1,rrz}dx.\cr}
    \label{3.47}
  \end{equation}
  Integrating by parts with respect to $z$ in the second term on the l.h.s. and using $(\ref{3.30})_2$ we derive that it is equal to
  $$
    \intop_{\bar\R^3}\tilde\psi_{1,zz}\tilde\psi_{1,rrzz}dx.
  $$
  Next, integrating by parts with respect to $r$ yields
  \begin{equation}\eqal{
      &\intop_{\bar\R^3}\tilde\psi_{1,zz}\tilde\psi_{1,rrzz}dx=\intop_{\bar\R^3}(\tilde\psi_{1,zz}\tilde\psi_{1,rzz}r)_{,r}drdz\cr
      &\quad-\intop_{\bar\R^3}\tilde\psi_{1,rzz}^2dx-\intop_{\bar\R^3}\tilde\psi_{1,zz}\tilde\psi_{1,rzz}drdz.\cr}
    \label{3.48}
  \end{equation}
  Using $(\ref{3.30})_2$ and Remark \ref{r2.8} the first term on the r.h.s. of (\ref{3.48}) vanishes.

  Exploiting (\ref{3.48}) in (\ref{3.47}) yields
  \begin{equation}\eqal{
      &\intop_{\bar\R^3}(\tilde\psi_{1,rrz}^2+\tilde\psi_{1,rzz}^2)dx+\intop_{\bar\R^3}\tilde\psi_{1,zz}\tilde\psi_{1,rzz}drdz\cr
      &\quad+3\intop_{\bar\R^3}\tilde\psi_{1,rz}\tilde\psi_{1,rrz}drdz=-\intop_{\bar\R^3}\tilde\Gamma_{,z}\tilde\psi_{1,rrz}dx\cr
      &\quad+\intop_{\bar\R^3}\bigg(2\psi_{1,rz}\dot\zeta+\psi_{1,z}\ddot\zeta+{3\over r}\psi_{1,z}\dot\zeta\bigg)\tilde\psi_{1,rrz}dx.\cr}
    \label{3.49}
  \end{equation}
  Using that $\tilde\psi_{1,rz}|_{r=0}=0$ and applying the H\"older and Young inequalities to the r.h.s. terms we obtain
  \begin{equation}\eqal{
    &\intop_{\bar\R^3}(\tilde\psi_{1,rrz}^2+\tilde\psi_{1,rzz}^2)dx-{1\over 2}\intop_{\R^1}\tilde\psi_{1,zz}^2\bigg|_{r=0}dz\cr
    &\le c(|\tilde\Gamma_{,z}|_{2,\bar\R^3}^2+|\psi_{1,rz}|_{2,\bar\R^3\cap\supp\dot\zeta}^2+ |\psi_{1,z}|_{2,\bar\R^3\cap\supp\dot\zeta}^2).\cr}
    \label{3.50}
  \end{equation}
  Integrating (\ref{3.50}) with respect to time, using (\ref{2.15}) and adding to (\ref{3.40}) imply (\ref{3.41}). This ends the proof.
\end{proof}

\begin{lemma}\label{l3.6}
  Let the assumptions of Lemmas \ref{l2.1} and \ref{l2.6} hold. For solutions to (\ref{3.30}) the following inequality
  \begin{equation}
    \bigg|{1\over r}\tilde\psi_{1,rz}\bigg|_{2,\bar\R_t^3}\le c|\tilde\Gamma_{,z}|_{2,\bar\R_t^3}+cD_1.
    \label{3.51}
  \end{equation}
  holds.
\end{lemma}

\begin{proof}
  Differentiating $(\ref{3.30})_1$ with respect to $z$ yields
  \begin{equation}
    {1\over r}\tilde\psi_{1,rz}=-\tilde\psi_{1,rrz}-\tilde\psi_{1,zzz}- \tilde\Gamma_{,z}+(2\psi_{1,rz}\dot\zeta+\psi_{1,z}\ddot\zeta+\psi_{1,z}\dot\zeta).
    \label{3.52}
  \end{equation}
  From (\ref{3.52}) we have
  $$
    \bigg|{1\over r}\tilde\psi_{1,rz}\bigg|_{2,\bar\R_t^3}\le|\tilde\psi_{1,rrz}|_{2,\bar\R_t^3}+ |\tilde\psi_{1,zzz}|_{2,\bar\R_t^3}+|\tilde\Gamma_{,z}|_{2,\bar\R_t^3}+cD_1.
  $$
  Applying (\ref{3.41}) yields (\ref{3.51}). This ends the proof.
\end{proof}

\section{Local considerations}\label{s4}

We introduce the partition of unity $\{\zeta,\vartheta\}$. Let $r_0>0$. Then
\begin{equation}
  \zeta(r)=\left\{\eqal{&1\ &r\le r_0,\cr &0\ &r\ge 2r_0\cr}\right.
  \label{4.1}
\end{equation}
and
\begin{equation}
  \vartheta(r)=\left\{\eqal{&0\ &r\le r_0,\cr &1\ &r\ge 2r_0.\cr}\right.
  \label{4.2}
\end{equation}
Let $u$ be any function. Then $\tilde u=u\zeta$, $\hat u=u\vartheta$. We localize problem (\ref{1.19})--(\ref{1.22}). Multiply (\ref{1.19})--(\ref{1.22}) by $\zeta$. Then we have
\begin{equation}\eqal{
    &\tilde\Phi_t+v\cdot\nabla\tilde\Phi-\nu\bigg(\Delta\tilde\Phi+{2\over r}\partial_r\tilde\Phi\bigg)-(\omega_r\partial_r+\omega_z\partial_z){v_r\over r}\zeta\cr
    &=v_r\dot\zeta\Phi-\nu\bigg(2\Phi_{,r}\dot\zeta+\Phi\ddot\zeta+{2\over r}\Phi\dot\zeta\bigg)+\tilde F'_r\equiv I_1,\cr}
  \label{4.3}
\end{equation}
\begin{equation}\eqal{
  &\tilde\Gamma_t+v\cdot\nabla\tilde\Gamma-\nu\bigg(\Delta\tilde\Gamma+{2\over r}\partial_r\tilde\Gamma\bigg)+2{v_\varphi\over r}\tilde\Phi\cr
  &=v_r\dot\zeta\Gamma-\nu\bigg(2\Gamma_{,r}\dot\zeta+\Gamma\ddot\zeta+{2\over r}\Gamma\dot\zeta\bigg)+\tilde F'_\varphi\equiv I_2,\cr}
  \label{4.4}
\end{equation}
and
\begin{equation}
  \tilde\Gamma|_{t=0}=\tilde\Gamma(0),\quad \tilde\Phi|_{t=0}=\tilde\Phi(0),
  \label{4.5}
\end{equation}
where $\dot\zeta={d\over dr}\zeta$, $\ddot\zeta={d^2\over dr^2}\zeta$.

We introduce the notation
\[
  \widetilde X(t)^2=\|\widetilde\Phi\|_{V(\bar\R_t^3)}^2+\|\widetilde\Gamma\|_{V(\bar\R_t^3)}^2,
\]
\[
  \widehat X(t)^2=\|\widehat\Phi\|_{V(\bar\R_t^3)}^2+\|\widehat\Gamma\|_{V(\bar\R_t^3)}^2,
\]
\[
  \mathcal X(t)^2=\widetilde X(t)^2+\widehat X(t)^2.
\]
By the properties of the partition of unity, we have the equivalence $X(t)\le C\mathcal X(t)$ and $\mathcal X(t)\le C X(t)$. We will slightly abuse notation and write $X$ instead of $\mathcal X$ when deriving global energy bounds up to a constant.

\begin{lemma}\label{l4.1}
  Let $\tilde\Phi$, $\tilde\Gamma$ be solutions to problem (\ref{4.3})--(\ref{4.5}), let $\tilde\Phi(0),\tilde\Gamma(0)\in L_2(\bar\R^3)$ and $\tilde F'_r,\tilde F'_\varphi\in L_2(0,t;L_{6/5}(\bar\R^3))$. Let $D_2$ be defined by (\ref{2.9}), $v_\varphi\in L_\infty(\bar\R_t^3)$ and
  \begin{equation}
    I=\bigg|\intop_{\bar\R_t^3}{v_\varphi\over r}\tilde\Phi\tilde\Gamma dxdt'\bigg|<\infty.
    \label{4.6}
  \end{equation}
  Let $|\Gamma|_{L_2(0,t;L_6(\bar\R^3\cap\supp\dot\zeta))}$ be finite and $\delta_1$, $\mu$ be small. Then
  \begin{equation}
    \widetilde X(t)^2
    \leq
    C_0(\mathrm{data})H_\mu
    \left[
      |I|
      +
      \delta_1X(t)^2
      +
      C_1(\mathrm{data})(1+\delta_1^{-1})
      \right],
    \label{4.7}
  \end{equation}
  where $H_\mu = 1+\|v_\varphi\|_{L^\infty(\bar\R_t^3)}^{2\mu}$.
\end{lemma}

\begin{proof}
  We multiply (\ref{4.3}) by $\tilde\Phi$ and integrate over $\bar\R^3$ to get
  \begin{equation}\eqal{
      &{1\over 2}{d\over dt}|\tilde\Phi|_{2,\bar\R^3}^2+\nu|\nabla\tilde\Phi|_{2,\bar\R^3}^2- \nu\intop_{\bar\R^3}\partial_r\tilde\Phi^2drdz\cr
      &=\intop_{\bar\R^3}(\omega_r\partial_r+\omega_z\partial_z){v_r\over r}\zeta\tilde\Phi dx+\intop_{\bar\R^3}v_r\dot\zeta\Phi\tilde\Phi dx\cr
      &\quad-\intop_{\bar\R^3}\bigg(2\Phi_{,r}\dot\zeta+\Phi\ddot\zeta+{2\over r}\Phi\dot\zeta\bigg)\tilde\Phi dx+\intop_{\bar\R^3}\tilde F'_r\tilde\Phi dx.\cr}
    \label{4.8}
  \end{equation}
  The last term on the l.h.s. of (\ref{4.8} equals
  \begin{equation}
    \nu\intop_{\R^1}\tilde \Phi^2\bigg|_{r=0}dz
    \label{4.9}
  \end{equation}
  because $\tilde\Phi$ has a compact support with respect to $r$.

  Equalities (\ref{1.13}), (\ref{1.18}) and (\ref{2.18}) imply for $r<r_0$ that
  $$
    \Phi={\omega_r\over r}=-{1\over r^2}u_{,z}=-{1\over r}v_{\varphi,z}=-b_{1,z}-b_{2,z}r^2-\dots\not=0.
  $$
  Since term (\ref{4.9}) is positive it can be dropped from (\ref{4.8}).

  Integrating (\ref{4.8}) with respect to time and employing the Poincar\'e inequality yield
  \begin{equation}\eqal{
      &|\tilde\Phi(t)|_{2,\bar\R^3}^2+\nu\|\tilde\Phi\|_{1,2,\bar\R_t^3}^2\le c\intop_{\bar\R_t^3}(\omega_r\partial_r+\omega_z\partial_z){v_r\over r}\zeta\tilde\Phi dx\cr
      &\quad+c\intop_{\R_t^3}v_r\dot\zeta\Phi\tilde\Phi dx-c\intop_{\bar\R_t^3}\bigg(2\Phi_{,r}\dot\zeta+\Phi\ddot\zeta+{2\over r}\Phi\dot\zeta\bigg)\tilde\Phi dx\cr
      &\quad+c\intop_{\bar\R_t^3}\tilde F'_r\tilde\Phi dx+|\tilde\Phi(0)|_{2,\bar\R^3}^2.\cr}
    \label{4.10}
  \end{equation}
  Using (\ref{1.13}), the first integral on the r.h.s. of (\ref{4.10}) equals
  $$
    \intop_{\bar\R_t^3}\bigg[-v_{\varphi,z}\partial_r{v_r\over r}+{1\over r}\partial_r(rv_\varphi)\partial_z{v_r\over r}\bigg]\zeta\tilde\Phi rdrdzdt'\equiv J.
  $$
  Integrating by parts yields
  $$\eqal{
      J&=-\intop_{\bar\R_t^3}\bigg(v_\varphi\partial_r{v_r\over r}\zeta\tilde\Phi\bigg)_{,z}rdrdzdt'+\intop_{\bar\R_t^3}\partial_r\bigg(rv_\varphi\partial_z{v_r\over r}\zeta\tilde\Phi\bigg)drdzdt'\cr
      &\quad+\intop_{\bar\R_t^3}\bigg[v_\varphi\bigg(\partial_r{v_r\over r}\zeta\tilde\Phi\bigg)_{,z}-v_\varphi\bigg(\partial_z{v_r\over r}\zeta\tilde\Phi\bigg)_{,r}\bigg]dxdt'\cr
      &\equiv J_1+J_2+J_3,\cr}
  $$
  where $J_1$ vanishes because $v_\varphi,v_r,\tilde\Phi\to 0$ as $|z|\to\infty$ and
  $$
    J_2=\intop_{\R_t^1}rv_\varphi\partial_z{v_r\over r}\Phi\bigg|_{r=0}dzdt'=0
  $$
  because $v_\varphi\sim b_1r$, ${v_r\over r}\sim a_1$ and $\Phi\sim-b_{1,z}$ near the axis of symmetry (see (\ref{2.17}), (\ref{2.18})).

  Finally,
  $$\eqal{
      J_3&=\intop_{\bar\R_t^3}\bigg[v_\varphi\partial_r{v_r\over r}\zeta\tilde\Phi_{,z}-v_\varphi\partial_z{v_r\over r}\zeta\tilde\Phi_{,r}\bigg]dxdt'\cr
      &\quad-\intop_{\bar\R_t^3}v_\varphi\partial_z{v_r\over r}\dot\zeta\tilde\Phi dxdt'\equiv J_3^1+J_3^2.\cr}
  $$
  In view of (\ref{1.17}), we have
  $$
    J_3^1=-\intop_{\bar\R_t^3}[v_\varphi\psi_{1,rz}\zeta\tilde\Phi_{,z}-v_\varphi\psi_{1,zz}\zeta\tilde\Phi_{,r}]dxdt'.
  $$
  Continuing,
  $$\eqal{
    |J_3^1|&\le|rv_\varphi|_{\infty,\bar\R_t^3}\bigg|{\tilde\psi_{1,zr}\over r}\bigg|_{2,\bar\R_t^3}|\tilde\Phi_{,z}|_{2,\bar\R_t^3}\cr
    &\quad+|rv_\varphi|_{\infty,\bar\R_t^3}^{1-\mu}|v_\varphi|_{\infty,\bar\R_t^3}^\mu\bigg|{\tilde\psi_{1,zz}\over r^{1-\mu}}\bigg|_{2,\bar\R_t^3}|\tilde\Phi_{,r}|_{2,\bar\R_t^3}\equiv J_3^{11}.\cr}
  $$
  Using (\ref{3.51}), the Hardy inequality (\ref{2.13}) and (\ref{2.9}), we derive
  $$\eqal{
      J_3^{11}&\le cD_2(|\tilde\Gamma_{,z}|_{2,\bar\R_t^3}+D_1)|\Phi_{,z}|_{2,\bar\R_t^3}\cr
      &\quad+cD_2^{1-\mu}|v_\varphi|_{\infty,\bar\R_t^3}^\mu|\tilde\psi_{1,zzr}|_{2,\bar\R_t^3}{r_0^\mu\over\mu} |\Phi_{,r}|_{2,\bar\R_t^3}\equiv J_3^{12}.\cr}
  $$
  In view of (\ref{3.41}) we get
  $$
    J_3^{12}\le cD_2\bigg(1+|v_\varphi|_{\infty,\bar\R_t^3}^\mu{r_0^\mu\over D_2^\mu\mu}\bigg)(|\tilde\Gamma_{,z}|_{2,\bar\R_t^3}+D_1)|\nabla\tilde\Phi|_{2,\bar\R_t^3}.
  $$
  Finally, we examine $J_3^2$. From the properties of $\dot\zeta$ we get
  $$
    |J_3^2|\le c\intop_0^t\intop_{\bar\R^3\cap\supp\dot\zeta}|v_\varphi|\,|v_{r,z}|\,|v_{\varphi,z}|dxdt'\le cD_2D_1^2,
  $$
  where (\ref{2.1}) and (\ref{2.9}) were used.

  Summarizing, the first integral on the r.h.s. of (\ref{4.10}) is bounded by
  \begin{equation}
    |J|\le cD_2\bigg(1+|v_\varphi|_{\infty,\bar\R_t^3}^\mu{r_0^\mu\over\mu D_2^\mu}\bigg) (|\tilde\Gamma_{,z}|_{2,\bar\R_t^3}+D_1)|\nabla\tilde\Phi|_{2,\bar\R_t^3}+cD_2D_1^2.
    \label{4.11}
  \end{equation}
  The second integral on the r.h.s. of (\ref{4.10}) equals
  $$\eqal{
      &\intop_{\bar\R_t^3}v_r\dot\zeta{\omega_r^2\over r^2}\zeta dxdt'=\intop_{\bar\R_t^3}v_r\dot\zeta\zeta{1\over r^2}\bigg({1\over r}u_{,z}\bigg)^2dxdt'\cr
      &=-\intop_{\bar\R_t^3}\dot\zeta\zeta{1\over r^4}(v_ru_{,z})_{,z}udxdt'\equiv I_1,\cr}
  $$
  where we used that $v_r,u_{,z}\to 0$ as $|z|\to\infty$.

  Estimating, we get
  \begin{equation}\eqal{
      |I_1|&\le c(1/r_0)D_2(|v_{r,z}|_{2,\bar\R_t^3}|v_{\varphi,z}|_{2,\bar\R_t^3}+|v_r|_{2,\bar\R_t^3}|u_{,zz}|_{2,\bar\R_t^3})\cr
      &\le c(1/r_0)D_2(D_1^2+D_1D_4),\cr}
    \label{4.12}
  \end{equation}
  where (\ref{2.1}), (\ref{2.22}) were used.

  Similarly, the third term on the r.h.s. of (\ref{4.10}) is bounded by
  \begin{equation}
    c(1/r_0)D_1^2.
    \label{4.13}
  \end{equation}
  Using estimate (\ref{4.11})--(\ref{4.13}) in the r.h.s. of (\ref{4.10}) and applying the H\"older and Young inequalities in the last but one term from the r.h.s. of (\ref{4.10}), we derive
  \begin{equation}\eqal{
      \|\tilde\Phi\|_{V(\bar\R_t^3)}^2&\le cD_2\bigg(1+|v_\varphi|_{\infty,\bar\R_t^3}^\mu{r_0^\mu\over\mu D_2^\mu}\bigg)|
      \nabla\tilde\Gamma|_{2,\bar\R_t^3}|\nabla\tilde\Phi|_{2,\bar\R_t^3}\cr
      &\quad+cD_2^2\bigg(1+|v_\varphi|_{\infty,\bar\R_t^3}^{2\mu}{r_0^{2\mu}\over\mu^2D_2^{2\mu}}\bigg)D_1^2\cr
      &\quad+cD_2D_1^2+cD_1D_2D_4+cD_1^2+c|\tilde F'_r|_{6/5,2,\bar\R_t^3}^2+|\tilde\Phi(0)|_{2,\bar\R^3}^2,
      \cr}
    \label{4.14}
  \end{equation}
  where $\mu>0$ is assumed small.

  Multiplying (\ref{4.4}) by $\tilde\Gamma$ and integrating over $\bar\R_t^3$ we get
  \begin{equation}\eqal{
      &\|\tilde\Gamma\|_{V(\bar\R_t^3)}^2-\nu\intop_{\R_t^3}\partial_r\tilde\Gamma^2drdzdt'+2\intop_{\bar\R_t^3}{v_\varphi\over r}\tilde\Phi\tilde\Gamma dxdt'\cr
      &=\intop_{\bar\R_t^3}v_r\dot\zeta\Gamma\tilde\Gamma dxdt'-\nu\intop_{\bar\R_t^3}\bigg(2\Gamma_{,r}\dot\zeta+\Gamma\ddot\zeta+{2\over r}\Gamma\dot\zeta\bigg)\tilde\Gamma dxdt'\cr
      &\quad+\intop_{\bar\R_t^3}\tilde F'_\varphi\tilde\Gamma dxdt'+|\tilde\Gamma(0)|_{2,\bar\R^3}^2.\cr}
    \label{4.15}
  \end{equation}
  The second term on the l.h.s. of (\ref{4.15}) equals
  $$
    \nu\intop_0^t\intop_{\R^1}\tilde\Gamma^2\bigg|_{r=0}dzdt'
  $$
  so it can be dropped.

  The first term on the r.h.s. of (\ref{4.15}) is bounded by
  $$\eqal{
      &\intop_0^t\bigg(\varepsilon_1|\tilde\Gamma|_{6,\bar\R^3}^2+{c\over\varepsilon_1}|v_r|_{2,\bar\R^3}^2 |\Gamma|_{3,\bar\R^3\cap\supp\dot\zeta}^2\bigg)dt'\cr
      &\le\intop_0^t\bigg(\varepsilon_1|\tilde\Gamma|_{6,\bar\R^3}^2+\varepsilon'_2{c\over\varepsilon_1}D_1^2 |\Gamma|_{6,\R^3\cap\supp\dot\zeta}^2+{c\over\varepsilon_1}{C\over\varepsilon'_2}D_1^2 |\Gamma|_{2,\bar\R^3\cap\supp\dot\zeta}^2\bigg)dt'\cr
      &\le\intop_0^t(\varepsilon_1|\tilde\Gamma|_{6,\R^3}^2+\delta_1|\Gamma|_{6,\bar\R^3\cap\supp\dot\zeta}^2) dt'+C(\mathrm{data})\delta_1^{-1},\cr}
  $$
  where we set $\varepsilon_2' = \frac{\delta_1 \varepsilon_1}{c D_1^2}$.
  The second term on the r.h.s. of (\ref{4.15}) is bounded by
  $$
    \phi(1/r_0)D_1^2,
  $$
  where (\ref{1.18}) and (\ref{2.1}) are used.

  Applying the H\"older and Young inequalities to the third term on the r.h.s. of (\ref{4.15}) and using the above estimates we derive from (\ref{4.15}) the inequality
  \begin{equation}\eqal{
      &\|\tilde\Gamma\|_{V(\bar\R_t^3)}^2\le c\bigg|\intop_{\bar\R_t^3}{v_\varphi\over r}\tilde\Phi\tilde\Gamma dxdt'\bigg|+\delta_1\intop_0^t|\Gamma|_{6,\bar\R^3\cap\supp\dot\zeta}^2dt'\cr
      &\quad+C(\mathrm{data})\delta_1^{-1}+\phi(1/r_0)D_1^2+c|\tilde F'_\varphi|_{6/5,2,\bar\R_t^3}^2+|\tilde\Gamma(0)|_{2,\bar\R^3}^2.\cr}
    \label{4.16}
  \end{equation}
  From (\ref{4.14}) it follows
  \begin{equation}\eqal{
      \|\tilde\Phi\|_{V(\bar\R_t^3)}^2&\le cD_2^2\bigg(1+|v_\varphi|_{\infty,\bar\R_t^3}^{2\mu}{r_0^{2\mu}\over\mu^2D_2^{2\mu}}\bigg) |\nabla\tilde\Gamma|_{2,\bar\R_t^3}^2\cr
      &\quad+cD_2^2D_1^2\bigg(1+|v_\varphi|_{\infty,\bar\R_t^3}^{2\mu}{r_0^{2\mu}\over\mu^2D_2^{2\mu}}\bigg)+cD_6^2,\cr}
    \label{4.17}
  \end{equation}
  Adding a suitable multiple of (\ref{4.16}) to (\ref{4.17}) to absorb the $\|\nabla \widetilde\Gamma\|_{2,\bar\R_t^3}^2$ term, and rewriting the constant dependencies on the data, we arrive at the estimate~\eqref{4.7}.
\end{proof}

\begin{remark}\label{rem4.1.data}
  The constants $C_0$ and $C_1$ in~\eqref{4.7} are independent
  of $\delta_1$. They may depend on the fixed parameter $\mu>0$
  and on the data quantities occurring in the estimates above.
  The dependence on $\delta_1$ is displayed explicitly through
  the factor $1+\delta_1^{-1}$.
\end{remark}

\begin{lemma}[Algebraic reduction in the lower-ratio regime]
  \label{l4.lower.ratio}
  Let $s>2$ and let
  $f\in L^2(\bar\R^3)\cap L^s(\bar\R^3)\cap L^\infty(\bar\R^3)$.
  Set
  \[
    L_s(f)=\|f\|_{L^s(\bar\R^3)},
    \qquad
    M(f)=\|f\|_{L^\infty(\bar\R^3)}.
  \]
  If $M(f)>0$ and
  \[
    \frac{L_s(f)}{M(f)}\geq c_0,
  \]
  then
  \[
    L_s(f)
    \leq
    c_0^{-(s-2)/2}\|f\|_{L^2(\bar\R^3)}.
  \]
\end{lemma}

\begin{proof}
  By interpolation between $L^2$ and $L^\infty$,
  \[
    L_s(f)^s
    =
    \int_{\bar\R^3}|f|^{s-2}|f|^2\,dx
    \leq
    M(f)^{s-2}\|f\|_{L^2(\bar\R^3)}^2.
  \]
  Since $M(f)\leq L_s(f)/c_0$, we obtain
  \[
    L_s(f)^s
    \leq
    c_0^{-(s-2)}
    L_s(f)^{s-2}
    \|f\|_{L^2(\bar\R^3)}^2.
  \]
  Division by $L_s(f)^{s-2}$ gives the assertion.
\end{proof}

\begin{corollary}
  \label{c4.good.times}
  Let
  \[
    A_\sharp
    =
    \max\left\{
    A,\,
    c_0^{-(s-2)/2}D_1
    \right\}.
  \]
  Then
  \[
    \|v_\varphi^{\rm loc}(\tau)\|_{L_s(\bar\R^3)} \leq A_\sharp
  \]
  for almost every
  \[
    \tau\in(0,t)\setminus W_{A,c_0}^{\rm loc}.
  \]
\end{corollary}

\begin{proof}
  Recall that the complement of the critical wedge consists of two disjoint regimes: $\mathcal R_1$ and $\mathcal R_2$.
  For $\tau \in \mathcal R_1$, we have $L_s(\tau) \le A \le A_\sharp$ by definition.
  For $\tau \in \mathcal R_2$, we have $L_s(\tau) > A$ and $L_s(\tau)/M(\tau) \ge c_0$. By Lemma~\ref{l4.lower.ratio} with $f = v_\varphi^{\rm loc}(\tau)$, we have
  \[
    L_s(\tau) \le c_0^{-(s-2)/2}\|v_\varphi^{\rm loc}(\tau)\|_{L^2(\bar\R^3)}.
  \]
  Since $\|v_\varphi^{\rm loc}(\tau)\|_{L^2(\bar\R^3)} \le \|v(\tau)\|_{L^2(\bar\R^3)} \le D_1$ by Lemma~\ref{l2.1}, we obtain $L_s(\tau) \le c_0^{-(s-2)/2} D_1 \le A_\sharp$.
  Thus, the bound holds for almost every $\tau \in \mathcal R_1 \cup \mathcal R_2$.
\end{proof}

\begin{lemma}\label{l4.4}
  Let $s>3$ be given. Let $E \subset (0,t)$ be a measurable subset of time, and assume that there exists a constant $A$ such that
  \begin{equation}
    \|v_\varphi^{\rm loc}(\tau)\|_{L_{s}(\bar\R^3)} \le A \qquad \text{for a.e.\ } \tau \in E.
    \label{4.30}
  \end{equation}
  Assume that all parameters from Section~\ref{s2.2} are finite.

  Then the following estimate holds
  \begin{equation}\eqal{
    I_E&:=\int_E |i(\tau)|\,d\tau = \int_E \bigg|\intop_{\bar\R^3}{v_\varphi\over r}\tilde\Phi\tilde\Gamma dx\bigg| d\tau\cr
    &\le\phi({\rm data},A)\cdot\cr
    &\quad\cdot(1+|v_\varphi|_{\infty,\bar\R_t^3}^{2\delta})^{\alpha_0/2}(\tilde X(t)^{2-{1\over 2}(\alpha-{1\over 2})}+\tilde X(t)^{2-(\alpha-{1\over 2})}+1),\cr}
    \label{4.31}
  \end{equation}
  where data replaces all parameters from Section \ref{s2.2} and $\alpha={1\over 2}+\alpha_0$, $\alpha_0={(s-3)(1-b)\over 2s}$, $b\in(0,1)$.
\end{lemma}

\begin{proof}
  Since $\tilde\Phi$ and $\tilde\Gamma$ are supported in $\supp\zeta$ and $v_\varphi^{\rm loc} = v_\varphi$ on this support, we write the integral over $E$ in the form
  $$
    I_E\le \int_E \intop_{\bar\R^3} |v_\varphi^{\rm loc}| r^br^{-{1+b\over 2}}|\tilde\Phi| r^{-{1+b\over 2}}|\tilde\Gamma| dxd\tau.
  $$
  By the H\"older inequality over space,
  $$\eqal{
    |I_E|&\le\int_E|v_\varphi^{\rm loc} r^b|_{{s\over 1-b},\bar\R^3}|r^{-{1+b\over 2}}\tilde\Phi|_{{2s\over s-1+b},\bar\R^3}|r^{-{1+b\over 2}}\tilde\Gamma|_{{2s\over s-1+b},\bar\R^3}d\tau\cr
    &\le D_2^b\int_E|v_\varphi^{\rm loc}|^{1-b}_{{s\over 1-b},\bar\R^3}|r^{-{1+b\over 2}}\tilde\Phi|_{{2s\over s-1+b},\bar\R^3}|r^{-{1+b\over 2}}\tilde\Gamma|_{{2s\over s-1+b},\bar\R^3}d\tau\cr
    &\equiv I_{E,1}.\cr}
  $$
  The case $b=1$ implies that
  $$
    I_{E,1}\le D_2\int_E |r^{-1}\tilde\Phi|_{2,\bar\R^3}|r^{-1}\tilde\Gamma|_{2,\bar\R^3}d\tau\equiv I'_{E,1}.
  $$
  Since $\tilde\Phi|_{r=0}$, $\tilde\Gamma|_{r=0}$ do not vanish, the Hardy inequalities can not be applied to estimate $I'_{E,1}$. Hence, we consider the case $b<1$.  This restriction is essential, since $\tilde\Phi$ and $\tilde\Gamma$ do not vanish on the symmetry axis and therefore the critical Hardy inequality corresponding to $d=1$ fails.
  To apply Lemma \ref{l2.7} we set
  $$
    q={2s\over s-1+b}, \quad \text{and the interpolation parameter } \tilde s={(1+b)s\over s-1+b}.
  $$
  We check the assumptions of Lemma \ref{l2.7}, $\tilde s<2$, $q\in[2,2(3-\tilde s)]$.

  The condition $\tilde s<2$ implies ${s(1+b)\over s-1+b}<2$ so $(s-2)(b-1)<0$ which holds for $s>2$, $b<1$.

  Next, $q\le 6-2\tilde s$ implies $(s-3)(1-b)\ge 0$. This holds for $s\ge 3$.

  Then Lemma \ref{l2.7} yields pointwise in time
  $$
    |r^{-{1+b\over 2}}\tilde\Phi(\tau)|_{{2s\over s-1+b},\bar\R^3}\le c|\tilde\Phi(\tau)|_{2,\bar\R^3}^{{3-\tilde s\over q}-{1\over 2}}|\nabla\tilde\Phi(\tau)|_{2,\bar\R^3}^{{3\over 2}-{3-\tilde s\over q}}\equiv J(\tau).
  $$
  We calculate
  $$
    {3-\tilde s\over q}={3-{s(1+b)\over s-1+b}\over{2s\over s-1+b}}={1\over 2}+{(s-3)(1-b)\over 2s}\equiv\alpha.
  $$
  Hence
  $$
    \alpha-{1\over 2}={(s-3)(1-b)\over 2s}\equiv\alpha_0,\quad
    {3\over 2}-\alpha=1-{(s-3)(1-b)\over 2s}\equiv 1-\alpha_0.
  $$
  Since $\alpha-1/2\not=0$, ${3\over 2}-\alpha\not=1$ we have that $s>3$.

  Thus
  $$
    J(\tau)\le c|\tilde\Phi(\tau)|_{2,\bar\R^3}^{\alpha-1/2}|\nabla\tilde\Phi(\tau)|_{2,\bar\R^3}^{{3\over 2}-\alpha}.
  $$
  Using the above estimate in $I_{E,1}$ yields
  $$
    \begin{aligned}
      I_{E,1}\leq {} & D_2^b\int_E
      |v_\varphi^{\rm loc}(\tau)|_{s,\bar\R^3}^{1-b}
      |\tilde\Phi(\tau)|_{2,\bar\R^3}^{\alpha-1/2}
      |\nabla\tilde\Phi(\tau)|_{2,\bar\R^3}^{3/2-\alpha} \\
                     & \qquad\times
      |\tilde\Gamma(\tau)|_{2,\bar\R^3}^{\alpha-1/2}
      |\nabla\tilde\Gamma(\tau)|_{2,\bar\R^3}^{3/2-\alpha}
      \,d\tau
      \equiv I_{E,2}.
    \end{aligned}
  $$
  By assumption \eqref{4.30}, for almost all $\tau \in E$ we have $|v_\varphi^{\rm loc}(\tau)|_s \le A$.
  Continuing,
  $$\eqal{
      I_{E,2}&\le D_2^b A^{1-b}\int_E (|\tilde\Phi|_{2,\bar\R^3}^{\alpha-1/2} |\tilde\Gamma|_{2,\bar\R^3}^{\alpha-1/2}\cdot\cr
      &\quad\cdot|\nabla\tilde\Phi|_{2,\bar\R^3}^{{3\over 2}-\alpha}|\nabla\tilde\Gamma|_{2,\bar\R^3}^{{3\over 2}-\alpha})d\tau\equiv I_{E,3}.\cr}
  $$
  Since the integrand is non-negative, we can enlarge the domain of integration from $E$ to $(0,t)$ and apply the H\"older inequality over time
  $$\eqal{
      I_{E,3}&\le D_2^b A^{1-b}\bigg(\int_0^t |\tilde\Phi|_{2,\bar\R^3}^{(\alpha-1/2)\lambda_1}d\tau\bigg)^{1/\lambda_1}\bigg(\int_0^t |\tilde\Gamma|_{2,\bar\R^3}^{(\alpha-1/2)\lambda_2}d\tau\bigg)^{1/\lambda_2}\cdot\cr
      &\quad\cdot\bigg(\int_0^t|\nabla\tilde\Phi|_{2,\bar\R^3}^{({3\over 2}-\alpha)\lambda_3}d\tau\bigg)^{1/\lambda_3}\bigg(\int_0^t|\nabla\tilde\Gamma|_{2,\bar\R^3}^{({3\over 2}-\alpha)\lambda_4}d\tau\bigg)^{1/\lambda_4}\equiv I_4,\cr}
  $$
  where
  $$
    1/\lambda_1+1/\lambda_2+1/\lambda_3+1/\lambda_4=1.
  $$
  Assuming $\big(\alpha-{1\over 2}\big)\lambda_1=2$, $\big(\alpha-{1\over 2}\big)\lambda_2=2$, $\big({3\over 2}-\alpha\big)\lambda_3=2$, $\big({3\over 2}-\alpha\big)\lambda_4=2$ we derive the restriction
  $$
    {\alpha-{1\over 2}\over 2}+{\alpha-{1\over 2}\over 2}+{{3\over 2}-\alpha\over 2}+{{3\over 2}-\alpha\over 2}=1\quad {\rm so}\quad \alpha-{1\over 2}+{3\over 2}-\alpha=1.
  $$
  Hence,
  \begin{equation}
    \begin{aligned}
      I_E\le I_4\le D_2^b A^{1-b} |\tilde\Phi|_{2,\bar\R_t^3}^{\alpha-1/2}|\tilde\Gamma|_{2,\bar\R_t^3}^{\alpha-1/2} |\nabla\tilde\Phi|_{2,\bar\R_t^3}^{{3\over 2}-\alpha}|\nabla\tilde\Gamma|_{2,\bar\R_t^3}^{{3\over 2}-\alpha} \\\equiv I_5.
    \end{aligned}
    \label{4.32}
  \end{equation}
  We can express (\ref{5.5}) in the form
  \begin{equation}
    |\tilde\Phi|_{2,\bar\R_t^3}^2\le\phi({\rm data})(1+|v_\varphi|_{\infty,\bar\R_t^3}^{2\delta})(\tilde X(t)+1),
    \label{4.33}
  \end{equation}
  where we set that $\delta=\varepsilon_0$.

  Using (\ref{4.33}) in $I_5$ yields
  \begin{equation}
    \begin{aligned}
      I_5 & \le \phi({\rm data}, A)
      (1+|v_\varphi|_{\infty,\bar\R_t^3}^{2\delta})^{\alpha_0/2}\cdot
      (\tilde X(t)^{{1\over 2}(\alpha-1/2)}+1)\tilde X(t)^{{5\over 2}-\alpha} \\
          & =\phi({\rm data}, A)
      (1+|v_\varphi|_{\infty,\bar\R_t^3}^{2\delta})^{\alpha_0/2}\cdot         \\
          & \quad\cdot(\tilde X(t)^{2-{1\over 2}(\alpha-{1\over 2})}+
      \tilde X(t)^{2-(\alpha-1/2)}),
    \end{aligned}
    \label{4.34}
  \end{equation}
  where $\alpha=1/2+\alpha_0$.
  Using (\ref{4.30}) in (\ref{4.34}) yields (\ref{4.31}). This concludes the proof.
\end{proof}
Corollary~\ref{c4.good.times} and Lemma~\ref{l4.4}
show that the nonlinear interaction is closable on
$(0,t)\setminus W_{A,c_0}^{\rm loc}$.
No global-in-time lower bound for the ratio of the
$L^s$ and $L^\infty$ norms is required.

Now, we consider solutions to problem (\ref{1.19})--(\ref{1.22}) in a neighborhood located at a positive distance from the axis of symmetry. Multiply (\ref{1.19})--(\ref{1.22}) by $\vartheta$. Then we get
\begin{equation}\eqal{
    &\hat\Phi_{,t}+v\cdot\nabla\hat\Phi-\nu\bigg(\Delta\hat\Phi+{2\over r}\partial_r\hat\Phi\bigg)= (\omega_r\partial_r+\omega_z\partial_z){v_r\over r}\vartheta\cr
    &\quad+v_r\dot\vartheta\Phi-\nu\bigg(2\Phi_{,r}\dot\vartheta+\Phi\ddot\vartheta+{2\over r}\Phi\dot\vartheta\bigg)+\hat F'_r,\cr}
  \label{4.35}
\end{equation}
\begin{equation}\eqal{
  &\hat\Gamma_{,t}+v\cdot\nabla\hat\Gamma-\nu\bigg(\Delta\hat\Gamma+{2\over r}\partial_r\hat\Gamma\bigg)=-2{v_r\over r}\hat\Phi\cr
  &\quad+v_r\dot\vartheta\Gamma-\nu\bigg(2\Gamma_{,r}\dot\vartheta+\Gamma\ddot\vartheta+{2\over r}\Gamma\dot\vartheta\bigg)+\hat F'_\varphi,\cr}
  \label{4.36}
\end{equation}
\begin{equation}
  \hat\Gamma|_{t=0}=\hat\Gamma(0),\quad \hat\Phi|_{t=0}=\hat\Phi(0),
  \label{4.37}
\end{equation}
where $\dot\vartheta={d\over dr}\vartheta$, $\ddot\vartheta={d^2\over dr^2}\vartheta$.

\begin{lemma}\label{l4.5}
  Let $\hat\Phi$ be a solution to problem (\ref{4.35}), (\ref{4.37}). Then
  \begin{equation}\eqal{
      &\|\hat\Phi\|_{V(\bar\R_t^3)}^2\le c[D_1^2D_2^2+D_1^2D_{10}^2(3)+D_1^2+D_{10}^2(6/5)\cr
          &\quad+|\hat F'_r|_{6/5,2,\bar\R_t^3}^2+|\hat\Phi(0)|_{2,\bar\R^3}^2].\cr}
    \label{4.38}
  \end{equation}
\end{lemma}

\begin{proof}
  Multiplying \eqref{4.35} by $\hat\Phi$, integrating over $\bar\R^3$ and then over $(0,t)$, we obtain
  \begin{equation}\eqal{
      &|\hat\Phi|_{2,\bar\R^3}^2+\nu|\nabla\hat\Phi|_{2,\bar\R_t^3}^2=\intop_{\bar\R_t^3} (\omega_r\partial_r+\omega_z\partial_z){v_r\over r}\vartheta\hat\Phi dxdt'\cr
      &\quad+\intop_{\bar\R_t^3}v_r\dot\vartheta\Phi\hat\Phi dxdt'-\nu\intop_{\bar\R_t^3}\bigg(2\Phi_{,r}\dot\vartheta+\Phi\ddot\vartheta+{2\over r}\Phi\dot\vartheta\bigg)\hat\Phi dxdt'\cr
      &\quad+\intop_{\bar\R_t^3}\hat F'_r\hat\Phi dxdt'+|\hat\Phi(0)|_{2,\bar\R^3}^2.\cr}
    \label{4.39}
  \end{equation}
  Now, we estimate the particular terms on the r.h.s. of (\ref{4.39}). Using (\ref{1.13}) the first term equals
  $$\eqal{
    &\intop_{\bar\R_t^3}\bigg(-{1\over r}u_{,z}\partial_r+{1\over r}u_{,r}\partial_z\bigg){v_r\over r}\vartheta\hat\Phi dxdt'\cr
    &=\intop_{\bar\R_t^3}\bigg[{1\over r}u\partial_z\bigg(\partial_r{v_r\over r}\vartheta\hat\Phi\bigg)-{1\over r}u\partial_r\bigg(\partial_z{v_r\over r}\vartheta\hat\Phi\bigg)\bigg]dxdt'\cr
    &=\intop_{\bar\R_t^3}\bigg[{u\over r}\partial_r{v_r\over r}\vartheta\hat\Phi_{,z}-{u\over r}\partial_z{v_r\over r}\dot\vartheta\hat\Phi-{u\over r}\partial_z{v_r\over r}\vartheta\hat\Phi_{,r}\bigg]dxdt',\cr}
  $$
  where the first and last integrals are bounded by
  $$\eqal{
      &\varepsilon|\nabla\hat\Phi|_{2,\bar\R_t^3}^2+c(1/\varepsilon)\bigg(|v_{r,r}|_{2,\bar\R_t^3}^2+ |v_{r,z}|_{2,\bar\R_t^3}^2+\bigg|{v_r\over r}\bigg|_{2,\bar\R_t^3}^2\bigg)|u|_{\infty,\bar\R_t^3}^2\cr
      &\le\varepsilon|\nabla\hat\Phi|_{2,\bar\R_t^3}+c(1/\varepsilon)D_1^2D_2^2.\cr}
  $$
  The middle term is bounded by
  $$
    \varepsilon|\hat\Phi|_{{10\over 3},\bar\R_t^3}^2+c(1/\varepsilon)|v_{r,z}|_{2,\bar\R_t^3}^2|u|_{5,\bar\R_t^3}^2,
  $$
  where the second expression is bounded by
  $$
    cD_1^2D_{10}^2(3),
  $$
  where (\ref{2.28}) was used.

  The second integral on the r.h.s. of (\ref{4.39}) equals
  $$\eqal{
      &\intop_{\bar\R_t^3}v_r\dot\vartheta\bigg(-{u_{,z}\over r^2}\bigg)\hat\Phi dxdt'=\intop_{\bar\R_t^3}v_{r,z}\bigg({1\over r^2}\dot\vartheta\bigg)u\hat\Phi dxdt\cr
      &\quad+\intop_{\bar\R_t^3}v_r\bigg({1\over r^2}\dot\vartheta\bigg)u\hat\Phi_{,z}dxdt'\equiv I_1+I_2,\cr}
  $$
  where
  $$\eqal{
      |I_1|&\le\varepsilon|\hat\Phi|_{10/3,\bar\R_t^3}^2+c(1/\varepsilon)|v_{r,z}|_{2,\bar\R_t^3}^2 |u|_{5,\bar\R_t^3}^2\cr
      &\le\varepsilon\|\Phi\|_{V(\bar\R_t^3)}^2+c(1/\varepsilon)D_1^2D_{10}^2(3)\cr}
  $$
  and
  $$
    |I_2|\le\varepsilon|\hat\Phi_{,z}|_{2,\bar\R_t^3}^2+c(1/\varepsilon)|v_r|_{10/3,\bar\R_t^3}^2 |u|_{5,\bar\R_t^3}^2,
  $$
  where the second expression is bounded by
  $$
    cD_1^2D_{10}^2(3).
  $$
  Since $\Phi=-{1\over r^2}u_{,z}$ the third term on the r.h.s. of (\ref{4.39}) equals
  $$\eqal{
      &-\intop_{\bar\R_t^3}\bigg[2\bigg({1\over r^2}u_{,z}\bigg)_{,r}\dot\vartheta+{1\over r^2}u_{,z}\ddot\vartheta+{2\over r}{1\over r^2}u_{,z}\dot\vartheta\bigg]\hat\Phi dxdt'\cr
      &=\intop_{\bar\R_t^3}\bigg[2\bigg({1\over r^2}u\bigg)_{,r}\dot\vartheta+{1\over r^2}u\ddot\vartheta+{2\over r^3}u\dot\vartheta\bigg]\hat\Phi_{,z}dxdt'\equiv J_1.\cr}
  $$
  Estimating and using properties of function $\vartheta$ we have
  $$\eqal{
      |J_1|&\le\varepsilon|\hat\Phi_{,z}|_{2,\bar\R_t^3}^2+ c(1/\varepsilon,1/r_0)(|u_{,r}|_{2,\bar\R_r^3}^2+|u|_{2,\bar\R_t^3}^2)\cr
      &\le\varepsilon|\hat\Phi_{,z}|_{2,\bar\R_t^3}^2+c(1/\varepsilon,1/r_0)\bigg(D_1^2+D_{10}^2\bigg({6\over 5}\bigg)\bigg).\cr
    }
  $$
  Finally, the last but one term on the r.h.s. of (\ref{4.39}) is bounded by
  $$\eqal{
      &\intop_0^t|\hat F'_r|_{6/5,\bar\R^3}|\hat\Phi|_{6,\bar\R^3}dt'\cr
      &\le\varepsilon\intop_0^t|\hat\Phi|_{6,\bar\R^3}^2dt'+c(1/\varepsilon)\intop_0^t|\hat F'_r|_{6/5,\bar\R^3}^2dt'\cr
      &\le\varepsilon|\nabla\hat\Phi|_{2,\bar\R^3}^2+c(1/\varepsilon)|\hat F'_r|_{6/5,2,\bar\R_t^3}^2.\cr}
  $$
  Using the above estimates in (\ref{4.39}) and assuming that $\varepsilon$ is sufficiently small, we obtain (\ref{4.38}). This concludes the proof.
\end{proof}
Lemma~\ref{l4.5} provides the required a priori estimate for the localized problem away from the symmetry axis, thereby completing the local analysis of the vorticity system.

The critical coupling term $(v_\varphi/r)\Phi$ appears only in the equation for $\Gamma$ and, due to Liu-Wang expansion \eqref{1.63} is not singular at the symmetry axis.
Since $\vartheta$ is supported away from $r=0$ and Lemma~\ref{l4.5} concerns the equation for $\hat\Phi$, no term of type $I$ arises in this estimate.

\begin{lemma}\label{l4.6}
  Let $\hat\Gamma$ be a solution to (\ref{4.36}), (\ref{4.37}). Then
  \begin{equation}\eqal{
      \|\hat\Gamma\|_{V(\bar\R_t^3)}^2&\le\delta_2D_1^2\intop_0^t|\Gamma|_{6,\bar\R^3}^2dt'+ c(D_2D_1^2\cr
      &\quad+c(1/\delta_2)D_1^2+|\hat F'_\varphi|_{6/5,2,\bar\R_t^3}^2+|\hat\Gamma(0)|_{2,\bar\R^3}^2).\cr}
    \label{4.40}
  \end{equation}
\end{lemma}

\begin{proof}
  Multiplying (\ref{4.36}) by $\hat\Gamma$ and integrating over $\bar\R_t^3$, we obtain
  \begin{equation}\eqal{
      \|\hat\Gamma\|_{V(\bar\R_t^3)}^2&=-2\intop_{\bar\R_t^3}{v_r\over r}\hat\Phi\hat\Gamma dxdt'\cr
      &\quad+\intop_{\bar\R_t^3}v_r\dot\vartheta\Gamma\hat\Gamma dxdt'-\nu\intop_{\bar\R_t^3}\bigg(2\Gamma_{,r}\dot\vartheta+\Gamma\ddot\vartheta+{2\over r}\Gamma\dot\vartheta\bigg)\hat\Gamma dxdt'\cr
      &\quad+\intop_{\bar\R_t^3}\hat F'_\varphi\hat\Gamma dxdt'+|\hat\Gamma(0)|_{2,\bar\R^3}^2.\cr}
    \label{4.41}
  \end{equation}
  Using (\ref{1.13}) in the first term on the r.h.s. of (\ref{4.41}), we get
  $$
    2\intop_{\bar\R_t^3}u{1\over r^3}\vartheta^2u_{,z}(v_{r,z}-v_{z,r})dxdt'\equiv I_1.
  $$
  Hence,
  $$
    |I_1|\le cD_2\bigg|{1\over r}u_{,z}\bigg|_{2,\bar\R_t^3}|\nabla v|_{2,\bar\R_t^3}\le cD_2D_1^2.
  $$
  By the H\"older inequality, the second term on the r.h.s. of (\ref{4.41}) is bounded by
  $$\eqal{
      &c\intop_0^t|v_r|_{2,\bar\R^3}|\dot\vartheta\Gamma|_{3,\bar\R^3}|\hat\Gamma|_{6,\bar\R^3}dt'\cr
      &\le\varepsilon_1\intop_0^t|\hat\Gamma|_{6,\bar\R^3}^2dt'+c(1/\varepsilon_1)D_1^2\intop_0^t |\dot\vartheta\Gamma|_{3,\bar\R^3}^2dt'\equiv I_2.\cr}
  $$
  Using the interpolation
  $$
    |\dot\vartheta\Gamma|_{3,\bar\R^3}\le |\dot\vartheta\Gamma|_{6,\bar\R^3}^{1/2}|\dot\vartheta\Gamma|_{2,\bar\R^3}^{1/2}
  $$
  in $I_2$, we obtain
  $$\eqal{
      I_2&\le\varepsilon_1\intop_0^t|\hat\Gamma|_{6,\bar\R^3}^2dt'+c(1/\varepsilon_1)D_1^2 \bigg(\varepsilon_2\intop_0^t|\Gamma|_{6,\bar\R^3}^2dt'+ c(1/\varepsilon_2)\intop_0^t|\dot\vartheta\Gamma|_{2,\bar\R^3}^2dt'\bigg),\cr}
  $$
  where
  $$
    \intop_0^t|\dot\vartheta\Gamma|_{2,\bar\R^3}^2dt'\le cD_1^2.
  $$
  Next, the third term on the r.h.s. of (\ref{4.41}) equals
  $$
    -2\nu\intop_{\bar\R_t^3}\Gamma_{,r}\Gamma\dot\vartheta\vartheta dxdt'-\nu\intop_{\bar\R_t^3}\Gamma^2\bigg(\ddot\vartheta+{2\over r}\dot\vartheta\bigg)\vartheta dxdt'\equiv I_3+I_4,
  $$
  where
  $$
    I_3=\nu\intop_{\bar\R_t^3}\Gamma^2(\dot\vartheta\vartheta r)_{,r}drdzdt'.
  $$
  Hence
  $$
    |I_3|\le cD_1^2.
  $$
  Finally
  $$
    |I_4|\le cD_1^2.
  $$
  By the H\"older and Young inequalities the last but one term from the r.h.s. of (\ref{4.41}) is bounded by
  $$
    \varepsilon_3|\hat\Gamma|_{6,2,\bar\R_t^3}^2+c(1/\varepsilon_3)|\hat F'_\varphi|_{6/5,2,\bar\R_t^3}^2.
  $$
  Using the above estimates with sufficiently small $\varepsilon_1$, $\varepsilon_3$ in the r.h.s. of (\ref{4.41}), we obtain (\ref{4.40}). This concludes the proof.
\end{proof}

\begin{remark}\label{r4.7}
  Combining estimates \eqref{4.38} and \eqref{4.40} we obtain
  \begin{equation}\eqal{
      \|\hat\Phi\|_{V(\bar\R_t^3)}^2+\|\hat\Gamma\|_{V(\bar\R_t^3)}^2
      &\le \delta_2 D_1^2 \intop_0^t |\Gamma|_{6,\bar\R^3}^2\,dt' + D_{15}^2,
    }
    \label{4.42}
  \end{equation}
  where
  \begin{equation*}\eqal{
      D_{15}^2
      &:= c\Big[D_1^2D_2^2 + D_1^2D_{10}^2(3) + \frac{1}{\delta_2}D_1^2
        + D_{10}^2(6/5) + D_2D_1^2\Big]\cr
      &\quad + c\Big(|\hat F'_r|_{6/5,2,\bar\R_t^3}^2+|\hat F'_\varphi|_{6/5,2,\bar\R_t^3}^2\Big)
      + c\Big(|\hat\Phi(0)|_{2,\bar\R^3}^2+|\hat\Gamma(0)|_{2,\bar\R^3}^2\Big).
    }
    \label{4.42a}
  \end{equation*}
\end{remark}

\begin{remark}\label{r4.8}
  The local estimates derived in this section for the regions near and away from the symmetry axis are precisely the ingredients used in the proof of Theorem~\ref{t1.1} to obtain the conditional estimate~\eqref{1.34}.
\end{remark}

\section{Local estimates for $\omega_r$ and $\omega_z$}\label{s5}

To localize equations $(\ref{1.9})_{1,3}$ to a neighborhood of the axis of symmetry, we use the cut-off function $\zeta$ (see (\ref{4.1})). Let
\begin{equation}
  \tilde\omega_r=\omega_r\zeta,\quad \tilde\omega_z=\omega_z\zeta.
  \label{5.1}
\end{equation}
From (\ref{1.9}), we have
\begin{equation}\eqal{
    &\tilde\omega_{r,t}+v\cdot\nabla\tilde\omega_r-\nu\Delta\tilde\omega_r+\nu{\tilde\omega_r\over r^2}=\tilde\omega_rv_{r,r}+\tilde\omega_zv_{r,z}\cr
    &\quad+v_r\dot\zeta\omega_r-\nu(2\omega_{r,r}\dot\zeta+\omega_r\ddot\zeta)+\tilde F_r\cr
    &\equiv\tilde\omega_rv_{r,r}+\tilde\omega_zv_{r,z}+L_1+\tilde F_r,\cr}
  \label{5.2}
\end{equation}
\begin{equation}\eqal{
    &\tilde\omega_{z,t}+v\cdot\nabla\tilde\omega_z-\nu\Delta\tilde\omega_z=\tilde\omega_rv_{z,r}+\tilde\omega_zv_{z,z}\cr
    &\quad+v_r\dot\zeta\omega_z-\nu(2\omega_{z,r}\dot\zeta+\omega_z\ddot\zeta)+\tilde F_z\cr
    &\equiv\tilde\omega_rv_{z,r}+\tilde\omega_zv_{z,z}+L_2+\tilde F_z,\cr}
  \label{5.3}
\end{equation}
\begin{equation}
  \tilde \omega_r|_{t=0}=\tilde\omega_r(0),\quad \tilde\omega_z|_{t=0}=\tilde\omega_z(0).
  \label{5.4}
\end{equation}

\begin{lemma}\label{l5.1}
  Let $D_{11}=(D_1+D_4+D_5)D_2$, $D_{12}=(D_1+D_4)D_1^{1-\varepsilon_0}{r_0^{\varepsilon_0}\over\varepsilon_0}$, $D_9=D_1D_2^{2(1-\varepsilon_0)}{r_0^{2\varepsilon_0}\over\varepsilon_0^2}$, $D_8=D_1D_{11}+D_1(D_1+D_4+D_5)(D_2+1)+|\tilde F_r|_{6/5,2,\bar\R_t^3}^2+|\tilde F_z|_{6/5,2,\bar\R_t^3}^2+|\tilde\omega_r(0)|_{2,\bar\R^3}^2+|\tilde\omega_z(0)|_{2,\bar\R^3}^2$.\\
  Then
  \begin{equation}\eqal{
    &\|\tilde\omega_r\|_{V(\bar\R_t^3)}^2+\|\tilde\omega_z\|_{V(\bar\R_t^3)}^2+\nu\bigg|{\tilde\omega_r\over r}\bigg|_{2,\bar\R_t^3}^2\cr
    &\le c(D_{11}+D_{12}|v_\varphi|_{\infty,\bar\R_t^3}^{\varepsilon_0}+ D_9|v_\varphi|_{\infty,\bar\R_t^3}^{2\varepsilon_0})|\nabla\tilde\Gamma|_{2,\bar\R_t^3}\cr
    &\quad+(D_{12}D_1|v_\varphi|_{\infty,\bar\R_t^3}^{\varepsilon_0}+ D_9D_1|v_\varphi|_{\infty,\bar\R_t^3}^{2\varepsilon_0})+cD_8\cr
    &\le\phi({\rm data})(1+|v_\varphi|_{\infty,\bar\R_t^3}^{2\varepsilon_0})|\nabla\tilde\Gamma|_{2,\bar\R_t^3}+\phi({\rm data})\cdot\cr
    &\quad\cdot(1+|v_\varphi|_{\infty,\bar\R_t^3}^{2\varepsilon_0}).\cr}
    \label{5.5}
  \end{equation}
\end{lemma}

\begin{proof}
  Multiplying (\ref{5.2}) by $\tilde\omega_r$, (\ref{5.3}) by $\tilde\omega_z$, adding the results and integrating over $\bar\R_t^3$, we get
  \begin{equation}\eqal{
    &{1\over 2}(|\tilde\omega_r(t)|_{2,\bar\R^3}^2+|\tilde\omega_z(t)|_{2,\bar\R^3}^2)+ \nu(|\nabla\tilde\omega_r|_{2,\bar\R_t^3}^2+|\nabla\tilde\omega_z|_{2,\bar\R_t^3}^2)\cr
    &\quad+\nu\bigg|{\tilde\omega_r\over r}\bigg|_{2,\bar\R_t^3}^2=\intop_{\bar\R_t^3}[v_{r,r}\tilde\omega_r^2+v_{z,z}\tilde\omega_z^2+ (v_{r,z}+v_{z,r})\tilde\omega_r\tilde\omega_z]dxdt'\cr
    &\quad+\intop_{\bar\R_t^3}L_1\tilde\omega_rdxdt'+\intop_{\bar\R_t^3}L_2\tilde\omega_zdxdt'+\intop_{\bar\R_t^3}\tilde F_r\tilde\omega_rdxdt'\cr
    &\quad+\intop_{\bar\R_t^3}\tilde F_z\tilde\omega_zdxdt'+{1\over 2}(|\tilde\omega_r(0)|_{2,\bar\R^3}^2+|\tilde\omega_z(0)|_{2,\bar\R^3}^2).\cr}
    \label{5.6}
  \end{equation}
  To estimate the first term on the r.h.s. of (\ref{5.6}) we recall from (\ref{1.13}) the relations
  \begin{equation}
    \omega_r=-{1\over r}u_{,z},\quad \omega_z={1\over r}u_{,r}
    \label{5.7}
  \end{equation}
  and from (\ref{1.15}) the following
  \begin{equation}\eqal{
    &v_{r,r}=-\psi_{,zr},\quad &v_{z,z}=\psi_{,rz}+{\psi_{,z}\over r},\cr
    &v_{r,z}=-\psi_{,zz},\quad &v_{z,r}=\psi_{,rr}+{1\over r}\psi_{,r}-{\psi\over r^2}.\cr}
    \label{5.8}
  \end{equation}
  Using (\ref{5.7}) and (\ref{5.8}) the first term on the r.h.s. of (\ref{5.6}) equals
  \begin{equation}\eqal{
      J&=\intop_{\bar\R_t^3}\bigg[-\psi_{,zr}\bigg({1\over r}u_{,z}\bigg)^2\zeta^2+\bigg(\psi_{,rz}+{\psi_{,z}\over r}\bigg)\bigg({1\over r}u_{,r}\bigg)^2\zeta^2\cr
        &\quad-\bigg(-\psi_{,zz}+\psi_{,rr}+{1\over r}\psi_{,r}-{\psi\over r^2}\bigg)\bigg({1\over r}u_{,z}\bigg)\bigg({1\over r}u_{,r}\bigg)\zeta^2\bigg]dxdt'\cr
      &\equiv J_1+J_2+J_3,\cr}
    \label{5.9}
  \end{equation}
  where
  $$\eqal{
      J_1&=-\intop_{\bar\R_t^3}(\tilde\psi_{,zr}-\psi_{,z}\dot\zeta){1\over r}\tilde u_{,z}{1\over r}u_{,z}dxdt',\cr
      J_2&=\intop_{\bar\R_t^3}\bigg(\tilde\psi_{,rz}-\psi_{,z}\dot\zeta+{\tilde\psi_{,z}\over r}\bigg)\bigg({1\over r}u_{,r}\bigg)^2\zeta dxdt',\cr
      J_3&=-\intop_{\bar\R_t^3}\bigg(-\tilde\psi_{,zz}+\tilde\psi_{,rr}-2\psi_{,r}\dot\zeta-\psi\ddot\zeta+{1\over r}\tilde\psi_{,r}-{1\over r}\psi\dot\zeta-{\tilde\psi\over r^2}\bigg)\cdot\cr
      &\quad\cdot{1\over r}\tilde u_{,z}{1\over r}u_{,r}dxdt'.\cr}
  $$
  First, we examine $J_1$. We can write it in the form
  \begin{equation}\eqal{
      J_1&=-\intop_{\bar\R_t^3}\tilde\psi_{,zr}{1\over r}\tilde u_{,z}{1\over r}u_{,z}dxdt'+\intop_{\bar\R_t^3}\psi_{,z}\dot\zeta{1\over r}\tilde u_{,z}{1\over r}u_{,z}dxdt'\cr
      &\equiv J_1^1+J_1^2.\cr}
    \label{5.10}
  \end{equation}
  Integrating by parts with respect to $z$ in $J_1^1$ yields
  $$\eqal{
      J_1^1&=\intop_{\bar\R_t^3}\tilde\psi_{,zzr}{1\over r}\tilde u_{,z}{1\over r}udxdt'+\intop_{\bar\R_t^3}\tilde\psi_{,zr}{1\over r^2}\tilde u_{,zz}udxdt'\equiv I_1+I_2.\cr}
  $$
  Using that $\tilde\psi=r\tilde\psi_1$ in $I_1$ implies
  $$
    I_1=\intop_{\bar\R_t^3}\bigg({\tilde\psi_{1,zz}\over r}+\tilde\psi_{1,zzr}\bigg){\tilde u_{,z}\over r}udxdt'\equiv I_1^1+I_1^2,
  $$
  where
  $$\eqal{
    |I_1^1|&\le\intop_{\bar\R_t^3}\bigg|{\tilde\psi_{1,zz}\over r^{1-\varepsilon_0}}\bigg|\,\bigg|{\tilde u_{,z}\over r}\bigg|\,|v_\varphi|^{\varepsilon_0}|u|^{1-\varepsilon_0}dxdt'\cr
    &\le D_2^{1-\varepsilon_0}|v_\varphi|_{\infty,\bar\R_t^3}^{\varepsilon_0}\bigg|{\tilde u_{,z}\over r}\bigg|_{2,\bar\R_t^3}\bigg|{\tilde\psi_{1,zz}\over r^{1-\varepsilon_0}}\bigg|_{2,\bar\R_t^3},\cr}
  $$
  where (\ref{2.9}) is used. Using (\ref{2.1}),
  $$
    \bigg|{\tilde u_{,z}\over r}\bigg|_{2,\bar\R_t^3}\le|v_{\varphi,z}|_{2,\bar\R_t^3}\le D_1,
  $$
  the Hardy inequality and (\ref{3.40}) imply
  $$
    \bigg|{\tilde\psi_{1,zz}\over r^{1-\varepsilon_0}}\bigg|_{2,\bar\R_t^3}\le c{r_0^{\varepsilon_0}\over\varepsilon_0}|\tilde\psi_{1,zzr}|_{2,\bar\R_t^3}\le c{r_0^{\varepsilon_0}\over\varepsilon_0}(|\tilde\Gamma_{,z}|_{2,\bar\R_t^3}+D_1).
  $$
  In virtue of the above estimates, we have
  \begin{equation}
    |I_1^1|\le cD_1D_2^{1-\varepsilon_0}{r_0^{\varepsilon_0}\over\varepsilon_0}|v_\varphi|_{\infty,\bar\R_t^3}^{\varepsilon_0} (|\tilde\Gamma_{,z}|_{2,\bar\R_t^3}+D_1).
    \label{5.11}
  \end{equation}
  Next, (\ref{2.1}), (\ref{2.9}) and (\ref{3.40}) imply
  \begin{equation}\eqal{
    |I_1^2|&\le|u|_{\infty,\bar\R_t^3}\bigg|{u_{,z}\over r}\bigg|_{2,\bar\R_t^3}|\tilde\psi_{1,rzz}|_{2,\bar\R_t^3}\cr
    &\le cD_1D_2(|\tilde\Gamma_{,z}|_{2,\bar\R_t^3}+D_1).\cr}
    \label{5.12}
  \end{equation}
  Summarizing, (\ref{5.11}) and (\ref{5.12}) imply
  \begin{equation}\eqal{
      |I_1|&\le c\bigg[D_1D_2+D_1D_2^{1-\varepsilon_0}{r_0^{\varepsilon_0}\over\varepsilon_0} |v_\varphi|_{\infty,\bar\R_t^3}^{\varepsilon_0}\bigg][|\tilde\Gamma_{,z}|_{2,\bar\R_t^3}+D_1].\cr}
    \label{5.13}
  \end{equation}
  In view of the mapping $\tilde\psi=r\tilde\psi_1$, we have
  $$
    I_2=\intop_{\bar\R_t^3}\bigg({\tilde\psi_{1,z}\over r^2}+{\tilde\psi_{1,rz}\over r}\bigg)\tilde u_{,zz}udxdt'\equiv I_2^1+I_2^2.
  $$
  Estimates (\ref{2.9}), (\ref{2.22}) and (\ref{3.51}) imply
  $$\eqal{
    |I_2^2|&\le|u|_{\infty,\bar\R_t^3}|\tilde u_{,zz}|_{2,\bar\R_t^3}\bigg|{\tilde\psi_{1,rz}\over r}\bigg|_{2,\bar\R_t^3}\le cD_2D_4(|\tilde\Gamma_{,z}|_{2,\bar\R_t^3}+D_1).\cr}
  $$
  Hence,
  \begin{equation}
    |I_2|\le cD_2D_4(|\tilde\Gamma_{,z}|_{2,\bar\R_t^3}+D_1)+\intop_{\bar\R_t^3}{\tilde\psi_{1,z}\over r^2}\tilde u_{,zz}udxdt'.
    \label{5.14}
  \end{equation}
  The decomposition of $J_1^1$ and (\ref{5.13}), (\ref{5.14}) imply
  \begin{equation}\eqal{
      |J_1^1|&\le c\bigg[D_1D_2+D_2D_4+D_1D_2^{1-\varepsilon_0}{r_0^{\varepsilon_0}\over\varepsilon_0} |v_\varphi|_{\infty,\bar\R_t^3}^{\varepsilon_0}\bigg]\cdot\cr
      &\quad\cdot[|\tilde\Gamma_{1,z}|_{2,\bar\R_t^3}+D_1]+\bigg|\intop_{\R_t^3}{\tilde\psi_{1,z}\over r^2}u_{,zz}udxdt'\bigg|.\cr}
    \label{5.15}
  \end{equation}
  Looking for $J_1$ defined in (\ref{5.10}), we need to estimate $J_1^2$. From the properties of $\dot\zeta$, we have
  \begin{equation}\eqal{
      |J_1^2|&\le\phi(r_0)\intop_{\bar\R_t^3}|\psi_{,z}|\,|\tilde u_{,z}|\,|u|dxdt'\cr
      &\le\phi(r_0)D_2|\psi_{,z}|_{2,\bar\R_t^3}|\tilde v_{\varphi,z}|_{2,\bar\R_t^3}\le\phi(r_0)D_1^2D_2,\cr}
    \label{5.16}
  \end{equation}
  where (\ref{2.1}) and (\ref{2.15}) were used.

  Using (\ref{5.15}) and (\ref{5.16}) in (\ref{5.10}) yields
  \begin{equation}\eqal{
      |J_1|&\le c\bigg[D_1D_2+D_2D_4+D_1D_2^{1-\varepsilon_0}{r_0^{\varepsilon_0}\over\varepsilon_0} |v_\varphi|_{\infty,\bar\R_t^3}^{\varepsilon_0}\bigg]\cdot\cr
      &\quad\cdot[|\tilde\Gamma_{,z}|_{2,\bar\R_t^3}+D_1]+\phi(r_0)D_1^2D_2\cr
      &\quad+\bigg|\intop_{\bar\R_t^3}{\tilde\psi_{,z}\over r^2}\tilde u_{,zz}udxdt'\bigg|.\cr}
    \label{5.17}
  \end{equation}
  Next, we examine
  \begin{equation}\eqal{
      J_2&=\intop_{\bar\R_t^3}\bigg(\tilde\psi_{,rz}+{\tilde\psi_{,z}\over r}\bigg)\bigg({1\over r}u_{,r}\bigg)^2\zeta dxdt'\cr
      &\quad-\intop_{\bar\R_t^3}\psi_{,z}\dot\zeta\bigg({1\over r}u_{,r}\bigg)^2\zeta dxdt'\equiv J_2^1+J_2^2.\cr}
    \label{5.18}
  \end{equation}
  In view of Remark \ref{r1.2}, we have
  $$
    \bigg(\tilde\psi_{,rz}+{\tilde\psi_{,z}\over r}\bigg){1\over r}u_{,r}u\bigg|_{r=0}=0.
  $$
  Therefore, integrating by parts with respect to $r$ in $J_2^1$ yields
  \begin{equation}\eqal{
      J_2^1&=-\intop_{\bar\R_t^3}\bigg(\tilde\psi_{,rz}+{\tilde\psi_{,z}\over r}\bigg)_{,r}{1\over r}u_{,r}u\zeta drdzdt'\cr
      &\quad-\intop_{\bar\R_t^3}\bigg(\tilde\psi_{,rz}+{\tilde\psi_{,z}\over r}\bigg)\bigg({1\over r}u_{,r}\bigg)_{,r}u\zeta drdzdt'\cr
      &\quad-\intop_{\bar\R_t^3}\bigg(\tilde\psi_{,rz}+{\tilde\psi_{,z}\over r}\bigg){1\over r}u_{,r}\dot\zeta udrdzdt'\equiv K_1+K_2+K_3.\cr}
    \label{5.19}
  \end{equation}
  Consider $K_1$. Using $\tilde\psi=r\tilde\psi_1$, we get
  $$
    K_1=-\intop_{\bar\R_t^3}(3\tilde\psi_{1,rz}+r\tilde\psi_{1,rrz}){1\over r}{1\over r}u_{,r}\zeta udxdt'.
  $$
  Therefore,
  \begin{equation}\eqal{
    |K_1|&\le\bigg(3\bigg|{\tilde\psi_{1,rz}\over r}\bigg|_{2,\bar\R_t^3}+|\tilde\psi_{1,rrz}|_{2,\bar\R_t^3}\bigg)\bigg|{1\over r}u_{,r}\bigg|_{2,\bar\R_t^3}|u|_{\infty,\bar\R_t^3}\cr
    &\le cD_1D_2(|\tilde\Gamma_{,z}|_{2,\bar\R_t^3}+D_1),\cr}
    \label{5.20}
  \end{equation}
  where we used (\ref{2.1}), (\ref{2.9}), (\ref{3.41}), (\ref{3.51}) and
  $$
    \bigg|{1\over r}u_{,r}\bigg|_{2,\bar\R_t^3}\le\bigg|{v_\varphi\over r}\bigg|_{2,\bar\R_t^3}+|v_{\varphi ,r}|_{2,\bar\R_t^3}\le cD_1.
  $$
  Using that $\tilde\psi=r\tilde\psi_1$, we have
  $$\eqal{
      K_2&=\intop_{\bar\R_t^3}(2\tilde\psi_{1,z}+r\tilde\psi_{1,rz}){1\over r}\bigg({1\over r}u_{,r}\bigg)_{,r}\zeta udxdt'\cr
      &=\intop_{\bar\R_t^3}\tilde\psi_{1,rz}\bigg({1\over r}u_{,r}\bigg)_{,r}\zeta udxdt'+2\intop_{\bar\R_t^3}{\tilde\psi_{1,z}\over r}\bigg({1\over r}u_{,r}\zeta\bigg)_{,r}udxdt'\cr
      &\quad-2\intop_{\bar\R_t^3}{\tilde\psi_{1,z}\over r}{1\over r}u_{,r}\dot\zeta udxdt'\equiv K_2^1+K_2^2+K_2^3.\cr}
  $$
  First, we examine
  $$\eqal{
      K_2^1&=\intop_{\bar\R_t^3}{\tilde\psi_{1,rz}\over r}u_{,rr}\zeta udxdt'-\intop_{\bar\R_t^3}{\tilde\psi_{1,rz}\over r}{1\over r}u_{,r}\zeta udxdt'\cr
      &\equiv K_2^{11}+K_2^{12},\cr}
  $$
  where
  $$\eqal{
    |K_2^{11}|&\le\bigg|{\tilde\psi_{1,rz}\over r}\bigg|_{2,\bar\R_t^3}|u_{,rr}|_{2,\bar\R_t^3}|u|_{\infty,\bar\R_t^3}\cr
    &\le cD_2D_5(|\tilde\Gamma_{,z}|_{2,\bar\R_t^3}+D_1),\cr}
  $$
  where (\ref{2.9}), (\ref{2.23}), (\ref{3.51}) were used, and
  $$\eqal{
    |K_2^{12}|&\le\bigg|{\tilde\psi_{1,rz}\over r}\bigg|_{2,\bar\R_t^3}\bigg|{1\over r}u_{,r}\bigg|_{2,\bar\R_t^3}|u|_{\infty,\bar\R_t^3}\cr
    &\le cD_1D_5(|\tilde\Gamma_{,z}|_{2,\bar\R_t^3}+D_1),\cr}
  $$
  where (\ref{2.1}), (\ref{2.9}), (\ref{3.51}) were used.

  Next,
  $$
    |K_2^2|\le c\intop_{\bar\R_t^3}\bigg|{\tilde\psi_{1,z}\over r}\bigg({1\over r}u_{,r}\zeta\bigg)_{,r}u\bigg|dxdt'.
  $$
  Finally,
  $$
    |K_2^3|\le\phi(r_0)\intop_{\bar\R_t^3}|\tilde\psi_{1,z}|\,|u_{,r}|\,|u|dxdt'\le\phi(r_0)D_1^2D_2,
  $$
  where (\ref{2.1}), (\ref{2.9}), (\ref{2.15}) have been used.

  Summarizing,
  \begin{equation}\eqal{
    |K_2|&\le c(D_2D_5+D_1D_2)(|\tilde\Gamma_{,z}|_{2,\bar\R_t^3}+D_1)\cr
    &\quad+c\intop_{\bar\R_t^3}\bigg|{\tilde\psi_{1,z}\over r}\bigg({1\over r}u_{,r}\zeta\bigg)_{,r}u\bigg|dxdt'.\cr}
    \label{5.21}
  \end{equation}
  Finally,
  \begin{equation}
    |K_3|\le\phi(r_0)\intop_{\bar\R_t^3}(|\psi_{,rz}|+|\tilde\psi_{,z}|)|u_{,r}|\,|u|dxdt'\le \phi(r_0)D_1^2D_2,
    \label{5.22}
  \end{equation}
  where we used (\ref{2.1}), (\ref{2.9}) and (\ref{2.15}).

  Using, (\ref{5.20})--(\ref{5.22}) in (\ref{5.19}) yields
  \begin{equation}\eqal{
    |J_2^1|&\le c(D_1D_2+D_2D_5)(|\tilde\Gamma_{,z}|_{2,\bar\R_t^3}+D_1)\cr
    &\quad+c\intop_{\bar\R_t^3}\bigg|{\tilde\psi_{1,z}\over r}\bigg({1\over r}u_{,r}\zeta\bigg)_{,r}u\bigg|dxdt'.\cr}
    \label{5.23}
  \end{equation}
  In the final step we examine $J_2^2$. Using properties of $\dot\zeta$ and integrating by parts with respect to $r$, we obtain
  $$\eqal{
      J_2^2&=-\intop_{\bar\R_t^3}\psi_{,z}\dot\zeta{1\over r}u_{,r}u_{,r}\zeta drdzdt'=\intop_{\bar\R_t^3}\bigg(\psi_{,z}\dot\zeta{1\over r}u_{,r}\zeta\bigg)_{,r}udrdzdt'\cr
      &=\intop_{\bar\R_t^3}\bigg[\psi_{,rz}{1\over r}\zeta\dot\zeta u_{,r}+\psi_{,z}\bigg({1\over r}\zeta\dot\zeta\bigg)_{,r}u_{,r}+\psi_{,z}{1\over r}\zeta\dot\zeta u_{,rr}\bigg]udrdzdt'.\cr}
  $$
  Therefore,
  \begin{equation}\eqal{
      |J_2^2|&\le\phi(r_0)\|\psi_{,z}\|_{1,2,\bar\R_t^3}(|u_{,r}|_{2,\bar\R_t^3}+|u_{,rr}|_{2,\bar\R_t^3})D_2\cr
      &\le\phi(r_0)D_1D_2D_5.\cr}
    \label{5.24}
  \end{equation}
  Using (\ref{5.23}) and (\ref{5.24}) in (\ref{5.18}) implies
  \begin{equation}\eqal{
    |J_2|&\le c(D_1D_2+D_2D_5)(|\tilde\Gamma_{,z}|_{2,\bar\R_t^3}+D_1)\cr
    &\quad+c\intop_{\bar\R_t^3}\bigg|{\tilde\psi_{1,z}\over r}\bigg({1\over r}u_{,r}\zeta\bigg)_{,r}u\bigg|dxdt'.\cr}
    \label{5.25}
  \end{equation}
  Finally, we examine $J_3$. We write it in the form
  \begin{equation}\eqal{
      J_3&=-\intop_{\bar\R_t^3}\bigg(-\tilde\psi_{,zz}+\tilde\psi_{,rr}+{1\over r}\tilde\psi_{,r}-{\tilde\psi\over r^2}\bigg){1\over r}u_{,z}{1\over r}u_{,r}\zeta dxdt'\cr
      &\quad-\intop_{\bar\R_t^3}\bigg(-2\psi_{,r}\dot\zeta-\psi\ddot\zeta-{1\over r}\psi\dot\zeta\bigg){1\over r}u_{,z}{1\over r}u_{,r}\zeta dxdt'\cr
      &\equiv J_3^1+J_3^2.\cr}
    \label{5.26}
  \end{equation}
  Using that $\tilde\psi=r\tilde\psi_1$ in $J_3^1$ yields
  $$
    J_3^1=-\intop_{\bar\R_t^3}\bigg(-\tilde\psi_{1,zz}+{3\tilde\psi_{1,r}\over r}+\tilde\psi_{1,rr}\bigg){1\over r}u_{,r}u_{,z}dxdt'.
  $$
  Integrating by parts with respect to $z$ implies
  $$\eqal{
      J_3^1&=\intop_{\bar\R_t^3}\bigg(-\tilde\psi_{1,zzz}+{3\over r}\tilde\psi_{1,rz}+\tilde\psi_{1,rrz}\bigg){1\over r}u_{,r}udxdt'\cr
      &\quad+\intop_{\bar\R_t^3}\bigg(-\tilde\psi_{1,zz}+{3\over r}\tilde\psi_{1,r}+\tilde\psi_{1,rr}\bigg){1\over r}u_{,rz}udxdt'\cr
      &\equiv J_3^{11}+J_3^{12}.\cr}
  $$
  Considering (\ref{2.1}), (\ref{2.9}), (\ref{3.41}) and (\ref{3.51}), we obtain
  \begin{equation}
    |J_3^{11}|\le cD_1D_2(|\tilde\Gamma_{,z}|_{2,\bar\R_t^3}+D_1).
    \label{5.27}
  \end{equation}
  To estimate $J_3^{12}$ we need equation (\ref{1.23}).

  Localizing it to a neighborhood of the axis of symmetry (we multiply (\ref{1.23}) by $\zeta$), we have
  \begin{equation}
    -\tilde\psi_{1,rr}-{3\over r}\tilde\psi_{1,r}-\tilde\psi_{1,zz}+2\psi_{1,r}\dot\zeta+\psi_1\ddot\zeta+{3\over r}\psi_1\dot\zeta=\tilde\Gamma.
    \label{5.28}
  \end{equation}
  Using (\ref{5.28}) in $J_3^{12}$ yields
  $$
    J_3^{12}=\intop_{\bar\R_t^3}\bigg(-2\tilde\psi_{1,zz}+2\psi_{1,r}\dot\zeta+\psi_1\ddot\zeta+{3\over r}\psi_1\dot\zeta-\tilde\Gamma\bigg){1\over r}u_{,rz}udxdt'.
  $$
  Now, we estimate the particular terms in $J_3^{12}$.

  First, we examine
  $$\eqal{
    &\bigg|\intop_{\bar\R_t^3}\tilde\psi_{1,zz}{1\over r}u_{,rz}udxdt'\bigg|=\bigg|\intop_{\bar\R_t^3} {\tilde\psi_{1,zz}\over r^{1-\varepsilon_0}}u_{,rz}{u\over r^{\varepsilon_0}}dxdt'\bigg|\cr
    &\le D_2^{1-\varepsilon_0}|v_\varphi|_{\infty,\bar\R_t^3}^{\varepsilon_0}\bigg|{\tilde\psi_{1,zz}\over r^{1-\varepsilon_0}}\bigg|_{2,\bar\R_t^3}|u_{,rz}|_{2,\bar\R_t^3}\cr
    &\le D_2^{1-\varepsilon_0}|v_\varphi|_{\infty,\bar\R_t^3}^{\varepsilon_0}D_4 {r_0^{\varepsilon_0}\over\varepsilon_0}|\tilde\psi_{1,zzr}|_{2,\bar\R_t^3}\cr
    &\le D_2^{1-\varepsilon_0}|v_\varphi|_{\infty,\bar\R_t^3}^{\varepsilon_0}D_4 {r_0^{\varepsilon_0}\over\varepsilon_0}(|\tilde\Gamma_{,z}|_{2,\bar\R_t^3}+D_1),\cr}
  $$
  where we used the Hardy inequality and (\ref{2.22}) in the second inequality and (\ref{3.41}) in the last.

  Next, we examine
  $$\eqal{
    &\bigg|\intop_{\bar\R_t^3}\tilde\Gamma{1\over r}u_{,rz}udxdt'\bigg|=\bigg|\intop_{\bar\R_t^3}{\tilde\Gamma\over r^{1-\varepsilon_0}}u_{,rz}{u\over r^{\varepsilon_0}}dxdt'\bigg|\cr
    &\le D_2^{1-\varepsilon_0}|v_\varphi|_{\infty,\bar\R_t^3}^{\varepsilon_0}D_4 {r_0^{\varepsilon_0}\over\varepsilon_0}|\tilde\Gamma_{,r}|_{2,\bar\R_t^3},\cr}
  $$
  where we used (\ref{2.9}), (\ref{2.22}) and the Hardy inequality (\ref{2.13}).

  Finally, we estimate the last part of $J_3^{12}$,
  $$\eqal{
      &\intop_{\bar\R_t^3}\bigg(2\psi_{1,r}\dot\zeta+\psi_1\ddot\zeta+{3\over r}\psi_1\dot\zeta\bigg){1\over r}u_{,rz}udxdt'\cr
      &=\intop_{\bar\R_t^3}\bigg(2{\psi_{,r}\dot\zeta\over r}+{2\psi\over r^2}\dot\zeta+\psi_1\ddot\zeta+{3\over r}\psi_1\dot\zeta\bigg){1\over r}u_{,rz}udxdt'.\cr}
  $$
  By the properties of $\dot\zeta$ and (\ref{2.9}), (\ref{2.15}) and (\ref{2.22}) the above expression is bounded by
  $$
    \phi(r_0)D_1D_2D_4.
  $$
  Summarizing,
  \begin{equation}
    |J_3^{12}|\le cD_2^{1-\varepsilon_0}D_4|v_\varphi|_{\infty,\bar\R_t^3}^{\varepsilon_0} {r_0^{\varepsilon_0}\over\varepsilon_0}(|\nabla\tilde\Gamma|_{2,\bar\R_t^3}+D_1)+cD_1D_2D_4.
    \label{5.29}
  \end{equation}
  Estimates (\ref{5.27}), (\ref{5.29}) and the form of $J_3^1$ yield
  \begin{equation}
    |J_3^1|\le c\bigg(D_1D_2+D_2^{1-\varepsilon_0}|v_\varphi|_{\infty,\bar\R_t^3}^{\varepsilon_0}D_4 {r_0^{\varepsilon_0}\over\varepsilon_0}\bigg)(|\nabla\tilde\Gamma|_{2,\bar\R_t^3}+D_1).
    \label{5.30}
  \end{equation}
  Integrating by parts with respect to $z$ in $J_3^2$ yields
  $$\eqal{
      J_3^2&=\intop_{\bar\R_t^3}\bigg(-2\psi_{,rz}\dot\zeta-\psi_{,z}\ddot\zeta-{1\over r}\psi_{,z}\dot\zeta\bigg){1\over r}u_{,r}{1\over r}u\zeta dxdt'\cr
      &\quad+\intop_{\bar\R_t^3}\bigg(-2\psi_{,r}\dot\zeta-\psi\ddot\zeta-{1\over r}\psi\ddot\zeta\bigg){1\over r}u_{,rz}{1\over r}u\zeta dxdt'.\cr}
  $$
  Applying properties of $\dot\zeta$ and (\ref{2.9}), (\ref{2.15}), (\ref{2.24}), we obtain
  \begin{equation}
    |J_3^2|\le cD_1D_2(D_1+D_4).
    \label{5.31}
  \end{equation}
  Using estimates (\ref{5.30}) and (\ref{5.31}) in (\ref{5.26}) implies
  \begin{equation}\eqal{
    |J_3|&\le c(D_1D_2+D_2^{1-\varepsilon_0}|v_\varphi|_{\infty,\bar\R_t^3}^{\varepsilon_0}D_4 {r_0^{\varepsilon_0}\over\varepsilon_0}\bigg)\cdot\cr
    &\quad\cdot(|\nabla\tilde\Gamma|_{2,\bar\R_t^3}+D_1).\cr}
    \label{5.32}
  \end{equation}
  Using estimates (\ref{5.17}), (\ref{5.25}) and (\ref{5.32}), we have
  \begin{equation}\eqal{
      |J|&\le c\bigg[D_1D_2+D_2D_4+D_2D_5\cr
        &\quad+(D_1+D_4)D_2^{1-\varepsilon_0}{r_0^{\varepsilon_0}\over\varepsilon_0} |v_\varphi|_{\infty,\bar\R_t^3}^{\varepsilon_0}\bigg]\cdot(|\nabla\tilde\Gamma|_{2,\bar\R_t^3}+D_1)\cr
      &\quad+c\bigg|\intop_{\bar\R_t^3}{\tilde\psi_{1,z}\over r^2}u_{,zz}u\zeta dxdt'\bigg|+c\bigg|\intop_{\bar\R_t^3}{\tilde\psi_{1,z}\over r}\bigg({1\over r}u_{,r}\zeta\bigg)_{,r}udxdt'\bigg|.\cr}
    \label{5.33}
  \end{equation}
  Using (\ref{1.13}) and applying an appropriate integration by parts, the second and the third terms on the r.h.s. of (\ref{5.6}) are bounded by
  \begin{equation}
    cD_1(D_1+D_4+D_5)(D_2+1).
    \label{5.34}
  \end{equation}
  Applying the H\"older and Young inequalities the fourth and the fifth terms on the r.h.s. of (\ref{5.6}) are estimated by
  \begin{equation}
    {\nu\over 4}(|\tilde\omega_r|_{6,\bar\R_t^3}^2+|\tilde\omega_z|_{6,\bar\R_t^3}^2)+c(1/\nu)(|\tilde F_r|_{6/5,2,\bar\R_t^3}^2+|\tilde F_z|_{6/5,2,\bar\R_t^3}^2).
    \label{5.35}
  \end{equation}
  Using estimates (\ref{5.33}), (\ref{5.34}) and (\ref{5.35}) in (\ref{5.6}) yields
  \begin{equation}\eqal{
    &\|\tilde\omega_r\|_{V(\bar\R_t^3)}^2+\|\tilde\omega_z\|_{V(\bar\R_t^3)}^2+\nu\bigg|{\tilde\omega_r\over r}\bigg|_{2,\bar\R_t^3}^2\cr
    &\le c\bigg|\intop_{\bar\R_t^3}{\tilde\psi_{1,z}\over r}{u_{,zz}\over r}\zeta udxdt'\bigg|+c\bigg|\intop_{\bar\R_t^3}{\tilde\psi_{1,z}\over r}\bigg({1\over r}u_{,r}\zeta\bigg)_{,r}udxdt'\bigg|\cr
    &\quad+c(D_{11}+D_{12}|v_\varphi|_{\infty,\bar\R_t^3}^{\varepsilon_0}|\nabla\tilde\Gamma|_{2,\bar\R_t^3}+ cD_{12}D_1|v_\varphi|_{\infty,\bar\R_t^3}^{\varepsilon_0}+cD_8,\cr}
    \label{5.36}
  \end{equation}
  where
  \begin{equation}\eqal{
      &D_{11}=D_1D_2+D_2D_4+D_2D_5,\cr
      &D_{12}=(D_1+D_4)D_1^{1-\varepsilon_0}{r_0^{\varepsilon_0}\over\varepsilon_0},\cr
      &D_8=D_{11}D_1+D_1(D_1+D_4+D_5)(D_2+1)\cr
      &\quad+|\tilde F_r|_{6/5,2,\bar\R_t^3}^2+|\tilde F_z|_{6/5,2,\bar\R_t^3}^2+|\tilde\omega_r(0)|_{2,\bar\R^3}^2+|\tilde\omega_z(0)|_{2,\bar\R^3}^2.\cr}
    \label{5.37}
  \end{equation}
  Recalling (\ref{1.13}), we have that $\omega_r=-{1\over r}u_{,z}$, $\omega_z={1\over r}u_{,r}$.

  Then the l.h.s. of (\ref{5.36}) reads
  \begin{equation}\eqal{&\bigg|{1\over r}u_{,z}(t)\zeta\bigg|_{2,\bar\R^3}^2+\bigg|{1\over r}u_{,r}(t)\zeta\bigg|_{2,\bar\R^3}^2+\nu\bigg(\bigg|\bigg({1\over r}u_{,z}\zeta\bigg)_{,z}\bigg|_{2,\bar\R_t^3}^2\cr
    &\quad+\bigg|\bigg({1\over r}u_{,r}\zeta\bigg)_{,z}\bigg|_{2,\bar\R_t^3}^2+\bigg|\bigg({1\over r}u_{,z}\zeta\bigg)_{,r}\bigg|_{2,\bar\R_t^3}^2+\bigg|\bigg({1\over r}u_{,r}\zeta\bigg)_{,r}\bigg|_{2,\bar\R_t^3}^2\bigg)\cr
    &\quad+\nu\bigg|{\tilde\omega_r\over r}\bigg|_{2,\bar\R_t^3}^2.\cr}
    \label{5.38}
  \end{equation}
  Now, we estimate the first and the second terms on the r.h.s. of (\ref{5.36}). By the H\"older and Young inequalities the first term on the r.h.s. of (\ref{5.36}) (denoted by $L_1$) is bounded by
  $$
    L_1\le{\nu\over 4}\bigg|\bigg({1\over r}u_{,z}\zeta\bigg)_{,z}\bigg|_{2,\bar\R_t^3}^2+{c\over\nu}\intop_{\bar\R_t^3}{\tilde\psi_{1,z}^2\over r^2}u^2dxdt'\equiv L_1^1.
  $$
  The second term in $L_1^1$ is bounded by
  $$
    \intop_{\bar\R_t^3}{\tilde\psi_{1,z}^2\over r^{2(1-\varepsilon_0)}}{u^2\over r^{2\varepsilon_0}}dxdt'\le|u|_{\infty,\R_t^3}^{2(1-\varepsilon_0)} |v_\varphi|_{\infty,\bar\R_t^3}^{2\varepsilon_0}\intop_{\bar\R_t^3}{\tilde\psi_{1,z}^2\over r^{2(1-\varepsilon_0)}}dxdt',
  $$
  where $\varepsilon_0$ is assumed to be small. Next, we examine the integral
  $$
    \intop_{\bar\R^3}{\tilde\psi_{1,z}^2\over r^{2(1-\varepsilon_0)}}rdrdz=\intop_{\R^3}r^{-2(1-\varepsilon_0)+1}\tilde\psi_{1,z}^2drdz\equiv M_1.
  $$
  We apply Lemma \ref{l2.5} with $\beta={1\over 2}-\varepsilon_0$.

  Then (\ref{2.13}) yields
  $$
    M_1\le{1\over\varepsilon_0^2}\intop_{\bar\R^3}r^{2\varepsilon_0}\tilde\psi_{1,zr}^2dx\le {r_0^{2\varepsilon_0}\over\varepsilon_0^2}\intop_{\bar\R^3}\tilde\psi_{1,zr}^2dx\equiv M_2.
  $$
  We use the interpolation inequality (see \cite[Ch. 3, Sect. 15]{BIN})
  $$
    \intop_{\bar\R^3}\tilde\psi_{1,rz}^2dx\le c\bigg(\intop_{\R^3}|\nabla^2\tilde\psi_{1,z}|^2dx\bigg)^\theta \bigg(\intop_{\bar\R^3}\tilde\psi_{1,z}^2dx\bigg)^{1-\theta},
  $$
  where $\theta$ satisfies
  $$
    {3\over 2}-1=(1-\theta){3\over 2}+\theta\bigg({3\over 2}-2\bigg),\quad {\rm so}\ \ \theta={1\over 2}
  $$
  and $\nabla^2=(\partial_r^2,\partial_r\partial_z,\partial_z^2)$. Using (\ref{2.15}), we get
  $$
    \intop_{\R_t^3}\tilde\psi_{1,rz}^2dxdt'\le c|\nabla^2\tilde\psi_{1,z}|_{2,\bar\R_t^3}|\tilde\psi_{1,z}|_{2,\bar\R_t^3}\le cD_1|\nabla^2\tilde\psi_{1,z}|_{2,\bar\R_t^3}.
  $$
  Summarizing, we derive
  $$\eqal{
    L_1&\le{\nu\over 4}\bigg|\bigg({1\over r}u_{,z}\zeta\bigg)_{,z}\bigg|_{2,\bar\R_t^3}^2+{c\over\nu} |u|_{\infty,\bar\R_t^3}^{2(1-\varepsilon_0)}|v_\varphi|_{\infty,\bar\R_t^3}^{2\varepsilon_0} {r_0^{2\varepsilon_0}\over\varepsilon_0^2}\cdot\cr
    &\quad\cdot D_1|\nabla^2\tilde\psi_{1,z}|_{2,\bar\R_t^3}\cr
    &\le{\nu\over 4}\bigg|\bigg({1\over r}u_{,z}\zeta\bigg)_{,z}\bigg|_{2,\bar\R_t^3}^2+{c\over\nu}D_1D_2^{2(1-\varepsilon_0)} {r_0^{2\varepsilon_0}\over\varepsilon_0^2}|v_\varphi|_{\infty,\bar\R_t^3}^{2\varepsilon_0}\cdot\cr
    &\quad\cdot(|\nabla\tilde\Gamma|_{2,\bar\R_t^3}+D_1),\cr}
  $$
  where (\ref{3.41}) was used.

  Similarly, the second term on the r.h.s. of (\ref{5.36}) (denoted by $L_2$) is bounded by
  $$%\eqal{
    L_2\le{\nu\over 4}\bigg|\bigg({1\over r}u_{,r}\zeta\bigg)_{,r}\bigg|_{2,\bar\R_t^3}^2+{c\over\nu}D_1D_2^{2(1-\varepsilon_0)} {r_0^{2\varepsilon_0}\over\varepsilon_0^2}|v_\varphi|_{\infty,\bar\R_t^3}^{2\varepsilon_0}\cdot%\cr&\quad
    (|\nabla\tilde\Gamma|_{2,\bar\R_t^3}+D_1).%\cr}
  $$
  Using estimates of $L_1$ and $L_2$ in (\ref{5.36}), we obtain
  \begin{equation}\eqal{
    &\|\tilde\omega_r\|_{V(\bar\R_t^3)}^2+\|\tilde\omega_z\|_{V(\bar\R_t^3)}^2+\nu|{\tilde\omega_r\over r}\bigg|_{2,\bar\R_t^3}^2\cr
    &\le cD_9|v_\varphi|_{\infty,\bar\R_t^3}^{2\varepsilon_0}|\nabla\tilde\Gamma|_{2,\bar\R_t^3}+ cD_9D_1|v_\varphi|_{\infty,\bar\R_t^3}^{2\varepsilon_0}\cr
    &\quad+c(D_{11}+D_{12}|v_\varphi|_{\infty,\bar\R_t^3}^{\varepsilon_0})|\nabla\tilde\Gamma|_{2,\bar\R_t^3}+ cD_{12}D_1|v_\varphi|_{\infty,\bar\R_t^3}^{\varepsilon_0}+cD_8,\cr}
    \label{5.39}
  \end{equation}
  where
  \begin{equation}
    D_9=D_1D_2^{2(1-\varepsilon_0)}{r_0^{2\varepsilon_0}\over\varepsilon_0^2}
    \label{5.40}
  \end{equation}
  The above inequality implies (\ref{5.5}) and concludes the proof.
\end{proof}

\begin{remark}\label{r5.2}
  We simplify (\ref{5.5}) in the following way
  \begin{equation}\eqal{
    &\|\tilde\omega_r\|_{V(\bar\R_t^3)}^2+\|\tilde\omega_z\|_{V(\bar\R_t^3)}^2+\nu\bigg|{\tilde\omega_r\over r}\bigg|_{2,\bar\R_t^3}^2\cr
    &\le c(D_{11}+D_{12}^2+(1+D_9)|v_\varphi|_{\infty,\bar\R_t^3}^{2\varepsilon_0}) |\nabla\tilde\Gamma|_{2,\bar\R_t^3}\cr
    &\quad+c(D_1^2D_{12}^2+(1+D_1D_9)|v_\varphi|_{\infty,\bar\R_t^3}^{2\varepsilon_0})+cD_8\cr
    &\le cD_{13}|v_\varphi|_{\infty,\bar\R_t^3}^{2\varepsilon_0}|\nabla\tilde\Gamma|_{2,\bar\R_t^3}+cD_{14} |v_\varphi|_{\infty,\bar\R_t^3}^{2\varepsilon_0}+cD_8.\cr}
    \label{5.41}
  \end{equation}
\end{remark}

\section{Estimates for $v_\varphi$}\label{s6}

The aim of this section is to derive $L_\infty$ and localized $L_s$ estimates
for the angular component $v_\varphi$, which will be used to close the global
a priori bounds. From (\ref{1.7}) the following problem for $v_\varphi$ reads
\begin{equation}\eqal{
    &v_{\varphi,t}+v\cdot\nabla v_\varphi+{v_r\over r}v_\varphi-\nu\Delta v_\varphi+\nu{v_\varphi\over r^2}=f_\varphi,\cr
    &v_\varphi|_{t=0}=v_\varphi(0).\cr}
  \label{6.1}
\end{equation}
Throughout Sections 4 and 6 we choose $s>4-2\delta$, which in particular ensures $s>3$ as required in Lemma~\ref{l4.4}.

\begin{lemma}\label{l6.1}
  Let $D_2$ be defined in (\ref{2.9}), $v_\varphi(0)\in L_\infty(\bar\R^3)$,\break $f_\varphi/r\in L_1(0,t;L_\infty(\bar\R^3))$. Then
  \begin{equation}\eqal{
    |v_\varphi|_{\infty,\bar\R_t^3}^2&\le c{D_2^2D_1^{1/2}\over\nu}X^{3/2}+2D_2\bigg|{f_\varphi\over r}\bigg|_{\infty,1,\bar\R_t^3}+|v_\varphi(0)|_{\infty,\bar\R^3}^2\cr
    &\equiv c{D_2^2D_1^{1/2}\over\nu}X^{3/2}+D_{16}^2.\cr}
    \label{6.2}
  \end{equation}
\end{lemma}

\begin{proof}
  Multiplying $(\ref{1.7})_2$ by $v_\varphi|v_\varphi|^{s-2}$ and integrating over $\bar\R^3$ yield
  \begin{equation}\eqal{
    &{1\over s}{d\over dt}|v_\varphi|_{s,\bar\R^3}^s+{4\nu(s-1)\over s^2}|\nabla|v_\varphi|^{s/2}|_{2,\bar\R^3}^2+\nu\intop_{\bar\R^3}{|v_\varphi|^s\over r^2}dx\cr
    &=\intop_{\bar\R^3}\psi_{1,z}|v_\varphi|^sdx+\intop_{\bar\R^3}f_\varphi v_\varphi|v_\varphi|^{s-2}dx.\cr}
    \label{6.3}
  \end{equation}
  The first term on the r.h.s. of (\ref{6.3}) is bounded by
  $$\eqal{
    &{\nu\over 2}\intop_{\bar\R^3}{|v_\varphi|^s\over r^2}dx+{D_2^2\over 2\nu}\intop_{\bar\R^3}\psi_{1,z}^2|v_\varphi|^{s-2}dx\cr
    &\le{\nu\over 2}\intop_{\bar\R^3}{|v_\varphi|^s\over r^2}dx+{D_2^2\over 2\nu}|\psi_{1,z}|_{s,\bar\R^3}^2|v_\varphi|_{s,\bar\R^3}^{s-2}.\cr}
  $$
  The second integral on the r.h.s. of (\ref{6.3}) we estimate by
  $$\eqal{
    &\intop_{\bar\R^3}|f_\varphi|\,|v_\varphi|^{s-1}dx=\intop_{\bar\R^3}\bigg|{f_\varphi\over r}\bigg|r|v_\varphi|^{s-1}dx\le D_2\intop_{\bar\R^3}\bigg|{f_\varphi\over r}\bigg|\,|v_\varphi|^{s-2}dx\cr
    &\le D_2\bigg|{f_\varphi\over r}\bigg|_{s/2,\bar\R^3}|v_\varphi|_{s,\bar\R^3}^{s-2}.\cr}
  $$
  In view of the above estimates, inequality (\ref{6.3}) takes the form
  $$\eqal{
    &{1\over s}{d\over dt}|v_\varphi|_{s,\bar\R^3}^s\le{D_2^2\over 2\nu}|\psi_{1,z}|_{s,\bar\R^3}^2 |v_\varphi|_{s,\bar\R^3}^{s-2}+D_2\bigg|{f_\varphi\over r}\bigg|_{s/2,\bar\R^3}|v_\varphi|_{s,\bar\R^3}^{s-2}.\cr}
  $$
  Simplifying, we get
  $$
    {d\over dt}|v_\varphi|_{s,\bar\R^3}^2\le{D_2^2\over\nu}|\psi_{1,z}|_{s,\bar\R^3}^2+2D_2\bigg|{f_\varphi\over r}\bigg|_{s/2,\bar\R^3}.
  $$
  Integrating the inequality with respect to time and passing with $s$ to $\infty$, we obtain
  \begin{equation}\eqal{
    |v_\varphi(t)|_{\infty,\bar\R^3}^2&\le{D_2^2\over\nu}\intop_0^t|\psi_{1,z}|_{\infty,\bar\R^3}^2dt'+ 2D_2\intop_0^t\bigg|{f_\varphi\over r}\bigg|_{\infty,\bar\R^3}dt'\cr
    &\quad+|v_\varphi(0)|_{\infty,\bar\R^3}^2.\cr}
    \label{6.4}
  \end{equation}
  Using the interpolation inequality (\ref{2.32}) in the form
  $$
    |g|_{\infty,\bar\R^3}\le c|g|_{2,\bar\R^3}^{1/4}|D^2g|_{2,\bar\R^3}^{3/4},
  $$
  we obtain
  $$\eqal{
    &|v_\varphi(t)|_{\infty,\bar\R^3}^2\le c{D_2^2\over\nu}\intop_0^t|\psi_{1,z}|_{2,\bar\R^3}^{1/2}|D^2\psi_{1,z}|_{2,\bar\R^3}^{3/2}dt'\cr
    &\quad+2D_2\intop_0^t\bigg|{f_\varphi\over r}\bigg|_{\infty,\bar\R^3}dt'+|v_\varphi(0)|_{\infty,\bar\R^3}^2\cr
    &\le c{D_2^2D_1^{1/2}\over\nu}X^{3/2}+2D_2\intop_0^t\bigg|{f_\varphi\over r}\bigg|_{\infty,\bar\R^3}dt'+|v_\varphi(0)|_{\infty,\bar\R^3}^2.\cr}
  $$
  To derive local estimates for $v_\varphi$ we introduce a smooth function $\zeta_1(r)$ such that
  \begin{equation}
    \zeta_1(r)=\begin{cases}1 & r\le 2r_0\cr
             0 & r\ge 4r_0\cr
    \end{cases}
    \label{6.5}
  \end{equation}
  Then the function
  \begin{equation}
    \tilde v'_\varphi=v_\varphi\zeta_1
    \label{6.6}
  \end{equation}
  is a solution to the problem
  \begin{equation}\eqal{
      &\tilde v'_{\varphi,t}+v\cdot\nabla\tilde v'_\varphi+{v_r\over r}\tilde v'_\varphi-\nu\Delta\tilde v'_\varphi+\nu{\tilde v'_\varphi\over r^2}\cr
      &=\tilde f'_\varphi+v_rv_\varphi\dot\zeta_1-\nu\bigg(2v_{\varphi,r}\dot\zeta_1+v_\varphi\bigg(\ddot\zeta_1+{1\over r}\dot\zeta_1\bigg)\bigg),\cr
      &\tilde v'_\varphi|_{t=0}=\tilde v'_\varphi(0),\cr}
    \label{6.7}
  \end{equation}
  where $\tilde f'_\varphi=f_\varphi\zeta_1$.
\end{proof}

\begin{lemma} \label{l6.2}
  Let $s>4-2\delta$ be given. Assume that there exists a positive constant $c_0$ such that
  \begin{equation}
    \operatorname*{ess\,inf}_{0<\tau<t} \frac{\|\tilde v'_\varphi(\tau)\|_{L_s(\bar\R^3)}}{\|\tilde v'_\varphi(\tau)\|_{L_\infty(\bar\R^3)}} \ge c_0.
    \label{6.8}
  \end{equation}
  Then
  \begin{equation}
    |\tilde v'_\varphi|_{s,\infty,\bar\R_t^3}^{4-2\delta}\le{D_{17}\over c_0^{s-4+2\delta}}X+{D_{18}\over c_0^{s-4+2\delta}}+D_{19},
    \label{6.9}
  \end{equation}
  where $D_{17}=c{r_0^{2\delta}\over\delta^2}D_2^{4-2\delta}$, $D_{18}=(4-2\delta)[|\tilde f'_\varphi|_{10/(1+6\delta),\bar\R_t^3}D_1^{3-2\delta}+D_1^2]$, $D_{19}=|\tilde v'_\varphi(0)|_{s,\bar\R^3}^{4-2\delta}$.
\end{lemma}

\begin{proof}
  Multiply (\ref{6.7}) by $\tilde v'_\varphi|\tilde v'_\varphi|^{s-2}$ and integrate over $\bar\R^3$ to obtain
  \begin{equation}\eqal{
    &{1\over s}{d\over dt}|\tilde v'_\varphi|_{s,\bar\R^3}^s+{4\nu(s-1)\over s^2}|\nabla|\tilde v'_\varphi|^{s/2}|_{2,\bar\R^3}^2+\nu\intop_{\bar\R^3}{|\tilde v'_\varphi|^s\over r^2}dx\cr
    &=\intop_{\bar\R^3}\psi_{1,z}|\tilde v'_\varphi|^sdx+\intop_{\bar\R^3}v_rv_\varphi\tilde v'_\varphi|\tilde v'_\varphi|^{s-2}\dot\zeta_1dx\cr
    &\quad-2\nu\intop_{\bar\R^3}v_{\varphi,r}\dot\zeta_1\tilde v'_\varphi|\tilde v'_\varphi|^{s-2}dx-\nu\intop_{\bar\R^3}v_\varphi\bigg(\ddot\zeta_1+{1\over r}\dot\zeta_1\bigg)\tilde v'_\varphi|\tilde v'_\varphi|^{s-2}dx\cr
    &\quad+\intop_{\bar\R^3}\tilde f'_\varphi\tilde v'_\varphi|\tilde v'_\varphi|^{s-2}dx.\cr}
    \label{6.10}
  \end{equation}
  We estimate the first term on the r.h.s. of (\ref{6.10}) by
  $$\eqal{
    &\nu\intop_{\bar\R^3}{|\tilde v'_\varphi|^s\over r^2}dx+{1\over 4\nu}\intop_{\bar\R^3\cap\supp\zeta_1}r^2|\psi_{1,z}|^2|\tilde v'_\varphi|^sdx\cr
    &\le\nu\intop_{\bar\R^3}{|\tilde v'_\varphi|^s\over r^2}dx+{D_2^{4-2\delta}\over 4\nu}\intop_{\bar\R^3\cap\supp\zeta_1}{\psi_{1,z}^2\over r^{2-2\delta}}|\tilde v'_\varphi|^{s-4+2\delta}dx,\cr}
  $$
  where $\delta>0$ is small and Lemma \ref{l2.3} is used.

  The last term on the r.h.s. of (\ref{6.10}) is bounded by
  $$
    \intop_{\bar\R^3}|\tilde f'_\varphi|\,|\tilde v'_\varphi|^{s-1}dx\le|\tilde v'_\varphi|_{\infty,\bar\R^3}^{s-4+2\delta}\intop_{\bar\R^3}|\tilde f'_\varphi|\,|\tilde v'_\varphi|^{3-2\delta}dx.
  $$
  We assume $s\ge 4-2\delta$ so that $s-4+2\delta\ge 0$ and the factor
  $|\tilde v'_\varphi|_{\infty}^{\,s-4+2\delta}$ can be extracted.

  The second term on the r.h.s. of (\ref{6.10}) is estimated by
  $$
    \intop_{\bar\R^3}|v_r|\,|v_\varphi|\,|\tilde v'_\varphi|^{s-1}|\dot\zeta_1|dx\le|\tilde v'_\varphi|_{\infty,\bar\R^3}^{s-4+2\delta}\intop_{\bar\R^3}|v_r|\,|v_\varphi|\,|\tilde v'_\varphi|^{3-2\delta}|\dot\zeta_1|dx.
  $$
  Using the properties of $\dot\zeta_1$ we can estimate the above expression by
  $$
    c|\tilde v'_\varphi|_{\infty,\bar\R^3}^{s-4+2\delta}D_2^{3-2\delta}\intop_{\bar\R^3}(|v_r|^2+|v_\varphi|^2)dx.
  $$
  The third term on the r.h.s. of (\ref{6.10}) is bounded by
  $$
    c|\tilde v'_\varphi|_{\infty,\bar\R^3}^{s-4+2\delta}D_2^{2-2\delta}\intop_{\bar\R^3}(|v_{\varphi,r}|^2+|v_\varphi|^2)dx.
  $$
  Similarly, the fourth term on the r.h.s. of (\ref{6.10}) is estimated by
  $$
    c|\tilde v'_\varphi|_{\infty,\bar\R^3}^{s-4+2\delta}D_2^{2-2\delta}\intop_{\bar\R^3}|v_\varphi|^2dx.
  $$
  Using the above estimates in (\ref{6.10}) yields
  \begin{equation}\eqal{
      &{1\over s}{d\over dt}|\tilde v'_\varphi|_{s,\bar\R^3}^s\le{D_2^{4-2\delta}\over 4\nu}|\tilde v'_\varphi|_{\infty,\bar\R^3}^{s-4+2\delta}\intop_{\bar\R^3\cap\supp\zeta_1}{|\psi_{1,z}|^2\over r^{2-2\delta}}dx\cr
      &\quad+|\tilde v'_\varphi|_{\infty,\bar\R^3}^{s-4+2\delta}\intop_{\bar\R^3}|\tilde f'_\varphi|\,|\tilde v'_\varphi|^{3-2\delta}dx\cr
      &\quad+c|\tilde v'_\varphi|_{\infty,\bar\R^3}^{s-4+2\delta}\intop_{\bar\R^3}(|v_r|^2+|v_\varphi|^2+|v_{\varphi,r}|^2)dx.\cr}
    \label{6.11}
  \end{equation}
  Now, we examine the expression from the l.h.s. of (\ref{6.11}),
  $$\eqal{
    &{1\over s}{d\over dt}|\tilde v'_\varphi|_{s,\bar\R^3}^s=|\tilde v'_\varphi|_{s,\bar\R^3}^{s-1}{d\over dt}|\tilde v'_\varphi|_{s,\bar\R^3}\cr
    &=|\tilde v'_\varphi|_{s,\bar\R^3}^{s-1-\alpha}|\tilde v'_\varphi|_{s,\bar\R^3}^\alpha{d\over dt}|\tilde v'_\varphi|_{s,\bar\R^3}\cr
    &={1\over\alpha+1}|\tilde v'_\varphi|_{s,\bar\R^3}^{s-1-\alpha}{d\over dt}|\tilde v'_\varphi|_{s,\bar\R^3}^{\alpha+1}.\cr}
  $$
  Setting $\alpha+1=4-2\delta$ we obtain from (\ref{6.11}) the inequality
  \begin{equation}\eqal{
      &{1\over 4-2\delta}\bigg({|\tilde v'_\varphi|_{s,\bar\R^3}\over|\tilde v'_\varphi|_{\infty,\bar\R^3}}\bigg)^{s-4+2\delta}{d\over dt}|\tilde v'_\varphi|_{s,\bar\R^3}^{4-2\delta}\cr
      &\le{D_2^{4-2\delta}\over 4\nu}\intop_{\bar\R^3\cap\supp\zeta_1}{\psi_{1,z}^2\over r^{2-2\delta}}dx+\intop_{\bar\R^3}|\tilde f'_\varphi|\,|\tilde v'_\varphi|^{3-2\delta}dx\cr
      &\quad+c\intop_{\bar\R^3}(|v_r|^2+|v_\varphi|^2+|v_{\varphi,r}|^2)dx.\cr}
    \label{6.12}
  \end{equation}
  In view of (\ref{6.8}) we have
  \begin{equation}\eqal{
      &{d\over dt}|\tilde v'_\varphi|_{s,\bar\R^3}^{4-2\delta}\le{4-2\delta\over 4\nu}{D_2^{4-2\delta}\over c_0^{s-4+2\delta}}\intop_{\bar\R^3\cap\supp\zeta_1}{\psi_{1,z}^2\over r^{2-2\delta}}dx\cr
      &\quad+{4-2\delta\over c_0^{s-4+2\delta}}\intop_{\bar\R^3}|\tilde f'_\varphi|\,|\tilde v'_\varphi|^{3-2\delta}dx\cr
      &\quad+{c(4-2\delta)\over c_0^{s-4+2\delta}}\intop_{\bar\R^3}(|v_r|^2+|v_\varphi|^2+|v_{\varphi,r}|^2)dx.\cr}
    \label{6.13}
  \end{equation}
  Using the Hardy inequality (\ref{2.13}), interpolation inequality (\ref{2.32}) and (\ref{2.14}), we have
  \begin{equation}\eqal{
      &\intop_{\bar\R_t^3\cap\supp\zeta_1}{\psi_{1,z}^2\over r^{2-2\delta}}dxdt'\le c{r_0^{2\delta}\over\delta^2}|\psi_{1,zr}|_{2,\bar\R_t^3}^2\cr
      &\le c{r_0^{2\delta}\over\delta^2}D_1X.\cr}
    \label{6.14}
  \end{equation}
  Moreover,
  \begin{equation}
    \intop_{\bar\R_t^3}|\tilde f'_\varphi|\,|\tilde v'_\varphi|^{3-2\delta}dxdt'\le|\tilde f'_\varphi|_{10/(1+6\delta),\bar\R_t^3}D_1^{3-2\delta}.
    \label{6.15}
  \end{equation}
  Integrating (\ref{6.13}) with respect to time and using (\ref{6.14}) and (\ref{6.15}), we obtain
  \begin{equation}\eqal{
      &|\tilde v'_\varphi(t)|_{s,\bar\R^3}^{4-2\delta}\le c{4-2\delta\over 4\nu}{D_2^{4-2\delta}\over c_0^{s-4+2\delta}}{r_0^{2\delta}\over\delta^2}D_1X\cr
      &\quad+{4-2\delta\over c_0^{s-4+2\delta}}|\tilde f'_\varphi|_{10/(1+6\delta),\bar\R_t^3}D_1^{3-2\delta}\cr
      &\quad+{c(4-2\delta)\over c_0^{s-4+2\delta}}D_1^2+|\tilde v'_\varphi(0)|_{s,\bar\R^3}^{4-2\delta}.\cr}
    \label{6.16}
  \end{equation}
  This inequality implies (\ref{6.9}) and concludes the proof.
\end{proof}

\section{Conditional higher-order estimates}\label{s7}

The purpose of this section is to propagate the conditional bound $\mathcal{K}(t)$ derived in Theorem~\ref{t1.1} through the subsequent bootstrap steps to obtain higher-order Sobolev estimates for regular solutions.

\begin{lemma}[Parabolic multiplication algebra]\label{l:parabolic-algebra}
  Let $0<t<\infty$. Then
  \[
    W^{2,1}_3(\mathbb R^3_t)
    \cdot
    W^{2,1}_3(\mathbb R^3_t)
    \hookrightarrow
    W^{2,1}_3(\mathbb R^3_t),
  \]
  and
  \[
    \|fg\|_{W^{2,1}_3(\mathbb R^3_t)}
    \le C_t
    \|f\|_{W^{2,1}_3(\mathbb R^3_t)}
    \|g\|_{W^{2,1}_3(\mathbb R^3_t)}.
  \]
  The constant $C_t$ is independent of the viscosity $\nu$.
\end{lemma}

\begin{proof}
  We apply \cite[Theorem~2.5(b), Remark~2.6(a),(b)]{KS}
  with the block dimensions
  \[
    \boldsymbol n=(3,1),
  \]
  the parabolic anisotropy
  \[
    \boldsymbol\omega=(1,2),
  \]
  the anisotropic smoothness order $\sigma=2$, and $p=3$.
  Then
  \[
    \dot\omega=\operatorname{lcm}(1,2)=2,
    \qquad
    \boldsymbol\omega\cdot\boldsymbol n=5,
  \]
  and the anisotropic Sobolev index equals
  \[
    \frac1{\dot\omega}
    \left(
    \sigma-\frac{\boldsymbol\omega\cdot\boldsymbol n}{p}
    \right)
    =
    \frac12\left(2-\frac53\right)
    =
    \frac16>0.
  \]
  Moreover, with equivalence of norms,
  \[
    H^{2,(1,2)}_3
    =
    W^{2,(1,2)}_3
    =
    W^{2,1}_3.
  \]
  The result on the finite strip
  $\mathbb R^3\times(0,t)$ follows by the standard bounded
  extension operator in the time variable.
\end{proof}

\begin{lemma}[Mixed-derivative embedding]\label{l:mixed-derivative}
  Let $m\in\mathbb N_0$ and $0<t<\infty$. Then
  \[
    W^1_3(0,t;W^m_3(\mathbb R^3))
    \cap
    L_3(0,t;W^{m+2}_3(\mathbb R^3))
    \hookrightarrow
    W^{1/2}_3(0,t;W^{m+1}_3(\mathbb R^3)).
  \]
  Moreover,
  \[
    \|u\|_{W^{1/2}_3(0,t;W^{m+1}_3)}
    \le C_t
    \left(
    \|u\|_{W^1_3(0,t;W^m_3)}
    +
    \|u\|_{L_3(0,t;W^{m+2}_3)}
    \right).
  \]
\end{lemma}

\begin{proof}
  After applying a bounded extension operator in the time variable,
  we use \cite[Theorem~3.12]{ALV} with
  \[
    \mathcal A=B,\qquad
    p_0=p_1=p=3,\qquad
    q_0=q_1=q=3,
  \]
  \[
    s_0=1,\qquad s_1=0,\qquad
    \theta=\frac12,
  \]
  and
  \[
    X_0=W^m_3(\mathbb R^3),\qquad
    X_1=W^{m+2}_3(\mathbb R^3),\qquad
    X_\theta=W^{m+1}_3(\mathbb R^3).
  \]
  The interpolation inequality
  \[
    \|u\|_{W^{m+1}_3}
    \le C
    \|u\|_{W^m_3}^{1/2}
    \|u\|_{W^{m+2}_3}^{1/2}
  \]
  verifies condition~(3.22) of \cite{ALV}. Using
  \[
    W^1_3(0,t;X_0)\hookrightarrow B^1_{3,3}(0,t;X_0),
    \qquad
    L_3(0,t;X_1)\hookrightarrow B^0_{3,3}(0,t;X_1),
  \]
  Theorem~3.12 gives
  \[
    u\in B^{1/2}_{3,3}(0,t;X_\theta)
    =
    W^{1/2}_3(0,t;X_\theta).
  \]
\end{proof}
\begin{lemma}\label{l7.1}
  Assume that $f\in W_3^{1,1/2}(\mathbb R_t^3)$, $v(0)\in B^{7/3}_{3,3}(\mathbb R^3)$, and $\operatorname{div}v(0)=0$. Then there exists a local solution to problem (\ref{1.1}) such that $v\in W_3^{3,3/2}(\mathbb R_t^3)$, $\nabla p\in W_3^{1,1/2}(\mathbb R_t^3)$ and the estimate holds
  \begin{equation}\eqal{
      &\|v\|_{W_3^{3,3/2}(\mathbb R_t^3)}+\|\nabla p\|_{W_3^{1,1/2}(\mathbb R_t^3)}\cr
      &\le C_{\nu,T_0}(\|f\|_{W_3^{1,1/2}(\mathbb R_t^3)}+\|v(0)\|_{B^{7/3}_{3,3}(\mathbb R^3)}),\cr}
    \label{7.1}
  \end{equation}
  where $0<t\le T_0$ and
  \[
    T_0=T_0\left(
    \nu,
    \|f\|_{W^{1,1/2}_3},
    \|v(0)\|_{B^{7/3}_{3,3}}
    \right)>0
  \]
  is chosen sufficiently small.
\end{lemma}

\begin{proof}
  Let $(w,\pi)$ be the solution of the linear Stokes problem
  \[
    w_{,t}-\nu\Delta w+\nabla\pi=f,
    \qquad \operatorname{div}w=0,
    \qquad w|_{t=0}=v(0).
  \]
  By the higher-order Stokes estimate proved in Lemma~\ref{l7.3} below,
  $w\in W^{3,3/2}_3(\mathbb R^3_{T_0})$. In particular,
  \[
    \|w\|_{W^{2,1}_3(\mathbb R^3_{T_0})}\longrightarrow0
    \qquad\text{as }T_0\downarrow0.
  \]

  We solve for the zero-trace remainder $z=v-w$. For a divergence-free
  $z\in W^{2,1}_3(\mathbb R^3_{T_0})$ with $z|_{t=0}=0$, define
  $\mathcal Tz=Z$, where $(Z,Q)$ solves
  \[
    Z_{,t}-\nu\Delta Z+\nabla Q
    =-(w+z)\cdot\nabla(w+z),
    \qquad \operatorname{div}Z=0,
    \qquad Z|_{t=0}=0.
  \]
  By Lemma~\ref{l:parabolic-algebra},
  $W^{2,1}_3(\mathbb R^3_{T_0})$ is a Banach algebra. Since $\operatorname{div}u=0$, we have
  \begin{equation*}\begin{aligned}
      \|(u\cdot\nabla)u\|_{L_3(\mathbb R^3_{T_0})}
       & =
      \|\operatorname{div}(u\otimes u)\|_{L_3(\mathbb R^3_{T_0})} \\
       & \le
      C\|u\otimes u\|_{W^{2,1}_3(\mathbb R^3_{T_0})}              \\
       & \le
      C_{T_0}\|u\|_{W^{2,1}_3(\mathbb R^3_{T_0})}^2.
    \end{aligned}\end{equation*}
  Hence the Stokes estimate~\eqref{2.33} gives
  \[
    \|\mathcal Tz\|_{W^{2,1}_3}
    \leq C_{\nu,T_0}
    \|w+z\|_{W^{2,1}_3}^{2}.
  \]
  Moreover,
  \[
    \|\mathcal Tz_1-\mathcal Tz_2\|_{W^{2,1}_3}
    \leq C_{\nu,T_0}
    \bigl(
    \|w+z_1\|_{W^{2,1}_3}
    +\|w+z_2\|_{W^{2,1}_3}
    \bigr)
    \|z_1-z_2\|_{W^{2,1}_3}.
  \]
  Choosing $T_0=T_0(\nu,\|f\|_{W^{1,1/2}_3},
    \|v(0)\|_{B^{7/3}_{3,3}})$ sufficiently small, and then a ball of
  sufficiently small radius in the zero-trace space, makes $\mathcal T$
  invariant and strictly contractive. Thus $z$ exists uniquely in
  $W^{2,1}_3(\mathbb R^3_{T_0})$, and $v=w+z$ is the unique local solution.

  Finally, Lemma~\ref{l:parabolic-algebra} together with the mixed-derivative
  embedding from Lemma~\ref{l:mixed-derivative} implies
  \[
    v\cdot\nabla v
    =\operatorname{div}(v\otimes v)
    \in W^{1,1/2}_3(\mathbb R^3_{T_0}).
  \]
  Applying Lemma~\ref{l7.3} to the equation for $z$ gives
  $z\in W^{3,3/2}_3$ and $\nabla q\in W^{1,1/2}_3$.
  Choosing $T_0$ so small that $\|v\|_{W^{2,1}_3(\mathbb R^3_{T_0})}\le 1$, we obtain
  \[
    \|z\|_{W^{3,3/2}_3}
    +\|\nabla q\|_{W^{1,1/2}_3}
    \le C_{\nu,T_0}
    \|v\cdot\nabla v\|_{W^{1,1/2}_3}
    \le C_{\nu,T_0}
    \|v\|_{W^{2,1}_3}^2,
  \]
  which together with the estimate of the linear part yields~\eqref{7.1}.
\end{proof}

\begin{lemma}[Higher-order Stokes estimate]\label{l7.3}
  Assume that $g\in W^{1,1/2}_3(\mathbb R^3_t)$ and $U_0\in B^{7/3}_{3,3}(\mathbb R^3)$, which agrees with $W^{7/3}_3(\mathbb R^3)$ with equivalence of norms, with $\operatorname{div}U_0=0$. Then the solution of the Stokes system
  \[
    \begin{cases}
      U_{,t}-\nu\Delta U+\nabla P=g, \\
      \operatorname{div}U=0,         \\
      U|_{t=0}=U_0
    \end{cases}
  \]
  satisfies
  \begin{equation}\label{7.high.stokes}
    \|U\|_{W^{3,3/2}_3(\mathbb R^3_t)}
    +\|\nabla P\|_{W^{1,1/2}_3(\mathbb R^3_t)}
    \leq C_{\nu,t}
    \left(
    \|g\|_{W^{1,1/2}_3(\mathbb R^3_t)}
    +\|U_0\|_{B^{7/3}_{3,3}(\mathbb R^3)}
    \right).
  \end{equation}
\end{lemma}

\begin{proof}
  Let $\mathbb P$ denote the Helmholtz--Leray projection in
  $\mathbb R^3$. Since $\operatorname{div}U_0=0$, the velocity
  satisfies
  \[
    U_{,t}-\nu\Delta U=\mathbb P g,
    \qquad U|_{t=0}=U_0.
  \]
  Moreover,
  \[
    \nabla P=(I-\mathbb P)g.
  \]
  The Helmholtz projection and $I-\mathbb P$ are bounded in the
  spaces occurring below, because they are given by spatial
  Riesz transforms.

  Applying the maximal-regularity estimate of
  Maremonti and Solonnikov \cite{MS} with $q=r=3$ gives
  \[
    \|U\|_{W^{2,1}_3(\mathbb R^3_t)}
    \leq C_{\nu,t}
    \left(
    \|g\|_{L_3(\mathbb R^3_t)}
    +\|U_0\|_{B^{4/3}_{3,3}(\mathbb R^3)}
    \right).
  \]
  For each $j=1,2,3$, the derivative $D_jU$ solves the Stokes
  problem with right-hand side $D_jg$ and initial datum
  $D_jU_0$. Since
  \[
    D_jg\in L_3(\mathbb R^3_t),
    \qquad
    D_jU_0\in B^{4/3}_{3,3}(\mathbb R^3),
  \]
  a second application of the same estimate yields
  \[
    U\in L_3(0,t;W^3_3(\mathbb R^3)),
    \qquad
    U_{,t}\in L_3(0,t;W^1_3(\mathbb R^3)).
  \]
  Since
  \[
    U\in W^1_3(0,t;W^1_3(\mathbb R^3))
    \cap
    L_3(0,t;W^3_3(\mathbb R^3)),
  \]
  Lemma~\ref{l:mixed-derivative}, applied with $m=1$, yields
  \[
    U\in W^{1/2}_3(0,t;W^2_3(\mathbb R^3)).
  \]
  Consequently,
  \[
    \Delta U\in W^{1/2}_3(0,t;L_3(\mathbb R^3)).
  \]
  Since $\mathbb P g\in W^{1/2}_3(0,t;L_3(\mathbb R^3))$, the equation implies
  \[
    U_{,t}\in W^{1/2}_3(0,t;L_3(\mathbb R^3)).
  \]
  Together with the preceding spatial estimates this gives
  \[
    U\in W^{3,3/2}_3(\mathbb R^3_t).
  \]
  Finally,
  \[
    \|\nabla P\|_{W^{1,1/2}_3(\mathbb R^3_t)}
    \leq C\|g\|_{W^{1,1/2}_3(\mathbb R^3_t)}.
  \]
  Combining these estimates proves \eqref{7.high.stokes}.
\end{proof}

\begin{lemma}\label{l7.2}
  Assume that all quantities listed in Notation 2.2 are finite for any $t<\infty$. Let $K_1(t)$ be the conditional bound from~\eqref{1.55b}. We have from Theorem~\ref{t1.1} that
  \begin{equation}
    \|\Phi\|_{V(\bar\R_t^3)}+\|\Gamma\|_{V(\bar\R_t^3)}\le c\mathcal K(t) \le K_1(t).
    \label{7.2}
  \end{equation}
  Assume that
  $$
    f\in W_3^{1,1/2}(\mathbb R_t^3),\quad v(0)\in B^{7/3}_{3,3}(\mathbb R^3),\quad \operatorname{div}v(0)=0.
  $$
  Then
  \begin{equation}\eqal{
      &\|v\|_{W_3^{3,3/2}(\mathbb R_t^3)}+\|\nabla p\|_{W_3^{1,1/2}(\mathbb R_t^3)}\cr
      &\le\phi(K_1(t),\|f\|_{W_3^{1,1/2}(\mathbb R_t^3)},\|v(0)\|_{B^{7/3}_{3,3}(\mathbb R^3)}).\cr}
    \label{7.3}
  \end{equation}
\end{lemma}

\begin{proof}
  From (\ref{7.2}) we have
  \begin{equation}
    \|\Gamma\|_{V(\bar\R_t^3)}\le K_1(t).
    \label{7.4}
  \end{equation}
  Let $\tilde\Gamma=\Gamma\zeta$ be defined in (\ref{3.29}) and $\zeta$ is introduced in the beginning of Section \ref{s4}.

  Hence,
  \begin{equation}
    \|\tilde\Gamma\|_{V(\bar\R_t^3)}\le\|\Gamma\|_{V(\bar\R_t^3)}\le K_1(t).
    \label{7.5}
  \end{equation}
  From (\ref{2.14}) and (\ref{3.31}), we obtain
  \begin{equation}
    \|\tilde\psi_1\|_{2,\infty,\bar\R_t^3}\le c K_1(t).
    \label{7.6}
  \end{equation}
  From (\ref{1.17}) the following relations hold
  \begin{equation}
    v_r=-r\psi_{1,z},\quad v_z=r\psi_{1,r}+2\psi_1.
    \label{7.7}
  \end{equation}
  Localizing (\ref{7.7}) to a neighborhood of the axis of symmetry yields
  \begin{equation}
    \tilde v_r=-r\tilde\psi_{1,z},\quad \tilde v_z=r\tilde\psi_{1,r}-r\psi_1\dot\zeta+2\tilde\psi_1.
    \label{7.8}
  \end{equation}
  Hence, (\ref{7.6}) and (\ref{2.14}) imply
  \begin{equation}
    \|\tilde v_r\|_{1,2,\infty,\bar\R_t^3}+\|\tilde v_z\|_{1,2,\infty,\bar\R_t^3}\le c K_1(t).
    \label{7.9}
  \end{equation}
  Next, we have to find an estimate for $v_r$, $v_z$ restricted to a neighborhood located in a positive distance from the axis of symmetry. From $(\ref{1.7})_4$ and $(\ref{1.13})_2$ we have
  \begin{equation}\eqal{
    &v_{r,z}-v_{z,r}=\omega_\varphi,\cr
    &v_{r,r}+v_{z,z}=-{v_r\over r}\cr}
    \label{7.10}
  \end{equation}
  Using the function $\vartheta$ introduced in the beginning of Section \ref{s4}, we have
  \begin{equation}\eqal{
    &\hat v_{r,z}-\hat v_{z,r}=\hat\omega_\varphi-v_z\dot\vartheta,\cr
    &\hat v_{r,r}+\hat v_{z,z}=-{\hat v_r\over r}+v_r\dot\vartheta.\cr}
    \label{7.11}
  \end{equation}
  From (\ref{7.11}), we obtain
  \begin{equation}\eqal{
      &\hat v_{r,rr}+\hat v_{r,zz}=(\hat\omega_\varphi-v_z\dot\vartheta)_{,z}+\bigg(-{\hat v_r\over r}+v_r\dot\vartheta\bigg)_{,r},\cr
      &\hat v_{z,rr}+\hat v_{z,zz}=-(\hat\omega_\varphi-v_z\dot\vartheta)_{,r}+\bigg(-{\hat v_r\over r}+v_r\dot\vartheta\bigg)_{,z}.\cr}
    \label{7.12}
  \end{equation}
  Multiplying (\ref{7.12}) by $\hat v_r$ and $\hat v_z$, respectively, integrating over $\bar\R^3$ and integrating by parts, we obtain
  \begin{equation}\eqal{
      &|\hat v_{r,r}|_{2,\bar\R^3}^2+|\hat v_{r,z}|_{2,\bar\R^3}^2+|\hat v_{z,r}|_{2,\bar\R^3}^2+|\hat v_{z,z}|_{2,\bar\R^3}^2\cr
      &\le c\bigg(|\hat\omega_\varphi-v_z\dot\vartheta|_{2,\bar\R^3}^2+\bigg|-{\hat v_r\over r}+v_r\dot\vartheta\bigg|_{2,\bar\R^3}^2\bigg)\cr
      &\le c|\hat\omega_\varphi|_{2,\bar\R^3}^2+c|\hat v'|_{2,\bar\R^3}^2,\cr}
    \label{7.13}
  \end{equation}
  where $v'=(v_r,v_z)$.

  To estimate $|\hat\omega_\varphi|_{2,\bar\R^3}$ we multiply $(\ref{1.9})_2$ by $\vartheta$. Then we get
  \begin{equation}\eqal{
    &\hat\omega_{\varphi,t}+v\cdot\nabla\hat\omega_\varphi-v_r\dot\vartheta\omega_\varphi-{v_r\over r}\hat\omega_\varphi-\nu\Delta\hat\omega_\varphi\cr
    &\quad+\nu\bigg(2\omega_{\varphi,r}\dot\vartheta+\omega_\varphi\ddot\vartheta+{1\over r}\omega_\varphi\dot\vartheta\bigg)+\nu{\hat\omega_\varphi\over r^2}={2\over r}v_\varphi v_{\varphi,z}\vartheta+\hat F_\varphi.\cr}
    \label{7.14}
  \end{equation}
  Multiply (\ref{7.14}) by $\hat\omega_\varphi$ and integrate over $\bar\R^3$ to obtain
  \begin{equation}\eqal{
      &{1\over 2}{d\over dt}|\hat\omega_\varphi|_{2,\bar\R^3}^2-\intop_{\bar\R^3}v_r\dot\vartheta\omega_\varphi\widehat\omega_\varphi dx-\intop_{\bar\R^3}{v_r\over r}\hat\omega_\varphi^2dx\cr
      &\quad+\nu|\nabla\hat\omega_\varphi|_{2,\R^3}^2+\nu\intop_{\bar\R^3} \bigg(2\omega_{\varphi,r}\dot\vartheta+\omega_\varphi\ddot\vartheta+{1\over r}\omega_\varphi\dot\vartheta\bigg)\hat\omega_\varphi dx\cr
      &\quad+\nu\intop_{\bar\R^3}{\hat\omega_\varphi^2\over r^2}dx=\intop_{\bar\R^3}{2\over r}v_\varphi v_{\varphi,z}\vartheta\hat\omega_\varphi dx+\intop_{\bar\R^3}\hat F_\varphi\hat\omega_\varphi dx.\cr}
    \label{7.15}
  \end{equation}
  The second term on the l.h.s. is bounded by
  $$
    \varepsilon|\hat\omega_\varphi|_{6,\bar\R^3}^2+c(1/\varepsilon) |v_r|_{2,\bar\R^3}^2|\omega_\varphi\dot\vartheta|_{3,\bar\R^3}^2,
  $$
  the third by
  $$
    \varepsilon|\hat\omega_\varphi|_{6,\bar\R^3}^2+c(1/\varepsilon) |v_r|_{2,\bar\R^3}^2|\hat\omega_\varphi|_{3,\bar\R^3}^2.
  $$
  We write the fifth term in the form
  $$
    \nu\bigg(-\intop_{\bar\R^3}\omega_\varphi^2(\vartheta\dot\vartheta r)_{,r}drdz+\intop_{\bar\R^3}\omega_\varphi\bigg(\ddot\vartheta+{1\over r}\dot\vartheta\bigg)\hat\omega_\varphi dx\bigg)
  $$
  so it is bounded by
  $$
    c|\omega_\varphi|_{2,\bar\R^3}^2.
  $$
  Finally, the first term on the r.h.s. equals
  $$
    -\intop_{\bar\R^3}{1\over r}v_\varphi^2\vartheta\hat\omega_{\varphi,z}dx
  $$
  so it is bounded by
  $$
    \varepsilon|\hat\omega_{\varphi,z}|_{2,\bar\R^3}^2+c(1/\varepsilon)\intop_{\bar\R^3}v_\varphi^4 dx.
  $$
  Using the above estimates in (\ref{7.15}), integrating the result with respect to time, using (\ref{2.1}) and (\ref{2.12}), we obtain for small $\varepsilon$
  \begin{equation}\eqal{
      &|\hat\omega_\varphi(t)|_{2,\bar\R^3}^2+\nu|\nabla\hat\omega_\varphi|_{2,\bar\R_t^3}^2+\nu\intop_{\bar\R_t^3} {\hat\omega_\varphi^2\over r^2}dxdt'\cr
      &\le cD_1^2|\omega_\varphi\dot\vartheta|_{3,2,\bar\R_t^3}^2+cD_1^2|\hat\omega_\varphi|_{3,2,\bar\R_t^3}^2+ c(D_1^2+D_1^2D_2^2)\cr
      &\quad+c|\hat F_\varphi|_{6/5,2,\bar\R_t^3}^2+|\hat\omega_\varphi(0)|_{2,\bar\R^3}^2.\cr}
    \label{7.16}
  \end{equation}
  Using the interpolations
  $$\eqal{
      &|\hat\omega_\varphi|_{3,2,\bar\R_t^3}^2\le\varepsilon|\hat\omega_\varphi|_{6,2,\bar\R_t^3}^2+c(1/\varepsilon) |\hat\omega_\varphi|_{2,\bar\R_t^3}^2,\cr
      &|\omega_\varphi\dot\vartheta|_{3,2,\bar\R_t^3}^2\le\varepsilon|\omega_\varphi\dot\vartheta|_{6,2,\bar\R_t^3}^2+ c(1/\varepsilon)|\omega_\varphi|_{2,\bar\R_t^3}^2,\cr}
  $$
  where in the second case we use local considerations (see \cite[Ch. 4, Sect. 10]{LSU}) we finally obtain
  \begin{equation}\eqal{
      \|\hat\omega_\varphi\|_{V(\bar\R_t^3)}^2&\le c(D_1^2+D_1^2D_2^2+|\hat F_\varphi|_{6/5,2,\bar\R_t^3}^2+|\hat\omega_\varphi(0)|_{2,\bar\R^3}^2)\cr
      &\le c K_1(t).\cr}
    \label{7.17}
  \end{equation}
  Then (\ref{7.13}) yields
  \begin{equation}
    \|\hat v_r\|_{1,2,\infty,\bar\R_t^3}+\|\hat v_z\|_{1,2,\infty,\bar\R_t^3}\le c K_1(t).
    \label{7.18}
  \end{equation}
  Using the properties of the partition of unity (\ref{7.9}) and (\ref{7.18}) imply
  \begin{equation}
    \|v'\|_{1,2,\infty,\bar\R_t^3}\le c K_1(t).
    \label{7.19}
  \end{equation}
  From (\ref{7.19}) have
  \begin{equation}
    |v'|_{6,\infty,\bar\R_t^3}\le c K_1(t).
    \label{7.20}
  \end{equation}
  To increase the regularity without imposing lower-integrability assumptions on $f$ or on the initial datum, we split the solution into a linear and a nonlinear part. Let $(w,\pi)$ solve
  \begin{equation}\label{7.21}
    \begin{cases}
      w_{,t}-\nu\Delta w+\nabla\pi=f, \\
      \operatorname{div}w=0,          \\
      w|_{t=0}=v(0).
    \end{cases}
  \end{equation}
  By Lemma~\ref{l7.3},
  \begin{equation}\label{7.22}
    \begin{aligned}
       & \|w\|_{W^{3,3/2}_3(\mathbb R^3_t)}
      +\|\nabla\pi\|_{W^{1,1/2}_3(\mathbb R^3_t)} \\
       & \qquad\leq
      C_{\nu,t}
      \left(
      \|f\|_{W^{1,1/2}_3(\mathbb R^3_t)}
      +\|v(0)\|_{B^{7/3}_{3,3}(\mathbb R^3)}
      \right)
      =:D_L(t).
    \end{aligned}
  \end{equation}
  Set
  \[
    z=v-w,\qquad q=p-\pi.
  \]
  Then
  \begin{equation}\label{7.23}
    \begin{cases}
      z_{,t}-\nu\Delta z+\nabla q=-v\cdot\nabla v, \\
      \operatorname{div}z=0,                       \\
      z|_{t=0}=0.
    \end{cases}
  \end{equation}
  All lower-order norms of $w$ used below are bounded by $C D_L(t)$ through the anisotropic Sobolev embeddings.

  From~\eqref{7.20}, \eqref{1.55b}, and~\eqref{1.56},
  \begin{equation}\label{7.24}
    \|v\cdot\nabla v\|_{L_2(0,t;L_{3/2}(\mathbb R^3))}
    \leq C K_1(t)D_1.
  \end{equation}
  Applying Lemma~\ref{l2.18} to~\eqref{7.23} gives
  \begin{equation}\label{7.25}
    \|z\|_{W^{2,1}_{3/2,2}(\mathbb R^3_t)}
    +\|\nabla q\|_{L_{3/2,2}(\mathbb R^3_t)}
    \leq C K_1(t)D_1=:d_2.
  \end{equation}

  We now choose the exponents with margins sufficient for the
  final product estimate. First fix
  \[
    \frac{16}{7}<q_1<\frac52.
  \]
  By Lemma~\ref{l2.13},
  \begin{equation}\label{7.26}
    \|\nabla z\|_{L_{q_1}(\mathbb R_t^3)}
    \leq C\|z\|_{W^{2,1}_{3/2,2}(\mathbb R_t^3)}.
  \end{equation}
  Together with~\eqref{7.22}, this yields
  \[
    \|\nabla v\|_{L_{q_1}(\mathbb R_t^3)}
    \leq C(d_2+D_L(t)).
  \]
  Define
  \[
    \frac1{r_1}=\frac16+\frac1{q_1}.
  \]
  Then
  \begin{equation}\label{7.27}
    \|v\cdot\nabla v\|_{L_{r_1,q_1}(\mathbb R_t^3)}
    \leq C K_1(t)(d_2+D_L(t)).
  \end{equation}
  A second application of Lemma~\ref{l2.18} gives
  \begin{equation}\label{7.28}
    \|z\|_{W^{2,1}_{r_1,q_1}(\mathbb R_t^3)}
    +\|\nabla q\|_{L_{r_1,q_1}(\mathbb R_t^3)}
    \leq C K_1(t)(d_2+D_L(t))=:d_3.
  \end{equation}

  Choose next
  \begin{equation}\label{7.exp.q2}
    \frac{80}{27}<q_2<
    \min\left\{
    \frac{10}{3},
    \left(\frac1{q_1}-\frac1{10}\right)^{-1}
    \right\}.
  \end{equation}
  The strict inequalities in~\eqref{7.exp.q2} satisfy the conditions required in Lemma~\ref{l2.13}. Hence
  \begin{equation}\label{7.29}
    \|\nabla z\|_{L_{q_2}(\mathbb R_t^3)}
    \leq C\|z\|_{W^{2,1}_{r_1,q_1}(\mathbb R_t^3)}.
  \end{equation}
  Set
  \[
    \frac1{r_2}=\frac16+\frac1{q_2}.
  \]
  Using again~\eqref{1.55b},
  \begin{equation}\label{7.30}
    \|v\cdot\nabla v\|_{L_{r_2,q_2}(\mathbb R_t^3)}
    \leq C K_1(t)(d_3+D_L(t)).
  \end{equation}
  Therefore
  \begin{equation}\label{7.31}
    \|z\|_{W^{2,1}_{r_2,q_2}(\mathbb R_t^3)}
    +\|\nabla q\|_{L_{r_2,q_2}(\mathbb R_t^3)}
    \leq C K_1(t)(d_3+D_L(t))=:d_4.
  \end{equation}

  The trace theorem and the spatial Besov embedding imply
  \begin{equation}\label{7.32}
    \|z\|_{L_\infty(0,t;B^{2-2/q_2}_{r_2,q_2}(\mathbb R^3))}
    \leq C d_4
  \end{equation}
  and
  \begin{equation}\label{7.33}
    \|z\|_{L_\infty(0,t;L_Q(\mathbb R^3))}
    \leq C d_4,
    \qquad
    Q<Q_{\max},
    \qquad
    \frac1{Q_{\max}}=\frac{5}{3q_2}-\frac12.
  \end{equation}
  The same bound for $w$ follows from~\eqref{7.22}, so~\eqref{7.33} also holds for $v$, with $d_4$ replaced by $d_4+D_L(t)$.

  Since $q_2>80/27$, one can select
  \begin{equation}\label{7.exp.p3}
    \frac{10}{3}<p_3<\left(\frac1{q_2}-\frac1{10}\right)^{-1}.
  \end{equation}
  Then Lemma~\ref{l2.13} yields
  \begin{equation}\label{7.34}
    \|\nabla z\|_{L_{p_3}(\mathbb R_t^3)}
    \leq C\|z\|_{W^{2,1}_{r_2,q_2}(\mathbb R_t^3)}.
  \end{equation}
  Moreover, since
  \[
    \frac1{Q_{\max}}
    +\left(\frac1{q_2}-\frac1{10}\right)
    =
    \frac8{3q_2}-\frac35
    <\frac3{10},
  \]
  we may choose $Q<Q_{\max}$ and $p_3$ sufficiently close to
  their respective upper endpoints so that
  \begin{equation}\label{7.exp.pstar}
    \frac1{p_*}=\frac1Q+\frac1{p_3}<\frac3{10}.
  \end{equation}
  Consequently, $p_*>10/3$. Since the time interval is finite, \eqref{7.33} also gives
  \[
    \|v\|_{L_Q(\mathbb R^3_t)}
    \le t^{1/Q}
    \|v\|_{L_\infty(0,t;L_Q(\mathbb R^3))}
    \le C_t\bigl(d_4+D_L(t)\bigr).
  \]
  Hence, by Hölder's inequality in space-time,
  \begin{equation}\label{7.35}
    \|v\cdot\nabla v\|_{L_{p_*}(\mathbb R^3_t)}
    \le
    \|v\|_{L_Q(\mathbb R^3_t)}
    \|\nabla v\|_{L_{p_3}(\mathbb R^3_t)}
    \leq C_t(d_4+D_L(t))^2.
  \end{equation}
  Applying Lemma~\ref{l2.18} to~\eqref{7.23}, we obtain
  \begin{equation}\label{7.36}
    \|z\|_{W^{2,1}_{p_*}(\mathbb R_t^3)}
    +\|\nabla q\|_{L_{p_*}(\mathbb R_t^3)}
    \leq C(d_4+D_L(t))^2=:d_5.
  \end{equation}
  Since $p_*>10/3$, the parabolic embeddings give $z\in L_\infty(\mathbb R_t^3)$ and
  \[
    \nabla z\in L_{q_*}(\mathbb R_t^3)
    \quad\text{for some }q_*>3.
  \]
  Combining this with~\eqref{7.22} and the energy bound~\eqref{1.56}, interpolation yields
  \begin{equation}\label{7.37}
    \|v\|_{L_\infty(\mathbb R^3_t)}
    +\|\nabla v\|_{L_3(\mathbb R^3_t)}
    \leq \phi(K_1(t),D_L(t),d_5,\mathrm{data}).
  \end{equation}
  Thus $v\cdot\nabla v\in L_3(\mathbb R^3_t)$, and another application of Lemma~\ref{l2.18} to~\eqref{7.23} gives
  \begin{equation}\label{7.38}
    \|z\|_{W^{2,1}_3(\mathbb R^3_t)}
    +\|\nabla q\|_{L_3(\mathbb R^3_t)}
    \leq \phi(K_1(t),D_L(t),d_5,\mathrm{data}).
  \end{equation}
  Hence $v=w+z\in W^{2,1}_3(\mathbb R^3_t)$. Set $F=v\otimes v$. By Lemma~\ref{l:parabolic-algebra},
  \[
    F\in W^{2,1}_3(\mathbb R^3_t)
  \]
  and
  \[
    \|F\|_{W^{2,1}_3(\mathbb R^3_t)}
    \le C_t
    \|v\|_{W^{2,1}_3(\mathbb R^3_t)}^2.
  \]
  In particular,
  \[
    F\in W^1_3(0,t;L_3(\mathbb R^3))
    \cap
    L_3(0,t;W^2_3(\mathbb R^3)).
  \]
  Lemma~\ref{l:mixed-derivative}, applied with $m=0$, implies
  \[
    F\in W^{1/2}_3(0,t;W^1_3(\mathbb R^3)).
  \]
  Therefore,
  \[
    \operatorname{div}F
    \in
    W^{1/2}_3(0,t;L_3(\mathbb R^3))
    \cap
    L_3(0,t;W^1_3(\mathbb R^3)).
  \]
  Thus
  \[
    v\cdot\nabla v
    =
    \operatorname{div}(v\otimes v)
    \in W^{1,1/2}_3(\mathbb R^3_t),
  \]
  and
  \begin{equation}\label{7.39}
    \|v\cdot\nabla v\|_{W^{1,1/2}_3(\mathbb R^3_t)}
    \leq C_t\|v\|_{W^{2,1}_3(\mathbb R^3_t)}^2
    \leq \phi(K_1(t),D_L(t),d_5,\mathrm{data}).
  \end{equation}
  Finally, applying Lemma~\ref{l7.3} to~\eqref{7.23} with zero initial datum and then adding the linear part~\eqref{7.22}, we obtain
  \begin{equation}\label{7.40}
    \|v\|_{W^{3,3/2}_3(\mathbb R^3_t)}
    +\|\nabla p\|_{W^{1,1/2}_3(\mathbb R^3_t)}
    \leq
    \phi\left(
    K_1(t),
    \|f\|_{W^{1,1/2}_3(\mathbb R^3_t)},
    \|v(0)\|_{B^{7/3}_{3,3}(\mathbb R^3)}
    \right).
  \end{equation}
  This is~\eqref{7.3} and concludes the proof.
\end{proof}

\section{Remarks on global small-data solutions to problem~\eqref{1.1}}
Fix $p>5/2$. To construct a small-data solution, we use the successive approximations
\begin{equation}\label{8.1}
  \begin{array}{ll}
    \partial_t v_{n+1}-\nu\Delta v_{n+1}+\nabla p_{n+1}
    =-v_n\cdot\nabla v_n+f
     & \text{in }\mathbb R_t^3, \\
    \operatorname{div}v_{n+1}=0
     & \text{in }\mathbb R_t^3, \\
    v_{n+1}|_{t=0}=v(0)
     & \text{in }\mathbb R^3,
  \end{array}
\end{equation}
with
\begin{equation}\label{8.2}
  v_0=0.
\end{equation}
Let $C_S$ denote the Stokes maximal-regularity constant, chosen independently of the finite time horizon $t$, and let $c_*$ denote the embedding constant in
\begin{equation}\label{8.4}
  \begin{aligned}
    \|v_n\|_{L_{p\lambda_1}(\mathbb R_t^3)}
     & \leq c_*\|v_n\|_{W_p^{2,1}(\mathbb R_t^3)}, \\
    \|\nabla v_n\|_{L_{p\lambda_2}(\mathbb R_t^3)}
     & \leq c_*\|v_n\|_{W_p^{2,1}(\mathbb R_t^3)},
  \end{aligned}
  \qquad
  \frac1{\lambda_1}+\frac1{\lambda_2}=1.
\end{equation}
Then
\begin{equation}\label{8.3}
  \begin{aligned}
     & \|v_{n+1}\|_{W_p^{2,1}(\mathbb R_t^3)}
    +\|\nabla p_{n+1}\|_{L_p(\mathbb R_t^3)}  \\
     & \qquad\leq
    C_S\left(
    c_*^2\|v_n\|_{W_p^{2,1}(\mathbb R_t^3)}^2
    +\|f\|_{L_p(\mathbb R_t^3)}
    +\|v(0)\|_{W_p^{2-2/p}(\mathbb R^3)}
    \right).
  \end{aligned}
\end{equation}
For a global statement, set
\begin{equation}\label{8.5}
  G_\infty
  =
  \|f\|_{L_p(\mathbb R^3\times(0,\infty))}
  +\|v(0)\|_{W_p^{2-2/p}(\mathbb R^3)}.
\end{equation}
Since the corresponding finite-time quantity is bounded by $G_\infty$, all estimates below are uniform in $t$.

Assume inductively that
\begin{equation}\label{8.6}
  \|v_n\|_{W_p^{2,1}(\mathbb R_t^3)}\leq M.
\end{equation}
From~\eqref{8.3}, this is preserved provided
\begin{equation}\label{8.7}
  C_S\left(c_*^2M^2+G_\infty\right)\leq M.
\end{equation}
Choosing $c_*M=G_\infty^{1/2}$, we obtain the smallness condition
\begin{equation}\label{8.8}
  2C_Sc_*G_\infty^{1/2}\leq1.
\end{equation}
Because $v_0=0$, this yields a uniform bound for all iterates.

Let
\begin{equation}\label{8.9}
  V_n=v_n-v_{n-1},
  \qquad
  P_n=p_n-p_{n-1}.
\end{equation}
Then
\begin{equation}\label{8.10}
  \begin{aligned}
    \partial_tV_{n+1}-\nu\Delta V_{n+1}+\nabla P_{n+1}
                              & =-V_n\cdot\nabla v_n-v_{n-1}\cdot\nabla V_n, \\
    \operatorname{div}V_{n+1} & =0,                                          \\
    V_{n+1}|_{t=0}            & =0.
  \end{aligned}
\end{equation}
Using~\eqref{8.4} and the Stokes estimate,
\begin{equation}\label{8.11}
  \begin{aligned}
    \|V_{n+1}\|_{W_p^{2,1}(\mathbb R_t^3)}
     & \leq
    C_Sc_*^2
    \left(
    \|v_n\|_{W_p^{2,1}(\mathbb R_t^3)}
    +\|v_{n-1}\|_{W_p^{2,1}(\mathbb R_t^3)}
    \right)
    \|V_n\|_{W_p^{2,1}(\mathbb R_t^3)} \\
     & \leq
    2C_Sc_*^2M
    \|V_n\|_{W_p^{2,1}(\mathbb R_t^3)} \\
     & =
    2C_Sc_*G_\infty^{1/2}
    \|V_n\|_{W_p^{2,1}(\mathbb R_t^3)}.
  \end{aligned}
\end{equation}
Hence the sequence is contractive if
\begin{equation}\label{8.12}
  2C_Sc_*G_\infty^{1/2}<1.
\end{equation}
Under~\eqref{8.8} and~\eqref{8.12}, the limit satisfies, for every finite $t$,
\begin{equation}\label{8.13}
  v\in W_p^{2,1}(\mathbb R_t^3),
  \qquad
  \nabla p\in L_p(\mathbb R_t^3),
\end{equation}
with a bound uniform in $t$. Therefore the solution extends globally in time.

Since $p>5/2$, the parabolic Sobolev embedding gives
$W_p^{2,1}(\mathbb R_t^3)\hookrightarrow L_\infty(\mathbb R_t^3)$. In particular,
\begin{equation}\label{8.14}
  \|v_\varphi\|_{L_\infty(\mathbb R_t^3)}
  \leq
  \|v_\varphi\|_{W_p^{2,1}(\mathbb R_t^3)}
  \leq
  C\frac{G_\infty^{1/2}}{c_*}.
\end{equation}
The assumptions~\eqref{8.8} and~\eqref{8.12} constitute a global small-data condition for the full velocity problem. The small $L^\infty$ bound for the swirl component in~\eqref{8.14} is a consequence of this full-data condition, not an independent small-swirl hypothesis. This result is separate from the conditional critical-wedge estimate of Theorem~\ref{t1.1}.

\end{document}